\documentclass[10pt,oneside]{article}
\usepackage{times,amsmath,amsbsy,amssymb,amscd,mathrsfs,graphicx}
\usepackage{amssymb}
\usepackage{amsmath,bm}
\usepackage{amsmath,amsthm}

\usepackage{xcolor,color}
\usepackage[font=small,labelfont=md,textfont=it]{caption}
\usepackage{floatrow,float}
\usepackage[titletoc, title]{appendix}
\usepackage[colorlinks,linkcolor=blue,citecolor=blue]{hyperref}
\usepackage{longtable}
\usepackage{pgfplots}
\usepackage{diagbox}
\usepackage{booktabs,makecell,multirow}
\usepackage[capitalise, nosort]{cleveref}
\usepackage{algorithm}
\usepackage{algorithmic}
\usepackage{amsmath,bm}
\usepackage{cases}
\usepackage{geometry}
\usepackage{extarrows}
\usepackage{tcolorbox}
\usepackage{bbm}
\usepackage{syntonly}
\usepackage[figuresright]{rotating}
\usepackage{epstopdf}

\crefname{equation}{}{}
\crefname{lem}{Lemma}{Lemmas}
\crefname{thm}{Theorem}{Theorems}

\newcommand{\dd}{\,{\rm d}}
\newcommand{\R}{\,{\mathbb R}}

\newcommand{\dual}[1]{\left\langle {#1} \right\rangle}
\newcommand{\prox}[0]{ {\bf prox}}

\newcommand{\argmin}[0]{ {\mathop{{\rm  argmin}}\,}}

\newcommand{\st}[0]{ {{\rm  s.t.}\,}}

\newcommand{\nm}[1]{\left\lVert {#1} \right\rVert}
\newcommand{\snm}[1]{\left\lvert {#1} \right\rvert}
\newcommand{\barlk}[0]{\bar{\lambda}_{k+1}}

\newcommand{\ssnm}[1]
{
	\left\vert\kern-0.25ex
	\left\vert\kern-0.25ex
	\left\vert
	{#1}
	\right\vert\kern-0.25ex
	\right\vert\kern-0.25ex
	\right\vert
}

\makeatletter
\def\spher@harm#1{%
	\vbox{\hbox{%
			\offinterlineskip
			\valign{&\hb@xt@2\p@{\hss$##$\hss}\vskip.2ex\cr#1\crcr}%
		}\vskip-.36ex}%
}
\def\gshone{\spher@harm{.}}
\def\gshtwo{\spher@harm{.&.}}
\def\gshthree{\spher@harm{.&.&.}}
\let\gsh\spher@harm
\makeatother

\newtheorem{coro}{Corollary}[section]

\newtheorem{assum}{Assumption}

\newtheorem{lem}{Lemma}[section]

\newtheorem{rem}{Remark}[section]

\newtheorem{thm}{Theorem}[section]

\usepackage{subcaption}
\newcolumntype{I}{!{\vrule width 1,5pt}}
\newlength\savedwidth

\newlength\savewidth

\newcounter{mnote}
\let\oldmarginpar\marginpar
\renewcommand\marginpar[1]
{\-\oldmarginpar[\raggedleft\footnotesize #1]{\raggedright\footnotesize #1}}

\makeatletter\def\@captype{table}\makeatother

\newfloatcommand{capbtabbox}{table}[][\FBwidth]
\usepackage[T1]{fontenc} 
\usepackage[utf8]{inputenc} 
\usepackage{authblk}
\begin{document}
	\title{
		\Large \bf A unified differential equation solver approach for separable convex optimization: splitting, acceleration and nonergodic rate 	\thanks{This work was supported by the Foundation of Chongqing Normal University (No. 202210000161).}
	}
	
	\author[,a,b]{Hao Luo\thanks{Email: luohao@cqnu.edu.cn}}
		\author[,c]{Zihang Zhang\thanks{Email: zhang-zihang@pku.edu.cn}}
	\affil[a]{National Center for Applied Mathematics in Chongqing, Chongqing Normal University, Chongqing, 401331, People’s Republic of China}
	\affil[b]{Chongqing Research Institute of Big Data, Peking University,  Chongqing, 401121, People’s Republic of China}
		\affil[c]{School of Mathematical Sciences, Peking University, Beijing, 100871, People’s Republic of China}

%
	
	\date{}
		\maketitle

\begin{abstract}
	This paper provides a self-contained ordinary differential equation solver approach for separable convex optimization problems. A novel primal-dual dynamical system with built-in time rescaling factors is introduced, and the exponential decay of a tailored Lyapunov function is established. Then several time discretizations of the continuous model are considered and analyzed via a unified discrete Lyapunov function. Moreover, two families of accelerated proximal alternating direction methods of multipliers are obtained, and nonergodic optimal mixed-type convergence rates shall be proved for the primal objective residual, the feasibility violation and the Lagrangian gap.  Finally, numerical experiments are provided to validate the practical performances.
\end{abstract}
	\medskip\noindent{\bf Keywords:} 
	Separable convex optimization, linear constraint, dynamical system, exponential decay, primal-dual method, acceleration, splitting, linearization, nonergodic rate



\section{Introduction}
\label{sec:intro}
Consider the separable convex optimization problem:
\begin{equation}\label{eq:min-f-g-A-B}
	\min_{x\in\mathcal X,y\in\mathcal Y} F(x,y):=f(x)+g(y)\quad \st~ Ax+By = b,
\end{equation}
where $\mathcal X\subset\R^m$ and $\mathcal Y\subset \R^n$ are two closed convex sets, $A\in\R^{r\times m}$ and $B\in \R^{r\times n}$ are linear operators, $b\in\R^{r}$ is a given vector, and $f:\R^m\to \R\cup\{+\infty\}$ and $g:\R^n\to \R\cup\{+\infty\}$ are two properly closed convex functions. We are mainly interested in first-order primal-dual methods for \cref{eq:min-f-g-A-B} based on the Lagrange function
\begin{equation}\label{eq:L}
	\mathcal L(x,y,\lambda): = F(x,y)+\delta_{\mathcal X\times \mathcal Y}(x,y)+\dual{\lambda,Ax+By-b},
\end{equation}
where $(x,y,\lambda)\in \R^m\times\R^n\times \R^r$ and $\delta_{\mathcal X\times \mathcal Y}$ denotes the indicator function of $\mathcal X\times \mathcal Y$. Throughout, assume $(x^*,y^*,\lambda^*)\in\mathcal X\times \mathcal Y\times \R^r$ is a saddle-point of $	\mathcal L$, which means
\[
\mathcal L(x^*,y^*,\lambda)\leq 	\mathcal L(x^*,y^*,\lambda^*)\leq 	\mathcal L(x,y,\lambda^*)\quad \forall\,(x,y,\lambda)\in \R^m\times\R^n\times \R^r.
\]
Then $(x^*,y^*)$ is a solution to \cref{eq:min-f-g-A-B} and for simplicity we set $F^*=F(x^*,y^*)$.

In this work, we propose new accelerated primal-dual splitting methods for the separable optimization problem \cref{eq:min-f-g-A-B} via a unified differential equation solver approach. To be more specific, we shall first introduce a novel continuous dynamical system
\begin{equation}		\label{eq:2bapd-sys-intro}
	\left\{
	\begin{aligned}
		x' = {}&v-x,\\
		\gamma v' \in {}&\mu_f(x-v)-\partial_x \mathcal L(x,y,\lambda),\\
		\theta\lambda ' = {}&\nabla_\lambda \mathcal L(v,w,\lambda),\\
		\beta w' \in {}&\mu_g(y-w)-\partial_y \mathcal L(x,y,\lambda),\\
		y' = {}&w-y,
	\end{aligned}
	\right.
\end{equation}
where $\partial_\times\mathcal L$ means the subdifferential (cf.\cref{eq:partial-L}) with respect to $\times = x$ or $y$, and $\mu_f,\,\mu_g\geq  0$ correspond to the strong convexity parameters of $f$ and $g$. Moreover, $\gamma,\,\beta$ and $\theta$ are three time scaling factors and governed by $	\gamma' = \mu_f-\gamma,\, \beta' = \mu_g-\beta$ and $	\theta' = -\theta$, respectively.
We equip \cref{eq:2bapd-sys-intro} with a tailored Lyapunov function
\begin{equation}
	\label{eq:Et-intro}
	\begin{split}
		\mathcal E(t)={}& 
		\mathcal L (x,y,\lambda^*)-\mathcal L (x^*,y^*,\lambda)+ \frac{\gamma}{2}\nm{v-x^*}^2 
		+\frac{\beta}{2}\nm{w-y^*}^2+
		\frac{\theta}{2}\nm{\lambda-\lambda^*}^2,
	\end{split}
\end{equation}
and establish the exponential decay $\mathcal E(t)\leq \mathcal E(0) e^{-t}$ under the assumption that both $f$ and $g$ have Lipschitzian gradients. Note that  our previous accelerated primal-dual flow model in \cite[Section 2]{luo_acc_primal-dual_2021} can also be applied to the separable case \cref{eq:min-f-g-A-B} but it treats $(x,y)$ as an entire variable and only involves a single time scaling parameter for $F$. However, the current one \cref{eq:2bapd-sys-intro} adopts different scaling factors $\gamma$ and $\beta$ respectively for $f$ and $g$. This not only allows us to handle the partially strongly convex case $\mu_f+\mu_g>0$ but also paves the way for designing new primal-dual splitting algorithms.

Indeed, based on proper numerical discretizations of the continuous model \cref{eq:2bapd-sys-intro}, we propose several families of accelerated primal-dual splitting methods for the original optimization problem \cref{eq:min-f-g-A-B}. Using a discrete analogue of \cref{eq:Et-intro}, we establish the corresponding {\it nonergodic, mixed-type} and {\it optimal} convergence rates for the quantity
\begin{equation}\label{eq:L-F-AB}
	\mathcal L (x_k,y_k,\lambda^*)-\mathcal L (x^*,y^*,\lambda_k)+\snm{F(x_k,y_k)-F^*}+\nm{Ax_k+By_k-b}.
\end{equation}
Here, we note that (i) ``nonergodic" means the estimate is proved for the last iterate $(x_k,y_k,\lambda_k)$ instead of its historic average (cf.\cref{eq:ergodic}); (ii) ``mixed-type" says the decay rate provides explicit dependence on $\nm{A},\,\nm{B},\mu_f,\,\mu_g$ and the Lipschitz constant of $\nabla f$ (and/or $\nabla g$) (cf.\cref{eq:mixed}); (iii) by ``optimal'' we mean the iteration complexities achieve the lower bounds of first-order primal-dual methods for problem \cref{eq:min-f-g-A-B}; see \cite{Li2019,ouyang_lower_2021,woodworth_tight_2016,zhang_lower_2022}.

\subsection{Outline}
The rest of this paper is organized as follows. In the introduction part, we shall complete the literature review of existing methods for \cref{eq:min-f-g-A-B}. Then in \cref{sec:2bapd-flow}, we introduce our continuous model and establish the exponential decay of the Lyapunov function \cref{eq:L} under the smooth setting. After that, we propose two classes of methods in \cref{sec:2bapd-im-x-im-y,sec:2bapd-ex-x-im-y}, respectively, and prove nonergodic mixed-type convergence rates via a unified discrete Lyapunov function. Finally, we provide several numerical experiments in \cref{sec:num} and 
give some concluding remarks and discussions  in \cref{sec:conclu}.
\subsection{Brief review of the one block case}
Let us start with one-block setting:
\begin{equation}\label{eq:min-f-A}
	\min_{x\in\mathcal X} f(x)\quad \st~Ax = b.
\end{equation}
The 
augmented Lagrangian method (ALM) reads as \cite{rockafellar_augmented_1976} 
\begin{equation}\label{eq:ALM}
	x_{k+1}= {}\mathop{\argmin}\limits_{x\in\mathcal X}\left\{\mathcal L(x,\lambda_k)+\frac{\sigma}{2}\nm{Ax-b}^2\right\},\quad
	\lambda_{k+1} = {}\lambda_{k}+\sigma(Ax_{k+1}-b),
\end{equation}
where $\sigma>0$ denotes the penalty parameter and $\mathcal L(x,\lambda) $ is defined by \cref{eq:L} without $y,g$ and $B$.
Combining Nesterov's extrapolation technique \cite{Nesterov1983,tseng_on_accelerated_Seattle_2008}, 
ALM can be further accelerated, and faster rate $O(1/k^2)$ for the dual objective residual has been proved in \cite{He_Yuan2010,huang_accelerated_2013,kang_inexact_2015,kang_accelerated_2013,tao_accelerated_2016}. 
The accelerated linearized ALM in \cite{Xu2017} and some quadratic penalty methods \cite{li_convergence_2017,tran-dinh_non-stationary_2020} can achieve the nonergodic rates $O(1/k)$ and $O(1/k^2)$ respectively for $\mu_f=0$ and $\mu_f>0$, in terms of the primal objective residual $\snm{f(x_k)-f(x^*)}$ and the feasibility violation $\nm{Ax_k-b}$. Moreover, primal-dual methods in \cite{chen_optimal_2014,luo_universal_2022,nesterov_smooth_2005,xu_iteration_2021} possess optimal mixed-type convergence rates.
\subsection{State-of-the-art methods for two-block case}
When applied to \cref{eq:min-f-g-A-B}, the classical ALM \cref{eq:ALM} has to minimize the augmented Lagrangian
\[
\mathcal L_\sigma(x,y,\lambda): = 	\mathcal L(x,y,\lambda)+\frac{\sigma}{2}\nm{Ax+By-b}^2,
\]
which is not separable for any $\sigma>0$. Hence, the original ALM \cref{eq:ALM}
is further relaxed as the alternating direction method of multipliers (ADMM) \cite{gabay_dual_1976}
\begin{subnumcases}{\label{eq:ADMM}}
	\label{eq:ADMM-x}
	x_{k+1} = {}\mathop{\argmin}\limits_{x\in\mathcal X}\mathcal L_\sigma(x,y_{k},\lambda_{k}),\\
	\label{eq:ADMM-y}
	y_{k+1} = {}\mathop{\argmin}\limits_{y\in\mathcal Y}\mathcal L_\sigma(x_{k+1},y,\lambda_{k}),\\	
	\label{eq:ADMM-l}
	\lambda_{k+1} = {}\lambda_{k}+\sigma(Ax_{k+1}+By_{k+1}-b),
\end{subnumcases}
which minimizes $\mathcal L_\sigma(\cdot,\cdot,\lambda)$ with respect to $x$ and $y$ successively (like the Gauss-Seidel iteration). 

So far, there are vast variants of ADMM, with proximal preconditioning \cite{fazel_hankel_2013,He_Yuan2012}, symmetrization \cite{he_strictly_2014,li_proximal_2015}, over-relaxation  \cite{eckstein_augmented_2012,eckstein_douglas-rachford_1992} and  parallelization \cite{chen_proximal-based_1994,han_augmented_2014}. As showed in \cite{eckstein_augmented_2012,gabay_chapter_1983}, the Douglas--Rachford splitting \cite{douglas_numerical_1956} and the Peaceman--Rachford splitting \cite{peaceman_numerical_1955} lead to equivalent forms of ADMM for solving the dual problem of \cref{eq:min-f-g-A-B}. The primal-dual hybrid gradient framework \cite{chambolle_first-order_2011,chambolle_ergodic_2016,esser_general_2010,he_convergence_2014,zhu_unified_2022} for bilinear saddle-point problems provides linearized versions of ADMM. Besides, accelerated ADMM with extrapolation can be found in \cite{goldfarb_fast_2013,goldstein_fast_2014,ouyang_accelerated_2015}.


The convergence rates $O(1/k)$ and $O(1/k^2)$ of ADMM and its variants  have been proved in \cite{davis_convergence_2016,goldstein_fast_2014,He_Yuan2012,monteiro_iteration-complexity_2013,shefi_rate_2014,tian_alternating_2018,Xu2017}. In \cite{ouyang_accelerated_2015}, Ouyang et al. proposed an accelerated linearized ADMM and established the mixed-type convergence rate 
\begin{equation}\label{eq:mixed}
	O\left(\frac{L_f}{k^2} + \frac{\nm{A}}{k}\right),
\end{equation}
where $L_f$ denotes the Lipschitz constant of $\nabla f$. Although this yields the final $  O(1/k)$ rate, it does make sense because the dependence on $L_f$ and $\nm{A}$ is optimal \cite{ouyang_lower_2021}. However, we mention that most existing works provide only ergodic convergence rates.
In other words, the error is not measured at the last iterate $X_k=(x_k,y_k,\lambda_k)$ but its average $\widetilde X_k=(\widetilde x_k,\widetilde y_k,\widetilde \lambda_k)$ (cf. \cite[Definition 1]{Li2019}): 
\begin{equation}\label{eq:ergodic}
	\widetilde{X}_k = \sum_{i=1}^{k}a_iX_i,\quad \text{with } \sum_{i=1}^{k}a_i = 1,\quad a_i>0.
\end{equation}
As mentioned in \cite{Li2019,tran-dinh_augmented_2018}, this might violate some key properties such as sparsity and low-rankness.  To achieve nonergodic rates for the primal objective residual $\snm{F(x_k,y_k)-F^*}$ and the feasibility violation $\nm{Ax_k+By_k-b}$, Li and Lin \cite{Li2019} and Tran-Dinh et al. \cite{tran-dinh_proximal_2019,tran-dinh_augmented_2018,tran-dinh_non-stationary_2020,tran-dinh_smooth_2018} proposed new accelerated ADMM. It should be noticed that, in each iteration, the methods of Tran-Dinh et al. require {\it one more} proximal calculation than ADMM, and the final rates become ergodic if the extra proximal step is replaced by averaging.

Recently, we were aware of the works of Sabach and Teboulle \cite{sabach_faster_2022} and Zhang et al. \cite{zhang_faster_2022}. Both two proposed accelerated ADMM, and their ingredients are the so-called  primal algorithmic map and the prediction-correction framework \cite{He_Yuan2012}, which are different from our differential equation solver approach. They also established nonergodic rates $O(1/k)$ and $O(1/k^2)$ respectively for convex and (partially) strongly convex objectives, but have not derived delicate mixed-type estimates. 

\subsection{Dynamical system approach}
As we can see, continuous dynamical system approaches \cite{attouch_fast_2018,attouch_rate_2019,chen_revisiting_2022,chen_first_2019,chen_luo_unified_2021,jordan_dynamical_2019,lin_control-theoretic_2021,luo_accelerated_2021,luo_differential_2021,shi_understanding_2021,su_dierential_2016,wang_search_2021,wibisono_variational_2016,wilson_lyapunov_2021} for the unconstrained convex optimization have been extended to linearly constrained problems. Zeng et al. \cite{zeng_dynamical_2019} generalized the continuous-time model of Nesterov accelerated gradient method \cite{Nesterov1983} derived by Su et al. \cite{su_dierential_2016} to the one block case \cref{eq:min-f-A}, and established the decay rate $O(1/t^2)$ via a new Lyapunov function. Further extensions with Bregman divergence and perturbation are given in \cite{he_perturbed_2021,zhao_accelerated_2022}. Based on time discretizations of proper continuous models, Luo \cite{luo_acc_primal-dual_2021,luo_primal-dual_2022,luo_universal_2022}, He et al. \cite{HE2022110547,he_inertial_2022} and Boţ et al. \cite{bot_fast_2022} proposed the corresponding accelerated ALM with nonergodic rate $O(1/k^2)$, and  Chen and Wei \cite{chen_transformed_2023} obtained linear convergence without strong convexity assumption.


Continuous dynamical systems for the separable problem \cref{eq:min-f-g-A-B} can be found in \cite{attouch_fast_2021,bitterlich_dynamical_2021,bot_primal-dual_2020,franca_nonsmooth_2021,he_convergence_rate_2021}, and for general saddle-point systems, we refer to  \cite{cherukuri_saddle-point_2017,lu_osr-resolution_2021,schropp_dynamical_2000}. However, it is rare to see new primal-dual splitting algorithms with provable nonergodic convergence rates based on dynamical models.

\section{Continuous Dynamical Systems}
\label{sec:2bapd-flow}
\subsection{Preliminaries}
Let $\dual{\cdot,\cdot}$ and $\nm{\cdot}$ be the usual inner product and the Euclidean norm, respectively.
For any properly closed convex function $f$ on $\mathcal X$, we say $f\in\mathcal S_{\mu}^0(\mathcal X)$ with  $\mu\geq  0$ if
\begin{equation}
	\label{eq:def-mu}
	f(x_1)-
	f(x_2)
	-\dual{p,x_1-x_2}
	\geq  {}
	\frac{\mu}{2}\nm{x_1-x_2}^2
	\quad\forall\,(x_1,x_2)\in\mathcal X\times \mathcal X,
\end{equation}
where $p\in\partial f(x_2)$ with $\partial f(x_2)$ denoting the subdifferential of $f$ at $x_2\in\mathcal X$. We write $f\in\mathcal S_{\mu,L}^{1,1}(\mathcal X)$ if $f\in\mathcal S_{\mu}^0(\mathcal X)$ has $L$-Lipschitz continuous gradient:
\begin{equation}
	\label{eq:def-L}
	f(x_1)-
	f(x_2)
	-\dual{\nabla f(x_2),x_1-x_2}
	\leq  {}
	\frac{L}{2}\nm{x_1-x_2}^2
	\quad\forall\,(x_1,x_2)\in\mathcal X\times \mathcal X.
\end{equation}
The function classes $\mathcal S_{\mu}^0(\mathcal Y)$ and $\mathcal S_{\mu,L}^{1,1}(\mathcal Y)$ are defined analogously. When $\mathcal X$($\mathcal Y$) becomes the entire space $\R^m$($\R^n$), it shall be omitted for simplicity.

Clearly, if $f\in\mathcal S_{\mu_f}^0(\mathcal X)$ and $g\in\mathcal S_{\mu_g}^0(\mathcal Y)$ with $\mu_f,\,\mu_g\geq  0$, then for all $(x_1,y_1,\lambda)$ and $(x_2,y_2,\lambda)\in\mathcal X\times \mathcal Y\times \R^r $, we have 
\begin{equation}\label{eq:ineq-mu}
	\frac{\mu_f}{2}\nm{x_1-x_2}^2
	+\frac{\mu_g}{2}\nm{y_1-y_2}^2		
	\leq   {}\mathcal L(x_1,y_1,\lambda)-\mathcal L (x_2,y_2,\lambda)-\dual{p,x_1-x_2}	-\dual{q,y_1-y_2},
\end{equation}
where $p\in \partial_x\mathcal L (x_2,y_2,\lambda)$ 
and $q\in\partial_y\mathcal L (x_2,y_2,\lambda)$. 
Above and in what follows, we set
\begin{equation}\label{eq:partial-L}
	\partial_x\mathcal L (x,y,\lambda) :={} \partial f(x)+A^{\top}\lambda+N_{\mathcal X}(x),\quad 
	\partial_y\mathcal L (x,y,\lambda) :={} \partial g(y)+B^{\top}\lambda+N_{\mathcal Y}(y),
\end{equation}
with $N_{\mathcal X}(x)$ and $N_{\mathcal Y}(y)$ being the norm cone of $\mathcal X$ and $\mathcal Y$ at $x$ and $y$, respectively. 

We also introduce the notation $M\lesssim N$, which means $M\leq CN$ with some generic bounded constant $C>0$ that is independent of $A,B,\mu_f,\mu_g,\gamma_0$ and $\beta_0$ (the initial conditions to \cref{eq:2bapd-scaling}) but can be different in each occurrence. 
\subsection{Continuous-time model and exponential decay}
Motivated by the accelerated primal-dual flow \cite{luo_acc_primal-dual_2021}, we consider the following differential inclusion
\begin{equation}\label{eq:2bapd-inclu}
	\left\{
	\begin{aligned}
		0&{}\in\gamma x'' + (\gamma+\mu_f)x' +\partial_x \mathcal L(x,y,\lambda),\\
		0&{}=\theta\lambda ' -\nabla_\lambda \mathcal L(x+x',y+y',\lambda),\\
		0&{}\in	\beta y'' + (\beta+\mu_g)y' +\partial_y \mathcal L(x,y,\lambda),
	\end{aligned}
	\right.
\end{equation}
where the parameters $(\theta,\gamma,\beta)$ are governed by 
\begin{equation}	\label{eq:2bapd-scaling}
	\theta'=-\theta,\quad
	\gamma' = \mu_f-\gamma,\quad
	\beta'=\mu_g-\beta,
\end{equation}
with positive initial conditions: $\theta(0) = 1,\,\gamma(0)=\gamma_0>0$ and $\beta(0)=\beta_0>0$. Here, the new model \cref{eq:2bapd-inclu} utilizes the separable structure of \cref{eq:min-f-g-A-B} and adopts different rescaling factors for $x$ and $y$, respectively. 

It is not hard to obtain the exact solution of \cref{eq:2bapd-scaling}:
\begin{equation*}
	\theta(t) = {}e^{-t},\quad
	\gamma(t) = {}\gamma_0e^{-t}+\mu_f(1-e^{-t}),\quad
	\beta(t) = {}\beta_0e^{-t}+\mu_g(1-e^{-t}).	
\end{equation*}
Besides, as introduced in \cref{eq:2bapd-sys-intro}, an alternative presentation of \cref{eq:2bapd-inclu} reads as 
\begin{subnumcases}{	\label{eq:2bapd}}
	\label{eq:2bapd-x}
	x' = {}v-x,\\
	\label{eq:2bapd-v}
	\gamma v' \in {}\mu_f(x-v)-\partial_x \mathcal L(x,y,\lambda),\\
	\label{eq:2bapd-y}
	\label{eq:2bapd-l}
	\theta\lambda ' = {}\nabla_\lambda \mathcal L(v,w,\lambda),\\
	\label{eq:2bapd-w}
	\beta w' \in {}\mu_g(x-v)-\partial_y \mathcal L(x,y,\lambda),\\
	y' = {}w-y.
\end{subnumcases}
This seems a little bit complicated but for algorithm designing and convergence analysis, it is more convenient for us to start form \cref{eq:2bapd} and treat $(\theta,\gamma,\beta)$ as unknowns that solve \cref{eq:2bapd-scaling}.

Let $\Theta=(\theta,\gamma,\beta)$ and $X=(x,y,v,w,\lambda)$ and define a Lyapunov function
\begin{equation}\label{eq:Et}
	\begin{aligned}
		\mathcal E(\Theta,X): = \mathcal L (x,y,\lambda^*)-\mathcal L (x^*,y^*,\lambda)
		+\frac{\theta}{2}\nm{\lambda-\lambda^*}^2
		+\frac{\beta}{2}\nm{w-y^*}^2
		+\frac{\gamma}{2}\nm{v-x^*}^2,
	\end{aligned}
\end{equation}
where $(x^*,y^*,\lambda^*)$ is a saddle-point of \cref{eq:L}.
We aim to establish the decay rate of $ 	\mathcal E$, by taking derivative with respect to $t$. However, we have to mention that (i) solution existence of \cref{eq:2bapd} (or \cref{eq:2bapd-inclu}) in proper sense has not been given and (ii) smoothness property of the solution is also unknown. We set those aspects aside as they are beyond the scope of this work. For more discussions, we refer to \cref{sec:discuss-well}.

To show the usefulness of our model, we prove the exponential decay under the smooth assumption: $f\in\mathcal S_{\mu_f,L_f}^{1,1}$ and $g\in\mathcal S_{\mu_g,L_g}^{1,1}$. In this setting, the differential inclusion \cref{eq:2bapd} becomes a standard first-order dynamical system, with subgradients being Lipschitzian gradients,
and it is not hard to conclude the well-posedness
of a classical $C^1$ solution by standard theory of ordinary differential equations. 
\begin{thm}\label{thm:ASPD-exp}
	Assume $f\in\mathcal S_{\mu_f,L_f}^{1,1}$ and $g\in\mathcal S_{\mu_g,L_g}^{1,1}$ with $\mu_f,\,\mu_g\geq  0$. Let $\Theta=(\theta,\gamma,\beta)$  solve \cref{eq:2bapd-scaling} and $X=(x,y,v,w,\lambda)$ be the unique $C^1$ solution to \cref{eq:2bapd}, then it holds that
	\begin{equation}\label{eq:dEt}
		\frac{\dd }{\dd t} 	\mathcal E(\Theta,X)
		\leq  -	\mathcal E(\Theta,X)
		-\frac{\mu_f}{2}\nm{x'}^2-\frac{\mu_g}{2}\nm{y'}^2,
	\end{equation}
	which yields the exponential decay
	\begin{equation}\label{eq:Et-exp}
		2e^{t}\mathcal E(\Theta(t),X(t))
		+\int_{0}^{t}e^{s}
		\left(\mu_f\nm{x'(s)}^2+\mu_g\nm{y'(s)}^2\right){\rm d}s
		\leq  2\mathcal E(\Theta(0),X(0)),
	\end{equation}
	for all $ 0\leq  t<\infty$.
\end{thm}
\begin{proof}
	As \cref{eq:Et-exp} can be obtained directly from \cref{eq:dEt}, it is sufficient to establish 
	the latter. Let us start from the identity $		\frac{\dd }{\dd t}	\mathcal E\left(\Theta,X\right)
	= \dual{\nabla_{\Theta} \mathcal E,\Theta'}+\dual{\nabla_X \mathcal E,X'} $.
	By \cref{eq:2bapd-scaling,eq:Et}, it is trivial that
	\[
	\begin{aligned}
		\dual{\nabla_{\Theta} \mathcal E,\Theta'}
		={}&	-\frac{\theta}{2}\nm{\lambda-\lambda^*}^2	
		+ \frac{\mu_f-\gamma}{2}\nm{v-x^*}^2
		+ \frac{\mu_g-\beta}{2}\nm{w-y^*}^2,
	\end{aligned}
	\]
	and according to \cref{eq:2bapd}, a direct computation gives
	\[
	\begin{aligned}
		\dual{\nabla_X \mathcal E,X'} 
		={}&\dual{\lambda-\lambda^*, \nabla_\lambda \mathcal L (v,w,\lambda)}	+\dual{v-x,\nabla_x\mathcal L (x,y,\lambda^*)}
		+\dual{w-y,\nabla_y\mathcal L (x,y,\lambda^*)}\\
		{}&\quad 	+\dual{v-x^*,\mu_f(x-v) -\nabla_x\mathcal L (x,y,\lambda)}
		+\dual{w-y^*,\mu_g(y-w) -\nabla_y\mathcal L (x,y,\lambda)}.
	\end{aligned}
	\]
	Shifting $\lambda$ to $\lambda^*$ yields
	\[
	\begin{aligned}
		{}&-\dual{v-x^*,\nabla_x\mathcal L (x,y,\lambda)}
		-\dual{w-y^*,\nabla_y\mathcal L (x,y,\lambda)}\\
		=	&-\dual{v-x^*,\nabla_x\mathcal L (x,y,\lambda^*)}
		-\dual{w-y^*,\nabla_y\mathcal L (x,y,\lambda^*)}
		-\dual{\lambda-\lambda^*,Av+Bw-b},
	\end{aligned}
	\]
	where we have used 	the optimality condition $Ax^*+By^* = b$. It follows from \cref{eq:ineq-mu} that
	\begin{equation}\label{eq:dEXdX}
		\begin{aligned}
			\dual{\nabla_X \mathcal E,X'}
			={}&
			\dual{x^*-x,\nabla_x\mathcal L (x,y,\lambda^*)}
			+\dual{y^*-y,\nabla_y\mathcal L (x,y,\lambda^*)}\\
			{}&\qquad 	+\mu_f\dual{x-v ,v-x^*}+\mu_g\dual{y-w ,w-y^*}\\
			\leq {}& 
			\mathcal L (x^*,y^*,\lambda)-\mathcal L (x,y,\lambda^*)-\frac{\mu_f}{2}\nm{x-x^*}^2
			-\frac{\mu_g}{2}\nm{y-y^*}^2\\
			{}&\qquad 	+\mu_f\dual{x-v ,v-x^*}+\mu_g\dual{y-w ,w-y^*}.
		\end{aligned}
	\end{equation}
	In view of the trivial but useful identity of vectors
	\begin{equation}\label{eq:x-y-z}
		2\dual{u-z,z-a} = \nm{u-a}^2-\nm{z-a}^2-\nm{u-z}^2\quad \forall\,u,z,a,
	\end{equation}
	we rearrange the last two cross terms in \cref{eq:dEXdX} and put everything together to get
	\[
	\frac{\dd }{\dd t}	\mathcal E\left(\Theta,X\right)\leq  
	-\mathcal E(\Theta,X)
	-\frac{\mu_f}{2}\nm{x-v}^2-\frac{\mu_g}{2}\nm{y-w}^2.
	\]
	Observing that $x-v=x'$ and $y-w=y'$, we obtain \cref{eq:dEt} and complete the proof.
\end{proof}
\begin{rem}
	Thanks to the three scaling parameters 	introduced in \cref{eq:2bapd-scaling}, the exponential decay of the Lyapunov function \cref{eq:Et} holds uniformly for $\mu_f,\,\mu_g\geq  0$. In discrete level, it allows us to treat convex and (partially) strongly convex cases in a unified manner and obtain automatically changing parameters by implicit discretization of \cref{eq:2bapd-scaling}, which is the key for our delicate mixed-type estimates. 
\end{rem}
\begin{rem}\label{rem:rate}
	From \cref{eq:Et-exp} we have the exponential decay rate of the Lagrangian duality gap:
	\[
	\mathcal L(x(t),y(t),\lambda^*)-\mathcal L(x^*,y^*,\lambda(t))=O(e^{-t}).
	\]
	Invoking the proofs of \cite[Lemma 2.1]{luo_acc_primal-dual_2021} and \cite[Corollary 2.1]{luo_primal-dual_2022}, we can further establish 
	\[
	\nm{Ax(t)+By(t)-b} + \snm{F(x(t),y(t))-F^*}=O(e^{-t}).
	\]
	Moreover, by \cref{eq:ineq-mu} and the above estimates, we conclude that
	\[
	\mu_f\nm{x(t)-x^*}^2
	+\mu_g\nm{y(t)-y^*}^2=O(e^{-t}),
	\]
	which means strong convergence $x(t)\to x^*$ (or $y(t)\to y^*$) follows if $\mu_f>0$ (or $\mu_g>0$).
\end{rem}

\section{The First Family of Methods}
\label{sec:2bapd-im-x-im-y}
We now turn to the numerical aspect of our continuous model \cref{eq:2bapd}. In view of \cref{eq:partial-L},
the ways to discretize $(x,\lambda)$ in \eqref{eq:2bapd-v} and $(y,\lambda)$ in \eqref{eq:2bapd-w} are crucial, and $\lambda$ plays an important role of decoupling $x$ and $y$. In this work, 
we always use the same discretization for $\lambda$ in \eqref{eq:2bapd-v} and \eqref{eq:2bapd-w}, and for the case of different choices, we refer to the discussion in \cref{sec:discuss-suc}.

In this section, we impose the following assumption:
\begin{assum}
	\label{assum:f-g}
	$f\in\mathcal S_{\mu_f}^0(\mathcal X)$ with $\mu_f\geq  0$ and 
	$g\in\mathcal S_{\mu_g}^0(\mathcal Y)$ with $\mu_g\geq  0$.
\end{assum}
For this nonsmooth setting, we adopt implicit discretizations $(x_{k+1},\barlk)$ and $(y_{k+1},\barlk)$ for \eqref{eq:2bapd-v} and \eqref{eq:2bapd-w}, where $\barlk$ is to be determined. That is, given the initial guess $(x_0,v_0,y_0,w_0,\lambda_0)$, consider an implicit discretization for \cref{eq:2bapd}:
\begin{subnumcases}{\label{eq:2bapd-im-x-im-y-lbar}}
	\label{eq:2bapd-im-x-im-y-lbar-x}
	\frac{x_{k+1}-x_k}{\alpha_k}=v_{k+1} - x_{k+1},\\
	\label{eq:2bapd-im-x-im-y-lbar-v}
	\gamma_{k} \frac{v_{k+1}-v_k}{\alpha_k} \in {}\mu_f(x_{k+1}-v_{k+1})- \partial_x\mathcal L(x_{k+1},y_{k+1},\barlk),\\
	\theta_{k} \frac{\lambda_{k+1}-\lambda_k}{\alpha_k} = {}\nabla_\lambda \mathcal L(v_{k+1} ,w_{k+1} ,\lambda_{k+1}),		\label{eq:2bapd-im-x-im-y-lbar-l}\\
	\beta_{k}\frac{w_{k+1}-w_k}{\alpha_k} \in {}\mu_g(y_{k+1}-w_{k+1})- \partial_{y}\mathcal L(x_{k+1},y_{k+1},\barlk),
	\label{eq:2bapd-im-x-im-y-lbar-w}\\
	\label{eq:2bapd-im-x-im-y-lbar-y}
	\frac{y_{k+1}-y_k}{\alpha_k}=w_{k+1} - y_{k+1},
\end{subnumcases}
where $\alpha_k>0$ is the step size and the parameter system \cref{eq:2bapd-scaling} is discretized implicitly by 
\begin{equation}
	\label{eq:2bapd-tk}
	\frac{\theta_{k+1}-\theta_{k}}{\alpha_k}={}-\theta_{k+1}, \quad
	\frac{\gamma_{k+1}-\gamma_{k}}{\alpha_k}={}\mu_f-\gamma_{k+1},\quad
	\frac{\beta_{k+1}-\beta_{k}}{\alpha_k}={}\mu_g-\beta_{k+1},
\end{equation}
with initial conditions: $	\theta_0 = 1,\,\gamma_0>0$ and $\beta_0>0$. 

Rearrange \cref{eq:2bapd-im-x-im-y-lbar} in the usual primal-dual formulation:
\begin{subnumcases}{	\label{eq:2bapd-im-x-im-y-lbar-arg}}
	v_{k+1} = x_{k+1}+(x_{k+1}-x_k)/\alpha_k,
	\label{eq:2bapd-im-x-im-y-lbar-v-arg}\\
	x_{k+1}
	={}\mathop{\argmin}\limits_{x\in\mathcal X}
	\left\{
	\mathcal L(x,y_{k+1},\barlk)
	+\frac{\eta_{f,k}}{2\alpha^2_k}\nm{x-\widetilde{x}_k}^2\right\},
	\label{eq:2bapd-im-x-im-y-lbar-x-arg}\\
	\lambda_{k+1}
	={}\lambda_k+\alpha_k/\theta_k(Av_{k+1}+Bw_{k+1}-b),
	\label{eq:2bapd-im-x-im-y-lbar-l-arg}\\		
	y_{k+1}
	={}\mathop{\argmin}\limits_{y\in\mathcal Y}
	\left\{
	\mathcal L(x_{k+1},y,\barlk)
	+\frac{\eta_{g,k}}{2\alpha^2_k}\nm{y-\widetilde{y}_k}^2\right\},
	\label{eq:2bapd-im-x-im-y-lbar-y-arg}\\
	w_{k+1} = y_{k+1}+(y_{k+1}-y_k)/\alpha_k,
	\label{eq:2bapd-im-x-im-y-lbar-w-arg}
\end{subnumcases}
where $\eta_{f,k}:=(\alpha_k+1)\gamma_k+\mu_f\alpha_k,\,\eta_{g,k}:=(\alpha_k+1)\beta_k+\mu_g\alpha_k$ and 
\begin{equation}
	\label{eq:tilde-xk}
	\widetilde{x}_k :={}x_k+\frac{\alpha_k\gamma_k}{\eta_{f,k}}(v_k-x_k),\quad
	\widetilde{y}_k :={}y_k+\frac{\alpha_k\beta_k}{\eta_{g,k}}(w_k-y_k).
\end{equation}
Note that \cref{eq:2bapd-im-x-im-y-lbar-arg} is an {\it informal} expression since  the term $\barlk$ has not been determined yet. It brings hidden augmented terms for \eqref{eq:2bapd-im-x-im-y-lbar-x-arg} and \eqref{eq:2bapd-im-x-im-y-lbar-y-arg} with possible linearization and decoupling, and 
different choices lead to our first family of methods. Specifically, we shall adopt two semi-implicit candidates \cref{eq:2bapd-im-x-im-y-lbar-lv,eq:2bapd-im-x-im-y-lbar-lw} and the explicit one \cref{eq:2bapd-im-x-im-y-lbar-lk};
see \cref{sec:ex-x-im-y-lbar,sec:ex-x-im-y-lbar-alter,sec:ex-x-im-y-lk} for more details. 

Below, we give a one-iteration analysis for the implicit scheme \cref{eq:2bapd-im-x-im-y-lbar}. Then the nonergodic mixed-type convergence rates of our first family of methods can be obtained.
\subsection{A single-step analysis}
For the sequence $\{(x_k,v_k,y_k,w_k,\lambda_k)\}_{k=0}^{\infty}$ generated by \cref{eq:2bapd-im-x-im-y-lbar} and the parameter sequence $\{(\theta_k,\gamma_k,\beta_k)\}_{k=0}^{\infty}$ defined by \cref{eq:2bapd-tk}, we introduce a discrete Lyapunov function
\begin{equation}\label{eq:Ek-apd}
	\mathcal E_k: = \mathcal L (x_k,y_k,\lambda^*)-\mathcal L (x^*,y^*,\lambda_k)
	+\frac{\gamma_k}{2}\nm{v_k-x^*}^2
	+\frac{\beta_k}{2}\nm{w_k-y^*}^2
	+\frac{\theta_k}{2}\nm{\lambda_k-\lambda^*}^2,
\end{equation}
which is the discrete analogue of \cref{eq:Et} and will be used for all the forthcoming methods.
\begin{lem}\label{lem:2bapd-im-x-im-y-one-step}
	Let $k\in\mathbb N$ be fixed. For the implicit scheme \cref{eq:2bapd-im-x-im-y-lbar} with \cref{assum:f-g} and the step size $\alpha_k>0$, we have that
	\begin{equation}\label{eq:2bapd-im-x-im-y-one-step}
		\begin{aligned}
			\mathcal E_{k+1}-\mathcal E_k
			\leq &-\alpha_k\mathcal E_{k+1}+
			\frac{\theta_{k}}{2}\nm{\lambda_{k+1}-\barlk}^2	
			-	\frac{\gamma_{k}}{2}\nm{v_{k+1}-v_k}^2
			-	\frac{\beta_{k}}{2}\nm{w_{k+1}-w_k}^2.
		\end{aligned}
	\end{equation}
\end{lem}
\begin{proof}
	Let us calculate the difference $	\mathcal E_{k+1}-	\mathcal E_k=\mathbb I_1+\mathbb I_2+\mathbb I_3+\mathbb I_4$, where
	\begin{equation}\label{eq:2bapd-im-x-im-y-I1-I4}
		\begin{aligned}
			\mathbb I_1:	={}&\mathcal L(x_{k+1},y_{k+1},\lambda^*)-\mathcal L(x_k,y_k,\lambda^*),\\
			\mathbb I_2:=	{}&\frac{\theta_{k+1}}{2}
			\nm{\lambda_{k+1}-\lambda^*}^2 - 
			\frac{\theta_{k}}{2}
			\nm{\lambda_{k}-\lambda^*}^2,\\
			\mathbb I_3:=	{}&\frac{\gamma_{k+1}}{2}
			\nm{v_{k+1}-x^*}^2 - 
			\frac{\gamma_{k}}{2}
			\nm{v_{k}-x^*}^2,\\
			\mathbb I_4:=	{}&\frac{\beta_{k+1}}{2}
			\nm{w_{k+1}-y^*}^2 - 
			\frac{\beta_{k}}{2}
			\nm{w_{k}-y^*}^2.
		\end{aligned}
	\end{equation}
	In what follows, we aim to estimate the above four terms one by one.
	
	In view of \eqref{eq:2bapd-im-x-im-y-lbar-x-arg} and \eqref{eq:2bapd-im-x-im-y-lbar-y-arg}, it is clear that $(x_{k+1},y_{k+1})\in\mathcal X\times \mathcal Y$. 
	By \eqref{eq:2bapd-im-x-im-y-lbar-v} and \eqref{eq:2bapd-im-x-im-y-lbar-w}, we have 
	\begin{align}
		\label{eq:pk1-2bapd-im-x-im-y-lbar}
		p_{k+1} :={}& \mu_f(x_{k+1}-v_{k+1})-\gamma_k\frac{v_{k+1}-v_k}{\alpha_k}
		\in\partial_x\mathcal L(x_{k+1},y_{k+1},\barlk),\\
		q_{k+1} :={}& \mu_g(y_{k+1}-w_{k+1})-\beta_k\frac{w_{k+1}-w_k}{\alpha_k}
		\in\partial_{y}\mathcal L(x_{k+1},y_{k+1},\barlk).
		\label{eq:qk1-2bapd-im-x-im-y-lbar}
	\end{align}
	Thanks to the inequality \cref{eq:ineq-mu}, it follows that
	\begin{equation}\label{eq:I1-2bapd-im-x-im-y}
		\begin{aligned}
			\mathbb I_1
			={}&\mathcal L(x_{k+1},y_{k+1},\barlk)-\mathcal L(x_k,y_k,\barlk)\\
			{}&\quad
			+\dual{\lambda^*-\barlk,A(x_{k+1}-x_k)+B(y_{k+1}-y_k)}\\
			\leq  {}&
			\dual{p_{k+1},x_{k+1}-x_k}
			+\dual{q_{k+1},y_{k+1}-y_k}\\
			{}&\quad		+\dual{\lambda^*-\barlk,A(x_{k+1}-x_k)+B(y_{k+1}-y_k)}.
		\end{aligned}
	\end{equation}
	By the equation of the sequence $\{\theta_k\}_{k=0}^{\infty}$ in \cref{eq:2bapd-tk}, there holds
	\[
	\begin{aligned}
		\mathbb I_2	=	{}&			\frac{\theta_{k+1}-\theta_k}{2}
		\nm{\lambda_{k+1}-\lambda^*}^2+\frac{\theta_{k}}{2}
		\left(\nm{\lambda_{k+1}-\lambda^*}^2-\nm{\lambda_{k}-\lambda^*}^2 
		\right)\\
		=&		-\frac{\alpha_k\theta_{k+1}}{2}\nm{\lambda_{k+1}-\lambda^*}^2
		+\theta_{k}\dual{\lambda_{k+1}-\lambda_k,\lambda_{k+1}-\lambda^*}
		-\frac{\theta_{k}}{2}\nm{\lambda_{k+1}-\lambda_k}^2.
	\end{aligned}
	\]	
	To match the term $\barlk$ in \cref{eq:I1-2bapd-im-x-im-y}, we use \eqref{eq:2bapd-im-x-im-y-lbar-l} to rewrite the last two terms
	\[
	\begin{aligned}
		{}&		\theta_{k}\dual{\lambda_{k+1}-\lambda_k,\lambda_{k+1}-\lambda^*}
		-\frac{\theta_{k}}{2}\nm{\lambda_{k+1}-\lambda_k}^2\\
		={}&\theta_{k}\dual{\lambda_{k+1}-\lambda_k,\lambda_{k+1}- \barlk+ \barlk-\lambda^*}-\frac{\theta_{k}}{2}
		\nm{\lambda_{k+1}-\lambda_k}^2\\	
		={}&\alpha_{k}\dual{Av_{k+1}+Bw_{k+1}-b, \barlk-\lambda^*}
		+\frac{\theta_{k}}{2}\nm{\lambda_{k+1}-\barlk}^2-\frac{\theta_{k}}{2}\nm{\lambda_{k}-\barlk}^2.
	\end{aligned}
	\]	
	This implies the estimate
	\begin{equation}
		\label{eq:I2-2bapd-im-x-im-y}
		\begin{aligned}
			\mathbb I_2	
			\leq  {}&\alpha_{k}\dual{Av_{k+1}+Bw_{k+1}-b, \barlk-\lambda^*}
			+\frac{\theta_{k}}{2}\nm{\lambda_{k+1}-\barlk}^2	-\frac{\alpha_k\theta_{k+1}}{2}\nm{\lambda_{k+1}-\lambda^*}^2.
		\end{aligned}
	\end{equation}
	
	Similarly, using the equation of $\{\gamma_k\}_{k=0}^{\infty}$ in \cref{eq:2bapd-tk}, we have
	\begin{equation}\label{eq:I3-2bapd-im-x-im-y-mid}
		\begin{aligned}
			\mathbb I_3={}&			\frac{\gamma_{k+1}-\gamma_k}{2}
			\nm{v_{k+1}-x^*}^2+\frac{\gamma_{k}}{2}
			\left(\nm{v_{k+1}-x^*}^2-\nm{v_{k}-x^*}^2 
			\right)\\
			={}&
			\frac{\alpha_k(\mu_f-\gamma_{k+1})}{2}
			\nm{v_{k+1}-x^*}^2		-	\frac{\gamma_{k}}2
			\nm{v_{k+1}-v_k}^2
			+		\gamma_{k}
			\dual{v_{k+1}-v_k, v_{k+1} -x^*}.
		\end{aligned}
	\end{equation}
	In view of \cref{eq:pk1-2bapd-im-x-im-y-lbar}, we rewrite the last cross term by that
	\[
	\begin{aligned}
		\gamma_{k}
		\dual{v_{k+1}-v_k, v_{k+1} -x^*}
		={}&	\mu_f\alpha_k\dual{x_{k+1}-v_{k+1}, v_{k+1} -x^*}
		-\alpha_k	\dual{p_{k+1}, v_{k+1} -x^*}.
	\end{aligned}
	\]
	Using \cref{eq:x-y-z} and \eqref{eq:2bapd-im-x-im-y-lbar-v-arg} and summarizing the above decompositions yield that
	\begin{equation}\label{eq:I3-2bapd-im-x-im-y}
		\begin{aligned}
			\mathbb I_3=&-
			\frac{\alpha_k \gamma_{k+1}}{2}
			\nm{v_{k+1}-x^*}^2-	\frac{\gamma_{k}}2\nm{v_{k+1}-v_k}^2
			-\frac{\mu_f\alpha_k}{2}\nm{x_{k+1}-v_{k+1}}^2
			\\
			{}&\quad+
			\frac{\mu_f\alpha_k}{2}
			\nm{x_{k+1}-x^*}^2
			-\alpha_k
			\dual{p_{k+1}, x_{k+1} -x^*}
			-\dual{p_{k+1}, x_{k+1} -x_k}.
		\end{aligned}
	\end{equation}
	Analogously, by \cref{eq:2bapd-tk}, \cref{eq:qk1-2bapd-im-x-im-y-lbar} and \eqref{eq:2bapd-im-x-im-y-lbar-w-arg}, we have
	\begin{equation}\label{eq:I4-2bapd-im-x-im-y}
		\begin{aligned}
			\mathbb I_4=&-
			\frac{\alpha_k \beta_{k+1}}{2}
			\nm{w_{k+1}-y^*}^2-	\frac{\beta_{k}}2\nm{w_{k+1}-w_k}^2
			-\frac{\mu_g\alpha_k}{2}\nm{y_{k+1}-w_{k+1}}^2
			\\
			{}&\quad+
			\frac{\mu_g\alpha_k}{2}
			\nm{y_{k+1}-y^*}^2
			-\alpha_k
			\dual{q_{k+1}, y_{k+1} -y^*}
			-
			\dual{q_{k+1}, y_{k+1} -y_k}.
		\end{aligned}
	\end{equation}
	
	Now, collecting \cref{eq:I2-2bapd-im-x-im-y,eq:I3-2bapd-im-x-im-y,eq:I4-2bapd-im-x-im-y,eq:ineq-mu},
	we arrive at the upper bound
	\[
	\begin{aligned}
		\mathbb I_2+\mathbb I_3+\mathbb I_4
		\leq  	&-\alpha_k\mathcal E_{k+1}
		+\frac{\theta_{k}}{2}\nm{\lambda_{k+1}-\barlk}^2-	\frac{\beta_{k}}{2}\nm{w_{k+1}-w_k}^2\\
		{}&\quad-\dual{p_{k+1}, x_{k+1} -x_k}-\dual{q_{k+1}, y_{k+1} -y_k}
		-	\frac{\gamma_{k}}{2}\nm{v_{k+1}-v_k}^2\\
		{}&\qquad 
		-\dual{\lambda^*-\barlk,A(x_{k+1}-x_k)+B(y_{k+1}-y_k)}
		.
	\end{aligned}
	\]
	Plugging \cref{eq:I1-2bapd-im-x-im-y} into the above estimate gives 
	\[
	\begin{aligned}
		\mathcal E_{k+1}-\mathcal E_k
		\leq &-\alpha_k\mathcal E_{k+1}+
		\frac{\theta_{k}}{2}\nm{\lambda_{k+1}- \barlk}^2
		-	\frac{\gamma_{k}}{2}\nm{v_{k+1}- v_k}^2
		-	\frac{\beta_{k}}{2}\nm{w_{k+1}-w_k}^2.
	\end{aligned}
	\]
	This establishes \cref{eq:2bapd-im-x-im-y-one-step} and finishes the proof of this lemma.
\end{proof}
\subsection{A semi-implicit choice}
\label{sec:ex-x-im-y-lbar}
According to the single step estimate \cref{eq:2bapd-im-x-im-y-one-step}, it is evident that the implicit choice $\barlk=\lambda_{k+1}$ indicates the contraction 
\begin{equation}\label{eq:2bapd-ex-x-im-y-lk1-contrac-Ek}
	\begin{aligned}
		\mathcal E_{k+1}-	\mathcal E_k
		\leq  -\alpha_k		\mathcal E_{k+1},
	\end{aligned}
\end{equation}
which holds for any $\alpha_k>0$, even for $\mu_f=\mu_g=0$.
This together with the fact (cf.\cref{eq:2bapd-tk})
\begin{equation}\label{eq:tk-prod}
	\theta_{k+1}=	\frac{\theta_{k}}{1+\alpha_k}\quad\Longrightarrow\quad 
	\theta_k = \prod_{i=0}^{k-1}\frac{1}{1+\alpha_i}
\end{equation}
implies that $\mathcal E_k\leq  \theta_{k}\mathcal E_0$ and the linear rate $\theta_k \leq   (1+\alpha_{\min})^{-k}$ follows immediately if $\alpha_k\geq  \alpha_{\min}>0$. But this does not lead to a splitting method since by \eqref{eq:2bapd-im-x-im-y-lbar-l-arg}, $\lambda_{k+1}$ depends on $v_{k+1}$ and $w_{k+1}$. In other words, $x_{k+1}$ and $y_{k+1}$ are coupled with each other; see \cref{sec:discuss-im} for more discussions.

Hence, let us consider other semi-implicit choices that decouple $x_{k+1}$ and $y_{k+1}$.
Recall again the estimate \cref{eq:2bapd-im-x-im-y-one-step}, which says if we want to maintain the contraction property \cref{eq:2bapd-ex-x-im-y-lk1-contrac-Ek}, then the positive gain $\|\lambda_{k+1}-\barlk\|^2$ shall be controlled by additional two negative square norm terms $-\nm{v_{k+1}-v_k}^2$ and $-\nm{w_{k+1}-w_k}^2$. To do this, the relation 
\[
\nm{\lambda_{k+1}-\barlk}=O(\nm{w_{k+1}-w_k})\,\text{ or }\, \nm{\lambda_{k+1}-\barlk}=O(\nm{v_{k+1}-v_k})
\]
is important to be satisfied. In view of \eqref{eq:2bapd-im-x-im-y-lbar-l-arg}, we are suggested to consider 
\begin{equation}\label{eq:2bapd-im-x-im-y-lbar-lv}
	\barlk = \lambda_k+
	\alpha_k/\theta_k
	\left(Av_{k+1}+Bw_{k}-b\right),
\end{equation}
which gives the desired identity
\begin{equation}\label{eq:lk1-barlk-im-x-im-y-lw}
	\lambda_{k+1}-\barlk=\alpha_k/\theta_kB(w_{k+1}-w_k).
\end{equation}
Then by \cref{eq:2bapd-im-x-im-y-one-step}, the contraction \cref{eq:2bapd-ex-x-im-y-lk1-contrac-Ek} follows 
directly, provided that 
\begin{equation}\label{eq:key-2bapd-im-x-im-y-lw}
	\frac{\theta_k}{2}\nm{\lambda_{k+1}-\barlk}^2=\frac{\alpha_k^2}{2\theta_k}
	\nm{B(w_{k+1}-w_k)}^2\leq  \frac{\beta_k}{2}\nm{w_{k+1}-w_k}^2,
\end{equation}
which can be easily promised if $		\alpha_k^2\nm{B}^2\leq   \theta_{k}\beta_{k}$.

For $\tau>0$, introduce the proximal operator of $g$ by that
\begin{align}
	\label{eq:prox-g}
	\prox_{\tau g}^{\mathcal Y}(z): = {}&
	\mathop{\argmin}_{y\in\mathcal Y}\left\{g(y)+\frac{1}{2\tau}\nm{y-z}^2\right\}
	\quad \forall\,z\in\R^n.
\end{align}
With the choice \cref{eq:2bapd-im-x-im-y-lbar-lv}, we reformulate \cref{eq:2bapd-im-x-im-y-lbar-arg} as the following iteration
\begin{equation}\label{eq:2bapd-im-x-im-y-lv}
	\left\{
	\begin{aligned}
		{}&			\widehat{\lambda}_k={} \lambda_k-\theta_k^{-1}\left(Ax_k+By_k-b\right)
		+\alpha_k/\theta_kB(w_k-y_k),\\
		{}&		x_{k+1}
		={}\mathop{\argmin}\limits_{x\in\mathcal X}
		\left\{
		\mathcal L_{\sigma_k}(x,y_{k},\widehat{\lambda}_k)
		+\frac{\eta_{f,k}}{2\alpha^2_k}\nm{x-\widetilde{x}_k}^2\right\},\quad \sigma_k = 1/\theta_{k+1},\\
		{}&		v_{k+1} ={} x_{k+1}+(x_{k+1}-x_k)/\alpha_k,\\
		{}&	\barlk ={} \lambda_k+
		\alpha_k/\theta_k
		\left(Av_{k+1}+Bw_{k}-b\right),\\
		{}&		y_{k+1} = {}\prox_{\tau_kg}^{\mathcal Y}(\widetilde{y}_k-\tau_kB^{\top}\barlk),\quad \tau_k = \alpha_k^2/\eta_{g,k},\\
		{}&		w_{k+1} ={}y_{k+1}+(y_{k+1}-y_k)/\alpha_k,\\
		{}&	\lambda_{k+1}
		={}\lambda_k+\alpha_k/\theta_k(Av_{k+1}+Bw_{k+1}-b),
	\end{aligned}
	\right.
\end{equation}
where $(\widetilde{x}_k ,\widetilde{y}_k)$ are defined by \cref{eq:tilde-xk} and $(\eta_{f,k},\eta_{g,k})$ are the same as that in \cref{eq:2bapd-im-x-im-y-lbar-arg}.
As $\bar{\lambda}_{k+1}$ depends only on $v_{k+1}$ and $w_k$, we see that 
(i) $x_{k+1}$ and $y_{k+1}$ are weakly coupled with each other, in the sense that they can be updated sequentially; (ii) the augmented term is used for computing $x_{k+1}$ but it has been linearized for updating $y_{k+1}$, which involves only the proximal calculation of $g$.

For simplicity, in the rest of this paper, we set
\begin{equation}\label{eq:b0-g0}
	\gamma_0={}\mu_f\text{ for } \mu_f>0,\quad\text{and }
	\beta_0={}\mu_g\text{ for } \mu_g>0.
\end{equation}
Then by \cref{eq:2bapd-tk}, both $\{\gamma_k\}_{k=0}^\infty$ and $\{\beta_k\}_{k=0}^\infty$ are decreasing, and it is clear that $\mu_f\leq \gamma_k\leq\gamma_0$ and $\mu_g\leq \beta_k\leq \beta_0$ for all $k\in\mathbb N$. Moreover, we claim that 
\begin{equation}\label{eq:est-bk-gk}
	\beta_k\geq \theta_k\beta_0,\quad 	\gamma_k\geq \theta_k\gamma_0.
\end{equation}
\begin{thm}\label{thm:conv-2bapd-im-x-im-y-lv}
	If $\barlk$ is chosen from \cref{eq:2bapd-im-x-im-y-lbar-lv}, then \cref{eq:2bapd-im-x-im-y-lbar-arg} reduces to \cref{eq:2bapd-im-x-im-y-lv}. Under \cref{assum:f-g}, the initial setting \cref{eq:b0-g0} and the condition 
	$		\alpha_k^2\nm{B}^2= \theta_{k}\beta_{k}$,
	it holds that $			\mathcal E_{k+1}-	\mathcal E_k
	\leq  -\alpha_k		\mathcal E_{k+1}$.
	Moreover, we have $\{(x_k,y_k)\}_{k=1}^{\infty}\subset \mathcal X\times \mathcal Y$ and 
	\begin{equation}	\label{eq:2bapd-im-x-im-y-lw-rate}
		\left\{
		\begin{aligned}
			{}&	\nm{Ax_{k}+By_k-b}\leq 		{}
			\theta_k		\mathcal R_0,	
			\\
			{}&	\mathcal L (x_k,y_k,\lambda^*)-\mathcal L (x^*,y^*,\lambda_k)
			\leq 		{}\theta_k\mathcal E_0,	\\		
			{}&	\snm{F(x_k,y_k)-F^*}
			\leq 	{} \theta_k(\mathcal E_0+\nm{\lambda^*}\mathcal R_0).
		\end{aligned}
		\right.
	\end{equation}
	Above, $		\mathcal R_0:=\sqrt{2\mathcal  E_0}+\nm{\lambda_0-\lambda^*}+\nm{Ax_0+By_0-b}$
	and 
	\begin{equation}
		\label{eq:2bapd-im-x-im-y-lw-tk}
		\theta_k\leq  
		\min\left\{\frac{Q}{Q+\sqrt{\beta_{0}}k},\,\frac{4Q^2}{(2Q+\sqrt{\mu_g} k)^2}\right\},
	\end{equation}
	where $Q=\nm{B}+\sqrt{\beta_{0}}$.
\end{thm}
\begin{proof}
	Based on the above discussions, to get the contraction $			\mathcal E_{k+1}-	\mathcal E_k
	\leq  -\alpha_k		\mathcal E_{k+1}$, we only need to verify \cref{eq:key-2bapd-im-x-im-y-lw} under the condition $		\alpha_k^2\nm{B}^2= \beta_{k}\theta_{k}$, which is trivial. By \cref{eq:tk-prod}, this gives $\mathcal E_k\leq  \theta_k\mathcal E_0$  and also implies that
	\[
	\mathcal L (x_k,y_k,\lambda^*)-\mathcal L (x^*,y^*,\lambda_k)
	\leq 		{}\theta_k\mathcal E_0.
	\]
	Following the proof of \cite[Theorem 3.1]{luo_accelerated_2021}, we can establish 
	\[
	\nm{Ax_{k}+By_k-b}\leq 		{}
	\theta_k		\mathcal R_0,\,\,\text{and}\,\, \snm{F(x_k,y_k)-F^*}
	\leq 	{} \theta_k(\mathcal E_0+\nm{\lambda^*}\mathcal R_0),
	\]
	which yields \cref{eq:2bapd-im-x-im-y-lw-rate}.

		It remains to verify the decay estimate \cref{eq:2bapd-im-x-im-y-lw-tk}.
		Since $		\alpha_k^2\nm{B}^2= \theta_{k}\beta_{k}$, we obtain
		\[
		\alpha_k = \sqrt{\beta_k\theta_k}/\nm{B}\leq 
		\sqrt{\beta_{0}}/\nm{B}
		\quad\Longrightarrow\quad 
		\frac{\theta_{k+1}}{\theta_k}=\frac{1}{1+\alpha_k}\geq \frac{\nm{B}}{\nm{B}+\sqrt{\beta_{0}}}.
		\]
		By \cref{eq:2bapd-tk,eq:est-bk-gk}, it holds that
		\begin{equation}\label{eq:diff-tk-2bapd-im-x-im-y-lw}
			\theta_{k+1}-\theta_{k} = -\alpha_k\theta_{k+1} = -\frac{\sqrt{\beta_k\theta_{k}}\theta_{k+1}}{\nm{B}}
			\leq -\frac{\sqrt{\beta_{0}}\theta_k\theta_{k+1}}{\nm{B}},
		\end{equation}
		and using \cref{lem:est-yk-case1} implies
		\begin{equation}\label{eq:2bapd-im-x-im-y-lw-tk1}
			\theta_k\leq  \frac{\nm{B}+\sqrt{\beta_{0}}}{\nm{B}+\sqrt{\beta_{0}}+\sqrt{\beta_{0}}k}.
		\end{equation}
		
		On the other hand, since $\beta_k\geq\mu_g$, \cref{eq:diff-tk-2bapd-im-x-im-y-lw}  becomes
		\[
		\theta_{k+1}-\theta_{k} \leq  -\frac{\sqrt{\mu_g\theta_{k}}\theta_{k+1}}{\nm{B}}
		\quad\overset{\text{by \cref{lem:est-yk-case1} }}{\Longrightarrow}\quad 
		\theta_k\leq \frac{4(\nm{B}+\sqrt{\beta_{0}})^2}{(2(\nm{B}+\sqrt{\beta_{0}})+\sqrt{\mu_g }k)^2}.
		\]
		Note that this estimate and the previous one \cref{eq:2bapd-im-x-im-y-lw-tk1} hold true simultaneously. This yields \cref{eq:2bapd-im-x-im-y-lw-tk} and concludes the proof of this theorem.
	\end{proof}
	
	In \cref{thm:conv-2bapd-im-x-im-y-lv}, we have established the same convergence rate for the objective residual and the feasibility violation. For the special case: $B = -I,\,b = 0$, the separable problem \cref{eq:min-f-g-A-B} is equivalent to the unconstrained composite optimization 
	\begin{equation}\label{eq:min-P}
		\min_{x\in \mathcal X} \,P(x): =  f(x)+g(Ax),
	\end{equation}
	where $P(x) = F(x,Ax)$ and the minimal value is $P^* = F^*$. As a corollary of \cref{thm:conv-2bapd-im-x-im-y-lv}, we can derive the convergence rate with respect to the composite objective in \cref{eq:min-P}.
	\begin{coro}\label{thm:conv-2bapd-obj-res}
		Assume 	$B = -I,\,b = 0$ and let $\{x_k\}_{k=1}^{\infty}\subset \mathcal X$ be generated by \cref{eq:2bapd-im-x-im-y-lv} under the assumptions of \cref{thm:conv-2bapd-im-x-im-y-lv}.
		If $g$ is $M_g$-Lipschitz continuous, then
		\begin{equation}\label{eq:convergence-unconstrained}
			0\leq P(x_k)-P^* \leq \theta_k\left(\mathcal E_0+(\nm{\lambda^*}+M_g)\mathcal R_0\right),
		\end{equation}
		where $\theta_{k}$ satisfies the decay estimate \cref{eq:2bapd-im-x-im-y-lw-tk}.
	\end{coro}
	\begin{proof}
		It follows that
		\[\begin{aligned}
			P(x_k)-P^* = {}&F(x_k,Ax_k)-F^*=g(Ax_k)-g(y_k)+F(x_k,y_k)-F^*\\
			\leq {}&\snm{g(Ax_k)-g(y_k)}+\snm{F(x_k,y_k)-F^*}\\
			\leq {}&M_g\nm{Ax_k-y_k}+\snm{F(x_k,y_k)-F^*}\\
			\leq{}&\theta_k\left(\mathcal E_0+(\nm{\lambda^*}+M_g)\mathcal R_0\right).
		\end{aligned}
		\]
		In the last step, we used \cref{eq:2bapd-im-x-im-y-lw-rate}. This concludes the proof.
	\end{proof}
	
	\begin{rem}
		Observing the proof of \cref{thm:conv-2bapd-obj-res}, the estimate \cref{eq:convergence-unconstrained} depends solely on $\nm{Ax_{k}-y_k}$ and $\snm{F(x_k,y_k)-F^*}$. Hence, we claim that it holds true for all the rest methods with the corresponding decay rate of $\theta_k$.
	\end{rem}

	\subsection{Another semi-implicit choice}
	\label{sec:ex-x-im-y-lbar-alter}
	As the roles of $(x,f,A)$ and $(y,g,B)$ are symmetric in \cref{eq:2bapd-im-x-im-y-lbar-arg}, the previous choice \cref{eq:2bapd-im-x-im-y-lbar-lv} is also equivalent to 
	\begin{equation}\label{eq:2bapd-im-x-im-y-lbar-lw}
		\barlk= \lambda_k+
		\alpha_k/\theta_k
		\left(Av_{k}+Bw_{k+1}-b\right),
	\end{equation}
	which leads to 
	\begin{equation}	\label{eq:2bapd-im-x-im-y-lw-arg}
		\left\{
		\begin{aligned}
			{}&			\widehat{\lambda}_k={} \lambda_k-\theta_k^{-1}\left(Ax_k+By_k-b\right)+\alpha_k/\theta_kA(v_k-x_k),\\
			{}&	y_{k+1}
			={}\mathop{\argmin}\limits_{y\in\mathcal Y}
			\left\{
			\mathcal L_{\sigma_k}(x_k,y,\widehat{\lambda}_k)
			+\frac{\eta_{g,k}}{2\alpha^2_k}\nm{y-\widetilde{y}_k}^2\right\},\quad \sigma_k = 1/\theta_{k+1},\\
			{}&	w_{k+1} ={} y_{k+1}+(y_{k+1}-y_k)/\alpha_k,\\
			{}&	\barlk= {}\lambda_k+
			\alpha_k/\theta_k
			\left(Av_{k}+Bw_{k+1}-b\right),\\		
			{}&	x_{k+1} = \prox_{s_kf}^{\mathcal X}(\widetilde{x}_k-s_kA^{\top}	\barlk),\quad s_k = \alpha_k^2/\eta_{f,k},\\
			{}&	v_{k+1} = x_{k+1}+(x_{k+1}-x_k)/\alpha_k,\\	
			{}&\lambda_{k+1}
			={}\lambda_k+\alpha_k/\theta_k(Av_{k+1}+Bw_{k+1}-b),
		\end{aligned}
		\right.
	\end{equation}
	where $(\widetilde{x}_k ,\widetilde{y}_k,\eta_{f,k},\eta_{g,k})$ are the same as that in \cref{eq:2bapd-im-x-im-y-lv} and the proximal operator $\prox_{\tau_kf}^{\mathcal X}$ of $f$ can be defined similarly as \cref{eq:prox-g}.
	
	Below, we state the convergence rate of \cref{eq:2bapd-im-x-im-y-lw-arg} but omit the detailed proof, which is almost identical to that of \cref{thm:conv-2bapd-im-x-im-y-lv}.
	\begin{thm}\label{thm:conv-2bapd-im-x-im-y-lw}
		Applying the choice \cref{eq:2bapd-im-x-im-y-lbar-lw} to \cref{eq:2bapd-im-x-im-y-lbar-arg} gives \cref{eq:2bapd-im-x-im-y-lw-arg}. In addition, 
		under \cref{assum:f-g}, the initial setting \cref{eq:b0-g0} and the condition $		\alpha_k^2\nm{A}^2= \gamma_{k}\theta_{k}$, we have $\{(x_k,y_k)\}_{k=1}^{\infty}\subset \mathcal X\times \mathcal Y$, and the estimate \cref{eq:2bapd-im-x-im-y-lw-rate} holds true with
		\[
		\theta_k\leq  
		\min\left\{\frac{Q}{Q+\sqrt{\gamma_{0}}k},\,\frac{4Q^2}{(2Q+\sqrt{\mu_f} k)^2}\right\},
		\]
		where $Q=\nm{A}+\sqrt{\gamma_0}$.
	\end{thm}

\subsection{The explicit choice}
\label{sec:ex-x-im-y-lk}
Now, let us consider the explicit one:
\begin{equation}\label{eq:2bapd-im-x-im-y-lbar-lk}
	\barlk= 
	\lambda_k+\alpha_k/\theta_k
	\left(Av_{k}+Bw_k-b\right),
\end{equation}
which yields the following method
\begin{subnumcases}{		\label{eq:2bapd-im-x-im-y-lk-arg}}
	\barlk= {}
	\lambda_k+\alpha_k/\theta_k
	\left(Av_{k}+Bw_k-b\right),\\
	x_{k+1} = \prox_{s_kf}^{\mathcal X}(\widetilde{x}_k-s_kA^{\top}		\barlk),\quad s_k = \alpha_k^2/\eta_{f,k},
	\label{eq:2bapd-im-x-im-y-lk-x-arg}\\
	v_{k+1} = x_{k+1}+(x_{k+1}-x_k)/\alpha_k,\\	
	y_{k+1} = \prox_{\tau_kg}^{\mathcal Y}(\widetilde{y}_k-\tau_kB^{\top}		\barlk),\quad \tau_k = \alpha_k^2/\eta_{g,k},
	\label{eq:2bapd-im-x-im-y-lk-y-arg}\\
	w_{k+1} = y_{k+1}+(y_{k+1}-y_k)/\alpha_k,\\	
	\lambda_{k+1}
	={}\lambda_k+\alpha_k/\theta_k(Av_{k+1}+Bw_{k+1}-b),
	\label{eq:2bapd-im-x-im-y-lk-l-arg}
\end{subnumcases}
where $(\widetilde{x}_k ,\widetilde{y}_k,\eta_{f,k},\eta_{g,k})$ are the same as that in \cref{eq:2bapd-im-x-im-y-lv}.
Note that \cref{eq:2bapd-im-x-im-y-lk-arg} is a parallel linearized proximal ADMM since the two proximal steps in \eqref{eq:2bapd-im-x-im-y-lk-x-arg} and \eqref{eq:2bapd-im-x-im-y-lk-y-arg} are independent.

Recall that $M\lesssim N$ means $M\leq CN$ with some generic bounded constant $C>0$ that is independent of $A,B,\mu_f,\mu_g,\gamma_0$ and $\beta_0$ but can be different in each occurrence. 
\begin{thm}\label{thm:conv-2bapd-im-x-im-y-lk}
	Applying the explicit choice \cref{eq:2bapd-im-x-im-y-lbar-lk} to \cref{eq:2bapd-im-x-im-y-lbar-arg} leads to \cref{eq:2bapd-im-x-im-y-lk-arg}. Under
	\cref{assum:f-g}, the initial setting \cref{eq:b0-g0} and the condition
	\begin{equation}\label{eq:2bapd-im-x-im-y-lk-ak}
		2\alpha_k^2(\beta_k\nm{A}^2+\gamma_k\nm{B}^2)= \gamma_k\beta_k\theta_{k},
	\end{equation}
	we have $\{(x_k,y_k)\}_{k=1}^{\infty}\subset \mathcal X\times \mathcal Y$ and $			\mathcal E_{k+1}-	\mathcal E_k
	\leq  -\alpha_k		\mathcal E_{k+1}$.
	Moreover, if  $\gamma_0\beta_0\leq 2\beta_0\nm{A}^2+2\gamma_0\nm{B}^2$, then the estimate \cref{eq:2bapd-im-x-im-y-lw-rate} holds true with
	\begin{equation}\label{eq:2bapd-im-x-im-y-lk-rate}
		\theta_k\lesssim \min\left\{\frac{\nm{A}}{\sqrt{\gamma_0}k},\,\frac{\nm{A}^2}{\mu_f k^2}\right\}+\min\left\{\frac{\nm{B}}{\sqrt{\beta_0}k},\,\frac{\nm{B}^2}{\mu_g k^2}\right\}.
	\end{equation}
\end{thm}
\begin{proof}
	By \cref{eq:2bapd-im-x-im-y-lbar-lk} and \eqref{eq:2bapd-im-x-im-y-lk-l-arg}, we have 
	\begin{equation}\label{eq:diff-lk-ex}
		\lambda_{k+1}-\barlk = \alpha_k/\theta_{k}A(v_{k+1}-v_k)+\alpha_k/\theta_{k}B(w_{k+1}-w_k).
	\end{equation}
	Taking this into the one-iteration estimate \cref{eq:2bapd-im-x-im-y-one-step} gives 
	\[
	\begin{aligned}
		\mathcal E_{k+1}-	\mathcal E_k
		\leq{}&  -\alpha_k		\mathcal E_{k+1}
		+\frac{2\alpha_k^2\nm{A}^2-\gamma_{k}\theta_k}{2\theta_k}\nm{v_{k+1}-v_k}^2
		+\frac{2\alpha_k^2\nm{B}^2-\beta_{k}\theta_k}{2\theta_k}\nm{w_{k+1}-w_k}^2,
	\end{aligned}
	\]
	and invoking the relation \cref{eq:2bapd-im-x-im-y-lk-ak}, we obtain the contraction $			\mathcal E_{k+1}-	\mathcal E_k
	\leq  -\alpha_k		\mathcal E_{k+1}$. By using the proof of 	\cref{thm:conv-2bapd-im-x-im-y-lv}, the estimate \cref{eq:2bapd-im-x-im-y-lw-rate} can still be verified. 
	
	Let us prove the mixed-type estimate \cref{eq:2bapd-im-x-im-y-lk-rate}. 
	Since $\beta_k\leq \beta_0,\,\gamma_k\leq \gamma_0$ and $\theta_k\leq 1$, by \cref{eq:2bapd-im-x-im-y-lk-ak}, we have
	\[
	\alpha_k 
	\leq 
	\frac{\sqrt{\gamma_0\beta_0}}{\sqrt{2\beta_0\nm{A}^2+2\gamma_0\nm{B}^2}}:=\sigma\leq 1,
	\]
	and it follows that $\theta_{k+1}/\theta_k = 1/(1+\alpha_k)\geq 1/2$.
	Analogously to \cref{eq:diff-tk-2bapd-im-x-im-y-lw}, one has 
	\[
	\theta_{k+1}-\theta_k =-\alpha_k\theta_{k+1}
	\overset{\text{by \cref{eq:est-bk-gk,eq:2bapd-im-x-im-y-lk-ak}}}{\leq}
	-\sigma\theta_k\theta_{k+1}
	\quad\overset{\text{by \cref{lem:est-yk-case1} }}{\Longrightarrow}\quad 
	\theta_k
	\lesssim\frac{\nm{A}}{\sqrt{\gamma_0}k}+\frac{\nm{B}}{ \sqrt{\beta_0}k}.
	\]
	On the other hand, as $\beta_k\geq \mu_g$, we obtain
	\[
	\theta_{k+1}-\theta_{k}
	\overset{\text{by \cref{eq:2bapd-im-x-im-y-lk-ak}}}{\leq} -\frac{\sqrt{\gamma_0\mu_g}\theta_{k}\theta_{k+1}}{\sqrt{2\mu_g\nm{A}^2+2\gamma_0\theta_k\nm{B}^2}}
	\quad\overset{\text{by \cref{lem:est-yk-case2} }}{\Longrightarrow}\quad 
	\theta_k\lesssim \frac{\nm{A}}{\sqrt{\gamma_0}k} + \frac{\nm{B}^2}{\mu_gk^2}.
	\]
	Consequently, we have 
	\begin{equation}\label{eq:2bapd-im-x-im-y-lk-rate1}
		\theta_k\lesssim \frac{\nm{A}}{\sqrt{\gamma_0}k}+\min\left\{\frac{\nm{B}}{\sqrt{\beta_0}k},\,\frac{\nm{B}^2}{\mu_gk^2}\right\}.
	\end{equation}
	
	Moreover, since $\gamma_k \geq \mu_f$, repeating the above discussions and using \cref{lem:est-yk-case1,lem:est-yk-case2}, we conclude that
	\[
	\theta_k\lesssim \frac{\nm{A}^2}{\mu_f k^2}+\min\left\{\frac{\nm{B}}{\sqrt{\beta_0}k},\,\frac{\nm{B}^2}{\mu_gk^2}\right\}.
	\]
	Therefore, combining this with \cref{eq:2bapd-im-x-im-y-lk-rate1} leads to \cref{eq:2bapd-im-x-im-y-lk-rate} and completes the proof.
	%
\end{proof}
\begin{rem}
	Our method \cref{eq:2bapd-im-x-im-y-lk-arg} is close to the parallel type ADMM, such as the predictor corrector proximal multipliers (PCPM) method \cite{chen_proximal-based_1994}, the proximal-center based decomposition method (PCBDM) \cite{necoara_application_2008} and the decomposition algorithms in \cite{trandinh_combining_2013}. However, it is rare to see mixed-type estimates like \cref{eq:2bapd-im-x-im-y-lk-rate}, especially for partially strongly convex case $\mu_f+\mu_g>0$.
\end{rem}

\section{The Second Family of Methods}
\label{sec:2bapd-ex-x-im-y}
We then focus on the second class of primal-dual splitting methods that apply semi-implicit and implicit discretizations to $x$ and $y$, separately, and also consider different discretizations for $\lambda$ as before.

To do this, let us start from the following scheme
\begin{subnumcases}{}
	\label{eq:2bapd-ex-x-im-y-lbar-x}
	\frac{x_{k+1}-x_k}{\alpha_k}=v_{k} - x_{k+1},\\
	\label{eq:2bapd-ex-x-im-y-lbar-v}
	\gamma_{k} \frac{v_{k+1}-v_k}{\alpha_k} \in{}\mu_f(x_{k+1}-v_{k+1})- \partial_x\mathcal L(x_{k+1},y_{k+1},\barlk),
	\\
	\theta_{k} \frac{\lambda_{k+1}-\lambda_k}{\alpha_k} = {}\nabla_\lambda \mathcal L(v_{k+1} ,w_{k+1} ,\lambda_{k+1}),		\label{eq:2bapd-ex-x-im-y-lbar-l}\\
	\beta_{k}\frac{w_{k+1}-w_k}{\alpha_k} \in {}\mu_g(y_{k+1}-w_{k+1})- \partial_y\mathcal L(x_{k+1},y_{k+1},\barlk),
	\label{eq:2bapd-ex-x-im-y-lbar-w}\\
	\label{eq:2bapd-ex-x-im-y-lbar-y}
	\frac{y_{k+1}-y_k}{\alpha_k}=w_{k+1} - y_{k+1},
\end{subnumcases}
where $\barlk$ is to be determined and the parameter system 
\cref{eq:2bapd-scaling} is still discretized by \cref{eq:2bapd-tk}. 
Note that $x_{k+1}$ is calculated easily from \eqref{eq:2bapd-ex-x-im-y-lbar-x}, and to update $v_{k+1}$ via \eqref{eq:2bapd-ex-x-im-y-lbar-v}, one has to compute the subgradient $p_{k+1}\in\partial f(x_{k+1})$. However, as a convex combination of $x_k$ and $v_k$, $x_{k+1}$ might be outside the constraint set $\mathcal X$ since \eqref{eq:2bapd-ex-x-im-y-lbar-v} cannot promise $\{v_k\}_{k=1}^\infty\subset\mathcal X$. 

To avoid this, we apply implicit discretization to $\partial f$. Or more generally, we consider the composite case $f = f_1+f_2$ with $f_1\in\mathcal S_{\mu_f,L_f}^{1,1}(\mathcal X)$ and $f_2\in\mathcal S_0^0(\mathcal X)$. Therefore, in this section, we impose the following assumption.
\begin{assum}
	\label{assum:ex-x-im-y}
	$g\in\mathcal S_{\mu_g}^0(\mathcal Y)$ with $\mu_g\geq 0$ and 
	$f = f_1+f_2$ where $f_2\in\mathcal S_0^0(\mathcal X)$ and $f_1\in\mathcal S_{\mu_f,L_f}^{1,1}(\mathcal X)$ with $0\leq \mu_f\leq L_f<\infty$.
\end{assum}

To utilize the separable structure of $f$, we adopt the operator splitting technique and to promise the contraction of the Lyapunov function $\mathcal E_k$, we borrow the correction idea from \cite[Section 7.3]{luo_differential_2021} and propose the following modified scheme
\begin{subnumcases}{	\label{eq:2bapd-ex-x-im-y-lbar-correc}}
	\label{eq:2bapd-ex-x-im-y-lbar-u-correc}
	\frac{u_{k}-x_k}{\alpha_k}=v_{k} - u_{k},\\
	\label{eq:2bapd-ex-x-im-y-lbar-v-correc}
	\gamma_{k} \frac{v_{k+1}-v_k}{\alpha_k} \in{}\mu_f(u_{k}-v_{k+1})- \mathcal G_x(u_k,v_{k+1},\barlk),
	\\
	\label{eq:2bapd-ex-x-im-y-lbar-x-correc}
	\frac{x_{k+1}-x_k}{\alpha_k}=v_{k+1} - x_{k+1},\\
	\theta_{k} \frac{\lambda_{k+1}-\lambda_k}{\alpha_k} = {}\nabla_\lambda \mathcal L(v_{k+1} ,w_{k+1} ,\lambda_{k+1}),	\label{eq:2bapd-ex-x-im-y-lbar-l-correc}\\	
	\beta_{k}\frac{w_{k+1}-w_k}{\alpha_k} \in {}\mu_g(y_{k+1}-w_{k+1})- \partial_y\mathcal L(x_{k+1},y_{k+1},\barlk),
	\label{eq:2bapd-ex-x-im-y-lbar-w-correc}\\
	\frac{y_{k+1}-y_k}{\alpha_k}=w_{k+1} - y_{k+1},
	\label{eq:2bapd-ex-x-im-y-lbar-y-correc}
\end{subnumcases}
where $\mathcal G_x(u_k,v_{k+1},\barlk)=\nabla f_1(u_k)+\partial f_2(v_{k+1})+A^{\top}\barlk+N_{\mathcal X}(v_{k+1})$. Above, we replaced $x_{k+1}$ in \eqref{eq:2bapd-ex-x-im-y-lbar-x} and \eqref{eq:2bapd-ex-x-im-y-lbar-v} by $u_k$ and updated it by \eqref{eq:2bapd-ex-x-im-y-lbar-x-correc}, which is an extra correction step.
Similarly with \cref{eq:2bapd-im-x-im-y-lbar-arg}, we have an informal primal-dual formulation:
\begin{subnumcases}{	\label{eq:2bapd-ex-x-im-y-lbar-arg}}
	u_{k} = (x_k+\alpha_kv_{k})/(1+\alpha_k),
	\label{eq:2bapd-ex-x-im-y-lbar-u-arg}\\	
	v_{k+1}
	=\mathop{\argmin}\limits_{v\in\mathcal X}
	\left\{
	f_2(v)+\dual{\nabla f_1(u_k)+A^{\top}\barlk,v}
	+\frac{\widetilde{\eta}_{f,k}}{2\alpha_k}\nm{v-\widetilde{v}_k}^2\right\},
	\label{eq:2bapd-ex-x-im-y-lbar-v-arg}\\
	x_{k+1} = (x_k+\alpha_kv_{k+1})/(1+\alpha_k),
	\label{eq:2bapd-ex-x-im-y-lbar-x-arg}\\	
	y_{k+1}
	={}\mathop{\argmin}\limits_{y\in\mathcal Y}
	\left\{
	g(y) + \dual{By,\barlk}
	+\frac{\eta_{g,k}}{2\alpha^2_k}\nm{y-\widetilde{y}_k}^2\right\},
	\label{eq:2bapd-ex-x-im-y-lbar-y-arg}\\
	w_{k+1} = y_{k+1}+(y_{k+1}-y_k)/\alpha_k,\\
	\lambda_{k+1}
	=\lambda_k+\alpha_k/\theta_{k}\left(Av_{k+1}+Bw_{k+1}-b\right),
	\label{eq:2bapd-ex-x-im-y-lbar-l-argmax}
\end{subnumcases}
where $(\widetilde{y}_k,\eta_{g,k})$ are the same as that in \cref{eq:2bapd-im-x-im-y-lbar-arg} and 
\[
\widetilde{v}_k =\frac{1}{\widetilde{\eta}_{f,k}}(\gamma_kv_k+\mu_f\alpha_ku_k)\quad\text{with}\quad 	\widetilde{\eta}_{f,k}:=\gamma_k+\mu_f\alpha_k.
\]
To compute $\nabla f_1(u_k)$, the step \eqref{eq:2bapd-ex-x-im-y-lbar-v-arg} needs $u_k\in\mathcal X$. In view of \eqref{eq:2bapd-ex-x-im-y-lbar-u-arg}, this is true if $(x_k,v_k)\in\mathcal X\times \mathcal X$. Thanks to the correction  \eqref{eq:2bapd-ex-x-im-y-lbar-x-arg}, we conclude that $\{(x_{k},u_k,v_k)\}_{k=1}^{\infty}\subset\mathcal X\times\mathcal X$  as long as $(x_0,v_0)\in\mathcal X\times\mathcal X$.

Similarly with the previous section, different choices of $\barlk$ (cf. \cref{eq:2bapd-im-x-im-y-lbar-lv,eq:2bapd-im-x-im-y-lbar-lw,eq:2bapd-im-x-im-y-lbar-lk}) result in our second family of methods. One thing that we shall emphasis is, the first class of methods in \cref{sec:2bapd-im-x-im-y} require {\it no} correction step since both $x$ and $y$ are discretized implicitly. However, all the methods in this section consider semi-implicit discretization for $x$ and thus need proper correction (cf.\eqref{eq:2bapd-ex-x-im-y-lbar-x-correc}) to promise the contraction property of the discrete Lyapunov function \cref{eq:Ek-apd}. By symmetry, the second class of methods can be easily rewritten and applied to the case $F(x,y)= f(x)+\big(g_1(y)+g_2(y)\big)$. For simplicity, we omit the detailed presentations.
\subsection{The one-iteration estimate}
Analogously to \cref{lem:2bapd-im-x-im-y-one-step}, we establish the one-iteration analysis in \cref{lem:2bapd-ex-x-im-y-one-step}, which helps us prove the nonergodic rates of the second family of methods.
\begin{lem}\label{lem:2bapd-ex-x-im-y-one-step}
	Let $k$ be fixed. For the scheme \cref{eq:2bapd-ex-x-im-y-lbar-correc} with \cref{assum:ex-x-im-y} and $(x_k,v_k)\in\mathcal X\times\mathcal X$, we have $(u_k,x_{k+1},v_{k+1})\in\mathcal X\times\mathcal X\times\mathcal X$ and
	\begin{equation}\label{eq:2bapd-ex-x-im-y-one-step}
		\begin{aligned}
			\mathcal E_{k+1}-\mathcal E_{k}
			\leq &	-\alpha_k\mathcal E_{k+1}
			+
			\frac{L_f\alpha_k^2\theta_{k+1}-\gamma_{k}\theta_k}{2\theta_k}
			\nm{v_{k+1}-v_k}^2	
			\\
			{}&\qquad	
			+	\frac{\theta_{k}}{2}\nm{\lambda_{k+1}-\barlk}^2
			-	\frac{\beta_{k}}2\nm{w_{k+1}-w_k}^2	.
		\end{aligned}
	\end{equation}
\end{lem}
\begin{proof}
	As before, we calculate the difference $	\mathcal E_{k+1}-	\mathcal E_k=\mathbb I_1+\mathbb I_2+\mathbb I_3+\mathbb I_4$ with $\mathbb I_1,\mathbb I_2,\mathbb I_3$ and $\mathbb I_4$ being defined in \cref{eq:2bapd-im-x-im-y-I1-I4}.
	
	Expand the first term $\mathbb I_1$ as follows
	\begin{equation*}
		\begin{aligned}
			\mathbb I_1 = f(x_{k+1})-f(x_k)+g(y_{k+1})-g(y_k)
			+\dual{\lambda^*,A(x_{k+1}-x_k)+B(y_{k+1}-y_k)},
		\end{aligned}
	\end{equation*}
	and then duplicate the estimate \cref{eq:I2-2bapd-im-x-im-y}:
	\begin{equation*}
		\begin{aligned}
			\mathbb I_2	
			\leq {}&\alpha_{k}\dual{Av_{k+1}+Bw_{k+1}-b, \barlk-\lambda^*}
			+\frac{\theta_{k}}{2}\nm{\lambda_{k+1}-\barlk}^2	-\frac{\alpha_k\theta_{k+1}}{2}\nm{\lambda_{k+1}-\lambda^*}^2.
		\end{aligned}
	\end{equation*}
	In addition, we claim that the relation \cref{eq:I4-2bapd-im-x-im-y} holds true here:
	\[
	\begin{aligned}
		\mathbb I_4=&-
		\frac{\alpha_k \beta_{k+1}}{2}
		\nm{w_{k+1}-y^*}^2-	\frac{\beta_{k}}2\nm{w_{k+1}-w_k}^2
		-\frac{\mu_g\alpha_k}{2}\nm{y_{k+1}-w_{k+1}}^2
		\\
		{}&\quad+
		\frac{\mu_g\alpha_k}{2}
		\nm{y_{k+1}-y^*}^2
		-\alpha_k
		\dual{q_{k+1}, y_{k+1} -y^*}
		-\dual{q_{k+1}, y_{k+1} -y_k},
	\end{aligned}
	\]
	where $q_{k+1}\in\partial_y\mathcal L(x_{k+1},y_{k+1},\barlk)$ has been defined by \cref{eq:qk1-2bapd-im-x-im-y-lbar}. By \cref{eq:ineq-mu} and the fact $y_{k+1}\in\mathcal Y$ (cf. \eqref{eq:2bapd-ex-x-im-y-lbar-y-arg}), we obtain that
	\[
	\begin{aligned}
		{}&	\frac{\mu_g\alpha_k}{2}
		\nm{y_{k+1}-y^*}^2
		-\alpha_k
		\dual{q_{k+1}, y_{k+1} -y^*}
		-\dual{q_{k+1}, y_{k+1} -y_k}\\
		\leq 	{}&\alpha_k\left[g(y^*)-g(y_{k+1}) + \dual{\bar{\lambda}_{k+1},B(y^*-y_{k+1})}\right]
		+ g(y_k)-g(y_{k+1}) + \dual{\bar{\lambda}_{k+1},B(y_k-y_{k+1})}.
	\end{aligned}
	\]
	Dropping the negative square term $-\nm{y_{k+1}-w_{k+1}}^2$ and shifting $\barlk$ to $\lambda^*$, we get
	\[
	\begin{aligned}
		\mathbb I_4\leq {}& \alpha_k\big[g(y^*)-g(y_{k+1})
		+\dual{\lambda^*,B(y^*-y_{k+1})}\big]
		-	\frac{\alpha_k \beta_{k+1}}{2}
		\nm{w_{k+1}-y^*}^2	\\
		{}&\quad	-\frac{\beta_{k}}2\nm{w_{k+1}-w_k}^2
		+\alpha_k\dual{\barlk-\lambda^*,B(y^*-y_{k+1})}\\
		{}&\qquad+g(y_k)-g(y_{k+1})+\dual{\barlk,B(y_k-y_{k+1})}.
	\end{aligned}
	\]
	The estimate for $\mathbb I_3$ starts from \cref{eq:I3-2bapd-im-x-im-y-mid} but 
	is more subtle. We list the desired result below:
	\begin{equation}\label{eq:I3-2bapd-ex-x-im-y}
		\begin{aligned}
			\mathbb I_3
			\leq {}&	\alpha_k\big[f(x^*)-f(x_{k+1})+\dual{\lambda^*,A(x^*-x_{k+1})}\big]
			-\frac{\alpha_k\gamma_{k+1}}{2}
			\nm{v_{k+1}-x^*}^2	\\
			&\quad+f_1(x_k)-f_1(x_{k+1})-\alpha_k(f_2(v_{k+1})-f_2(x_{k+1}))	-	\frac{\gamma_{k}}2\nm{v_{k+1}-v_k}^2\\
			{}& \qquad+(1+\alpha_k)\left(f_1(x_{k+1})-f_1(u_k)\right)
			-\alpha_k\dual{\nabla f_1(u_k),v_{k+1}-v_k}\\
			{}&\quad \qquad+\alpha_k\dual{\barlk-\lambda^*,A(x^*-x_{k+1})}+\dual{\barlk,A(x_k-x_{k+1})}	.
		\end{aligned}
	\end{equation}
	The detailed proof can be found in \cref{app:ex-x-im-y-I3}. 
	Consequently, combining these estimates from $\mathbb I_1$ to $\mathbb I_4$ gives
	\begin{equation}\label{eq:diff-Ek-ex-x-im-y}
		\begin{aligned}
			\mathcal E_{k+1}-\mathcal E_{k}
			\leq &	-\alpha_k\mathcal E_{k+1}
			+(1+\alpha_k)f_2(x_{k+1})-f_2(x_{k})	-\alpha_kf_2(v_{k+1})
			\\
			{}&	\quad	+(1+\alpha_k)\left(f_1(x_{k+1})-f_1(u_{k})\right)	-\alpha_k\dual{\nabla f_1(u_k),v_{k+1} - v_k}\\
			{}&	\qquad	+	\frac{\theta_{k}}{2}\nm{\lambda_{k+1}-\barlk}^2
			-	\frac{\gamma_{k}}2\nm{v_{k+1}-v_k}^2
			-	\frac{\beta_{k}}2\nm{w_{k+1}-w_k}^2	.
		\end{aligned}
	\end{equation}
	
	Notice that by \eqref{eq:2bapd-ex-x-im-y-lbar-x-arg}, $x_{k+1}$ is a convex combination of $x_k$ and $v_{k+1}$ and 
	\begin{equation}\label{eq:est-f2-ex-x-im-y}
		(1+\alpha_k)f_2(x_{k+1})-f_2(x_{k})	-\alpha_kf_2(v_{k+1})\leq 0.
	\end{equation}
	By \cref{eq:def-L} and \cref{assum:ex-x-im-y}, it follows immediately that
	\[
	f_1(x_{k+1})-f_1(u_k)
	\leq {}\dual{\nabla f_1(u_k),x_{k+1}-u_k}+\frac{L_f}{2}\nm{x_{k+1}-u_k}^2.
	\]
	Besides, by \eqref{eq:2bapd-ex-x-im-y-lbar-u-arg} and \eqref{eq:2bapd-ex-x-im-y-lbar-x-arg} we have 
	\begin{equation}\label{eq:xk1-uk}
		x_{k+1}-u_k = \alpha_k(v_{k+1}-v_k)/(1+\alpha_k),
	\end{equation}
	which implies 
	\begin{equation}\label{eq:est-ex-x-im-y-f1}
		\begin{aligned}
			{}&(1+\alpha_k)\left(f_1(x_{k+1})-f_1(u_{k})\right)	-\alpha_k\dual{\nabla f_1(u_k),v_{k+1} - v_k}
			\leq 
			\frac{L_f\alpha_k^2}{2+2\alpha_k}
			\nm{v_{k+1}-v_k}^2.
		\end{aligned}
	\end{equation}
	Therefore, plugging \cref{eq:est-f2-ex-x-im-y,eq:est-ex-x-im-y-f1} into \cref{eq:diff-Ek-ex-x-im-y} gives 
	\[
	\begin{aligned}
		\mathcal E_{k+1}-\mathcal E_{k}
		\leq &	-\alpha_k\mathcal E_{k+1}
		+	\frac{L_f\alpha_k^2-\gamma_{k}(1+\alpha_k)}{2+2\alpha_k}
		\nm{v_{k+1}-v_k}^2
		+	\frac{\theta_{k}}{2}\nm{\lambda_{k+1}-\barlk}^2
		-	\frac{\beta_{k}}2\nm{w_{k+1}-w_k}^2.
	\end{aligned}
	\]
	In view of the relation $\theta_{k} = \theta_{k+1}(1+\alpha_k)$, we obtain \cref{eq:2bapd-ex-x-im-y-one-step}
	and finish the proof.
\end{proof}

\subsection{The semi-implicit choice \cref{eq:2bapd-im-x-im-y-lbar-lv}}
By \cref{lem:2bapd-ex-x-im-y-one-step}, if $L_f\alpha_k^2\leq \gamma_{k}(1+\alpha_k)$, then $\barlk = \lambda_{k+1}$ leads to \cref{eq:2bapd-ex-x-im-y-lk1-contrac-Ek}. However, this does not give a splitting algorithm. Thus, as before, we consider other semi-implicit and explicit choices. 

Different from the first class of methods in which \cref{eq:2bapd-im-x-im-y-lbar-lv} and \cref{eq:2bapd-im-x-im-y-lbar-lw} are equivalent, the scheme \cref{eq:2bapd-ex-x-im-y-lbar-arg} loses this symmetric property. In this part, we consider the first one \cref{eq:2bapd-im-x-im-y-lbar-lv}, which gives
\begin{equation}\label{eq:2bapd-ex-x-im-y-lv}
	\left\{
	\begin{aligned}
		{}&		u_{k} ={} (x_k+\alpha_kv_{k})/(1+\alpha_k),\quad
		d_k=\nabla f_1(u_k)+A^{\top}\lambda_k,\\
		{}&		v_{k+1}
		\!=\!\mathop{\argmin}\limits_{v\in\mathcal X}\!\!
		\left\{
		f_2(v)+\dual{d_k,v}
		+\frac{\alpha_k}{2\theta_{k}}\nm{Av+Bw_{k}-b}^2
		+\frac{\widetilde{\eta}_{f,k}}{2\alpha_k}\nm{v-\widetilde{v}_k}^2\right\}\!,\\
		{}&		x_{k+1} ={} (x_k+\alpha_kv_{k+1})/(1+\alpha_k),\\
		{}&	\barlk = {}\lambda_k+
		\alpha_k/\theta_k
		\left(Av_{k+1}+Bw_{k}-b\right),\\		
		{}&		y_{k+1} = {}\prox_{\tau_kg}^{\mathcal Y}(\widetilde{y}_k-\tau_kB^{\top}\barlk),\quad \tau_k = \alpha_k^2/\eta_{g,k},\\
		{}&				w_{k+1} = {}y_{k+1}+(y_{k+1}-y_k)/\alpha_k,\\
		{}&	\lambda_{k+1}
		={}\lambda_k+\alpha_k/\theta_{k}\left(Av_{k+1}+Bw_{k+1}-b\right),
	\end{aligned}
	\right.
\end{equation}
where $(\widetilde{v}_k,\widetilde{y}_k,\widetilde{\eta}_{f,k},\eta_{g,k})$ are the same as that in \cref{eq:2bapd-ex-x-im-y-lbar-arg} and $(x_0,v_0)\in\mathcal X\times\mathcal X$. According to \cref{lem:2bapd-ex-x-im-y-one-step}, we have $\{(x_k,y_k)\}_{k=1}^{\infty}\subset \mathcal X\times \mathcal Y$, and the following result should appear natural.
\begin{thm}\label{thm:2bapd-ex-x-ex-y-lv-conv}
	If $\barlk$ is chosen from \cref{eq:2bapd-im-x-im-y-lbar-lv}, then \cref{eq:2bapd-ex-x-im-y-lbar-arg} reduces to \cref{eq:2bapd-ex-x-im-y-lv}. Besides, under the initial setting \cref{eq:b0-g0}, \cref{assum:ex-x-im-y} and the condition 
	\begin{equation}\label{eq:ak-ex-x-im-y-lk1}
		\big(L_f\beta_k\theta_{k}+\gamma_k\nm{B}^2\big)\alpha_k^2 = \gamma_{k}\beta_k\theta_k,
	\end{equation}
	we have $\{(x_k,y_k)\}_{k=1}^{\infty}\subset \mathcal X\times \mathcal Y$ and 
	\begin{equation}\label{eq:2bapd-ex-x-im-y-lk1-conv}
		\nm{Ax_k+By_k-b}\leq \theta_k\mathcal R_0,\quad
		|F(x_k,y_k)-F^*|\leq 	\theta_k\left(\mathcal E_0+\nm{\lambda^*}\mathcal R_0\right).
	\end{equation}
	Above, $\mathcal R_0$ is defined in \cref{thm:conv-2bapd-im-x-im-y-lv} and $\theta_k$ satisfies
	\begin{equation}
		\label{eq:2bapd-ex-x-im-y-lw-rate}
		\theta_k\lesssim
		\min\left\{
		\frac{\nm{B}}{\sqrt{\beta_0}k},\,\frac{\nm{B}^2}{\mu_g k^2}
		\right\}
		+ \min\left\{
		\frac{L_f}{\gamma_0k^2},\,\exp\left(-\frac{k}{4}\sqrt{\frac{\mu_f}{L_f}}\right)
		\right\},
	\end{equation}
	provided that $\gamma_0\beta_0\leq L_f\beta_0+\gamma_0\nm{B}^2$.
\end{thm}
\begin{proof}
	In view of \cref{lem:2bapd-ex-x-im-y-one-step} and the relation \cref{eq:lk1-barlk-im-x-im-y-lw}, it follows that
	\[
	\begin{aligned}
		\mathcal E_{k+1}-\mathcal E_{k}
		\leq &	-\alpha_k\mathcal E_{k+1}
		+		\frac{L_f\alpha_k^2\theta_{k+1}-\gamma_{k}\theta_k}{2\theta_k}
		\nm{v_{k+1}-v_k}^2	
		+	\frac{\alpha_k^2\nm{B}^2-\beta_{k}\theta_k}{2\theta_k}
		\nm{w_{k+1}-w_k}^2.
	\end{aligned}
	\]
	Thanks to the condition \cref{eq:ak-ex-x-im-y-lk1} and the fact $\theta_{k+1}\leq \theta_k$, the above two square terms can be dropped. This promises $\mathcal E_k\leq \theta_k\mathcal E_0$ and thus implies \cref{eq:2bapd-ex-x-im-y-lk1-conv}, by repeating the proof of 	\cref{eq:2bapd-im-x-im-y-lw-rate}. Then using \cref{lem:est-yk-case1,lem:est-yk-case2,lem:est-yk-case3}, the proof of the mixed-type estimate \cref{eq:2bapd-ex-x-im-y-lw-rate} is in line with that of  \cref{eq:2bapd-im-x-im-y-lk-rate}.
\end{proof}

\subsection{The semi-implicit choice \cref{eq:2bapd-im-x-im-y-lbar-lw}}
We then apply another one \cref{eq:2bapd-im-x-im-y-lbar-lw} to \cref{eq:2bapd-ex-x-im-y-lbar-arg} and obtain 
\begin{equation}\label{eq:2bapd-ex-x-im-y-lw}
	\left\{
	\begin{aligned}
		{}&			u_{k} ={} (x_k+\alpha_kv_{k})/(1+\alpha_k),\\
		{}&		\widehat{\lambda}_k={} \lambda_k-\theta_k^{-1}\left(Ax_k+By_k-b\right)+\alpha_k/\theta_kA(v_k-x_k),\\
		{}&y_{k+1}
		={}\mathop{\argmin}\limits_{y\in\mathcal Y}
		\left\{
		\mathcal L_{\sigma_k}(x_k,y,\widehat{\lambda}_k)
		+\frac{\eta_{g,k}}{2\alpha^2_k}\nm{y-\widetilde{y}_k}^2\right\},\quad \sigma_k = 1/\theta_{k+1},\\
		{}&w_{k+1} = {}y_{k+1}+(y_{k+1}-y_k)/\alpha_k,\\		
		{}&		\barlk = {}\lambda_k+
		\alpha_k/\theta_k
		\left(Av_{k}+Bw_{k+1}-b\right),\\		
		{}&	v_{k+1}
		={}\prox^{\mathcal X}_{s_kf_2}\left[\widetilde{v}_k-s_k(\nabla f_1(u_k)+A^{\top}\barlk)\right],\quad s_k = \alpha_k/\widetilde{\eta}_{f,k},\\
		{}&		x_{k+1} ={}(x_k+\alpha_kv_{k+1})/(1+\alpha_k),\\
		{}&		\lambda_{k+1}
		={}\lambda_k+\alpha_k/\theta_{k}\left(Av_{k+1}+Bw_{k+1}-b\right),
	\end{aligned}
	\right.
\end{equation}
where $(\widetilde{v}_k,\widetilde{y}_k,\widetilde{\eta}_{f,k},\eta_{g,k})$ are the same as that in \cref{eq:2bapd-ex-x-im-y-lbar-arg} and $(x_0,v_0)\in\mathcal X\times\mathcal X$. By \cref{eq:2bapd-im-x-im-y-lbar-lw} and the last equation of \cref{eq:2bapd-ex-x-im-y-lw}, we have 
\[
\lambda_{k+1}-\barlk=\alpha_k/\theta_kA(v_{k+1}-v_k).
\]
Plugging this into \cref{lem:2bapd-ex-x-im-y-one-step}, one finds that
\[
\mathcal E_{k+1}-\mathcal E_{k}
\leq 
-\alpha_k\mathcal E_{k+1}
+\frac{1}{2\theta_{k}}
\big((L_f\theta_{k+1}+\nm{A}^2)	\alpha_k^2
-	\gamma_{k}\theta_k\big)
\nm{v_{k+1}-v_k}^2.
\]
Thus under the condition \cref{eq:ak-2bapd-ex-x-im-y-lw}, the contraction follows easily. As the mixed-type estimate \cref{eq:2bapd-ex-x-im-y-lv-rate} of $\theta_k$ can be proved by using \cref{lem:est-yk-case2} and a similar argument as before, we conclude the following.
\begin{thm}\label{thm:2bapd-ex-x-im-y-lw-conv}
	If $\barlk$ is chosen from \cref{eq:2bapd-im-x-im-y-lbar-lw}, then \cref{eq:2bapd-ex-x-im-y-lbar-arg} becomes \cref{eq:2bapd-ex-x-im-y-lw}. 
	Under the initial setting \cref{eq:b0-g0}, \cref{assum:ex-x-im-y} and the condition 
	\begin{equation}\label{eq:ak-2bapd-ex-x-im-y-lw}
		(L_f\theta_k+\nm{A}^2)\alpha_k^2=\gamma_{k}	\theta_k,
	\end{equation}
	we have $\{(x_k,y_k)\}_{k=1}^{\infty}\subset \mathcal X\times \mathcal Y$ and $\mathcal E_k\leq \theta_k\mathcal E_0$. Moreover, if $\gamma_0\leq L_f+\nm{A}^2$, then the estimate \cref{eq:2bapd-ex-x-im-y-lk1-conv} holds true with
	\begin{equation}\label{eq:2bapd-ex-x-im-y-lv-rate}
		\theta_k\lesssim\min\left\{
		\frac{\nm{A}}{\sqrt{\gamma_0}k}+\frac{L_f}{\gamma_0k^2},\quad
		\frac{\nm{A}^2}{\mu_f k^2}
		+\exp\left(-\frac{k}{4}\sqrt{\frac{\mu_f}{L_f}}\right)
		\right\}.
	\end{equation}
\end{thm}

\subsection{The explicit choice \cref{eq:2bapd-im-x-im-y-lbar-lk}}
To the end, we adopt the explicit one \cref{eq:2bapd-im-x-im-y-lbar-lk} and obtain
\begin{equation}\label{eq:2bapd-ex-x-im-y-lk}
	\left\{
	\begin{aligned}
		{}&			u_{k} = {}(x_k+\alpha_kv_{k})/(1+\alpha_k),\\
		{}&		\barlk	={}\lambda_k+\alpha_k/\theta_k(Av_k+Bw_{k}-b),\\
		{}&		v_{k+1}
		={}\prox^{\mathcal X}_{s_kf_2}\left[\widetilde{v}_k-s_k(\nabla f_1(u_k)+A^{\top}\barlk)\right],\quad s_k = \alpha_k/\widetilde{\eta}_{f,k},
		\\
		{}&		x_{k+1} ={} (x_k+\alpha_kv_{k+1})/(1+\alpha_k),\\
		{}&		y_{k+1} = {}\prox_{\tau_kg}^{\mathcal Y}(\widetilde{y}_k-\tau_kB^{\top}\barlk),\quad \tau_k = \alpha_k^2/\eta_{g,k},\\
		{}&		w_{k+1} ={} y_{k+1}+(y_{k+1}-y_k)/\alpha_k,\\
		{}&		\lambda_{k+1}
		={}\lambda_k+\alpha_k/\theta_{k}\left(Av_{k+1}+Bw_{k+1}-b\right),
	\end{aligned}
	\right.
\end{equation}
where $(\widetilde{v}_k,\widetilde{y}_k,\widetilde{\eta}_{f,k},\eta_{g,k})$ are the same as that in \cref{eq:2bapd-ex-x-im-y-lbar-arg} and $(x_0,v_0)\in\mathcal X\times\mathcal X$.

By \cref{eq:2bapd-im-x-im-y-lbar-lk} and the last equation of \cref{eq:2bapd-ex-x-im-y-lk}, we see that \cref{eq:diff-lk-ex} still holds true and invoking \cref{lem:2bapd-ex-x-im-y-one-step}, we obtain the estimate
\[
\begin{aligned}
	\mathcal E_{k+1}-\mathcal E_{k}
	\leq &	-\alpha_k\mathcal E_{k+1}
	+	\frac{2\alpha_k^2\nm{B}^2-\beta_{k}\theta_k}{2\theta_k}
	\nm{w_{k+1}-w_k}^2	\\
	{}&\quad		
	+\frac{1}{2\theta_{k}}
	\big((L_f\theta_{k+1}+2\nm{A}^2)	\alpha_k^2
	-	\gamma_{k}\theta_k\big)
	\nm{v_{k+1}-v_k}^2	.
\end{aligned}
\]
Hence, it is not hard to conclude the following result from this. By using \cref{lem:est-yk-case2,lem:est-yk-case3}, the proof of the mixed-type estimate \cref{eq:rate-2bapd-ex-x-ex-y-lk} is a little bit tedious but similar with the spirit of \cref{eq:2bapd-ex-x-im-y-lv-rate}.
\begin{thm}\label{thm:2bapd-ex-x-ex-y-lk-conv}
	Applying \cref{eq:2bapd-im-x-im-y-lbar-lk} to
	\cref{eq:2bapd-ex-x-im-y-lbar-arg} leads to \cref{eq:2bapd-ex-x-im-y-lk}. In addition, under the initial setting \cref{eq:b0-g0}, \cref{assum:ex-x-im-y} and the condition
	\[
	\big(L_f\beta_k\theta_{k}+2\beta_k\nm{A}^2+2\gamma_k\nm{B}^2\big)\alpha_k^2= 			\gamma_k\beta_k\theta_k,
	\]
	we 						have $\{(x_k,y_k)\}_{k=1}^{\infty}\subset \mathcal X\times \mathcal Y$ and $\mathcal E_k\leq \theta_k\mathcal E_0$. If $\gamma_0\beta_0\leq L_f\beta_0+2\beta_0\nm{A}^2+2\gamma_0\nm{B}^2$, then the estimate \cref{eq:2bapd-ex-x-im-y-lk1-conv} holds true with
	\begin{equation}
		\label{eq:rate-2bapd-ex-x-ex-y-lk}
		\theta_k\lesssim \min\left\{
		\frac{\nm{B}}{\sqrt{\beta_0}k},\frac{\nm{B}^2}{\mu_g k^2}
		\right\}
		+\min\left\{
		\frac{\nm{A}}{\sqrt{\gamma_0}k}+\frac{L_f}{\gamma_0k^2},\,
		\frac{\nm{A}^2}{\mu_f k^2}
		+\exp\left(-\frac{k}{4}\sqrt{\frac{\mu_f}{L_f}}\right)
		\right\}.
	\end{equation}
\end{thm}

\begin{rem}\label{rem:class3}
	Based on the discretization \cref{eq:2bapd-ex-x-im-y-lbar-correc}, one can further apply the operator splitting technique to $\partial_y\mathcal L(x,y,\lambda)$	and replace \eqref{eq:2bapd-ex-x-im-y-lbar-w-correc} and \eqref{eq:2bapd-ex-x-im-y-lbar-y-correc} by that
	\[
	\left\{
	\begin{aligned}
		{}&			\frac{z_{k}-y_k}{\alpha_k}={}w_{k} - z_{k},\\
		{}&			\beta_{k}\frac{w_{k+1}-w_k}{\alpha_k} \in {}\mu_g(z_{k}-w_{k+1})- \mathcal G_y(z_k,w_{k+1},\barlk),\\
		{}&			\frac{y_{k+1}-y_k}{\alpha_k}={}w_{k+1} - y_{k+1},
	\end{aligned}
	\right.
	\]
	where $\mathcal G_x(u_k,v_{k+1},\barlk)$ is the same as that in \cref{eq:2bapd-ex-x-im-y-lbar-correc} and 
	\[
	\mathcal G_y(z_k,w_{k+1},\barlk)=\nabla g_1(z_k)+\partial g_2(w_{k+1})+B^{\top}\barlk+N_{\mathcal Y}(w_{k+1}).
	\]
	This yields the third family of methods by considering different choices of $\barlk$ (cf.\cref{eq:2bapd-im-x-im-y-lbar-lv,eq:2bapd-im-x-im-y-lbar-lw,eq:2bapd-im-x-im-y-lbar-lk}).
	
	Impose the following condition: 	
	\begin{assum}
		\label{assum:ex-x-ex-y}
		$f = f_1+f_2$ where $f_2\in\mathcal S_0^0(\mathcal X)$ and $f_1\in\mathcal S_{\mu_f,L_f}^{1,1}(\mathcal X)$ with $0\leq \mu_f\leq L_f<\infty$, and 		$g = g_1+g_2$ where $g_2\in\mathcal S_0^0(\mathcal Y)$ and $g_1\in\mathcal S_{\mu_g,L_g}^{1,1}(\mathcal Y)$ with $0\leq \mu_g\leq L_g<\infty$.
	\end{assum}
	\noindent Analogously to \cref{lem:2bapd-ex-x-im-y-one-step}, the one step analysis reads as follows
	\begin{equation*}
		\begin{split}
			\mathcal E_{k+1}-\mathcal E_{k}
			\leq &	-\alpha_k\mathcal E_{k+1}
			+		\frac{L_f\alpha_k^2\theta_{k+1}-\gamma_{k}\theta_k}{2\theta_k}
			\nm{v_{k+1}-v_k}^2	\\
			{}&\quad			+	\frac{\theta_{k}}{2}\nm{\lambda_{k+1}-\barlk}^2			
			+		\frac{L_g\alpha_k^2\theta_{k+1}-\beta_{k}\theta_k}{2\theta_k}
			\nm{w_{k+1}-w_k}^2.
		\end{split}
	\end{equation*}
	Then, nonergodic optimal mixed-type convergence rates can be established as well. For simplicity, we omit the detailed presentations of these methods and their proofs as well. 
\end{rem}

\section{Numerical Experiments}
\label{sec:num}
In this part, we investigate the practical performances of our methods on the least absolute deviation (LAD) regression and the support vector machine  (SVM), both of which admit the following form
\begin{equation}\label{eq:A-I}
	\min_{x\in\R^n,y\in\R^m} F(x,y): = f(x)+g(y)\quad\st Ax -y = 0.
\end{equation}
For all cases in the sequel, $f$ and $g$ are nonsmooth but have explicit proximal calculations. Hence, we focus only on the semi-implicit scheme \cref{eq:2bapd-im-x-im-y-lw-arg} (denoted by Semi-APD) and report the detailed comparisons with related algorithms: 
\begin{itemize}
	\item the standard linearized ADMM (LADMM) \cite[Algorithm 2]{shefi_rate_2014},
	\item the accelerated linearized ADMM (ALADMM) \cite[Algorithm 2]{Xu2017},
	\item the fast alternating minimization algorithm (Fast-AMA) \cite[Algorithm 9]{goldstein_fast_2014},
	\item the accelerated LADMM with nonergodic rate (ALADMM-NE) \cite[Algorithm 1]{Li2019},
	\item the new primal-dual (New-PD) algorithm \cite[Scheme (39)]{tran-dinh_non-stationary_2020},
	\item the Chambolle--Pock (CP) method \cite{chambolle_first-order_2011}.
\end{itemize}

All these methods (including our Semi-APD) linearize the augmented term and thus share the same proximal operations of $f$ and $g$ and the matrix-vector multiplications of $A$ and $A^\top$.  We mention that ALADMM-NE is designed only for convex problems and the convergence rate is $O(1/k)$.
Both New-PD and Fast-AMA require strong convexity and possess the fast rate $O(1/k^2)$. Our Semi-APD, ALADMM and CP enjoy the rates $O(1/k)$ and $O(1/k^2)$ respectively for convex and partially strongly convex objectives, but the latter two use ergodic sequences.

To measure the convergence behavior, we look at three relative errors:
\begin{itemize}
	\item the objective residual: $|F(x_k,y_k)-F^*|/|F(x_0,y_0)|$,
	\item the violation of feasibility: $\nm{Ax_k-y_k}/\nm{Ax_0-y_0}$,
	\item the composite objective residual: $(P(x_k)-P^*) /|P(x_0)|$,
\end{itemize}
where the composite objective is $P(x) = f(x) + g(Ax)$, and the minimal value $F^*=P^*$ is approximated by running LADMM with enough iterations. Moreover, for the LAD regression problem, we also illustrate the capability of each algorithm for maintaining the sparsity.



	\subsection{LAD regression}
	\label{sec:num-lasso}
	Consider the LAD regression problem
	\begin{equation}\label{eq:lasso}
		\min_{x\in\R^n}\,P(x): = f(x) + \nm{Ax-b}_1,
	\end{equation}
	where $A\in\R^{m\times n}$ and $b\in\R^m$ are given data with $m\ll n$, and $f$ is a regularization function. Clearly, problem \cref{eq:lasso} is equivalent to \cref{eq:A-I} with $g(y) = \nm{y-b}_1$. Here, we choose two types of regularizer:
	\begin{itemize}
		\item {\it Case 1}: $f(x)=\lambda\nm{x}_1$,
		\item {\it Case 2}: $f(x)=\lambda \nm{x}_1+\mu_f/2\nm{x}^2$,
	\end{itemize}
	where the regularization parameter is $\lambda = 2$ 
	and the strong convexity constant is $\mu_f = 0.1$. 
	Similarly with \cite{tran-dinh_non-stationary_2020}, we generate the matrix $A$ from the standard normal distribution and set $b = Ax^\#+e$, where $x^{\#}$ is a sparse vector and $e$ is a Gaussian noise with variance $\sigma^2 = 0.01$.
	
	
	\begin{figure}[H]
		\centering
		\begin{subfigure}[h]{0.35\textwidth}
			\centering
			\includegraphics[width=\textwidth]{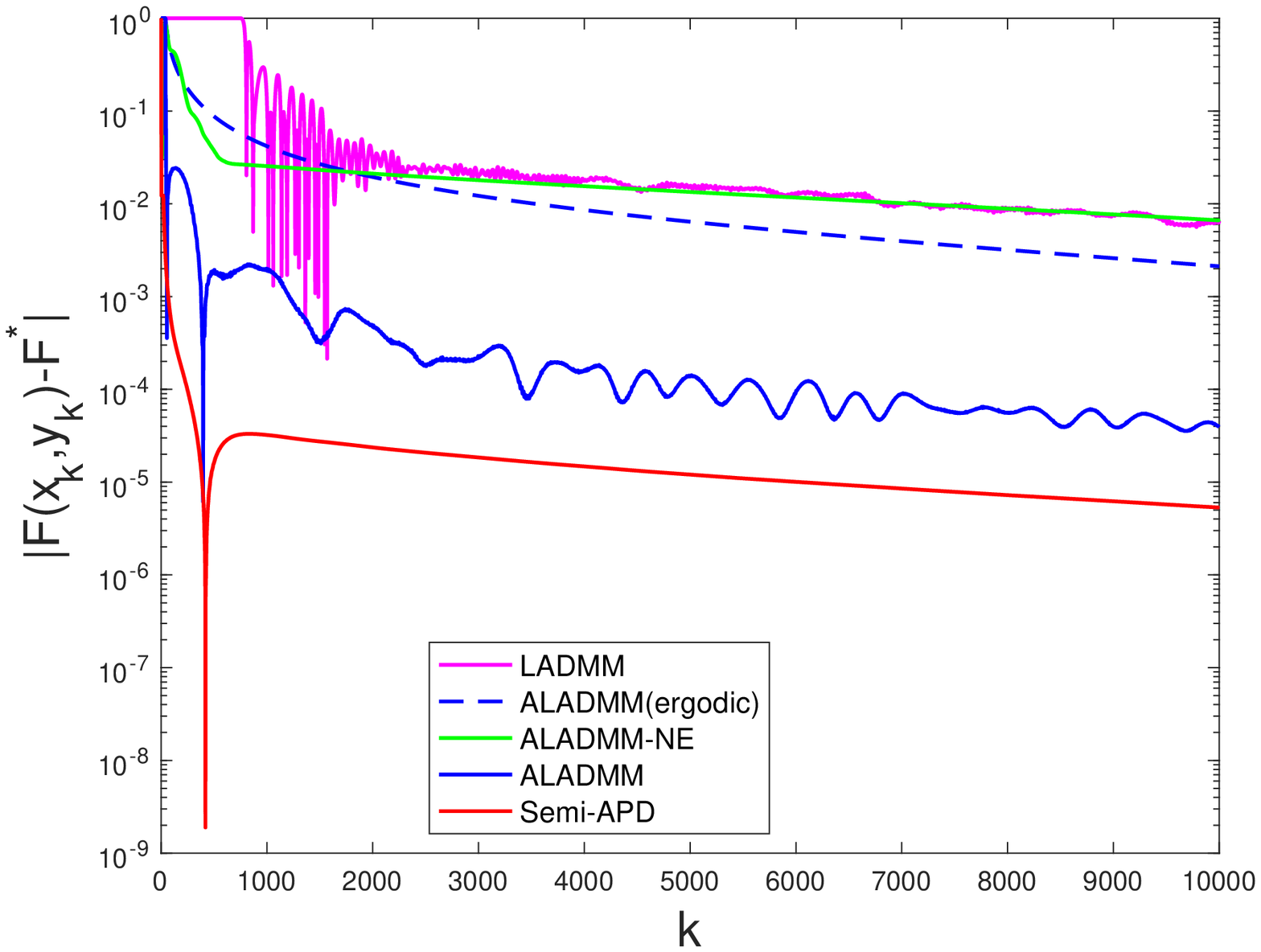}
		\end{subfigure}
		\begin{subfigure}[h]{0.35\textwidth}
			\centering
			\includegraphics[width=\textwidth]{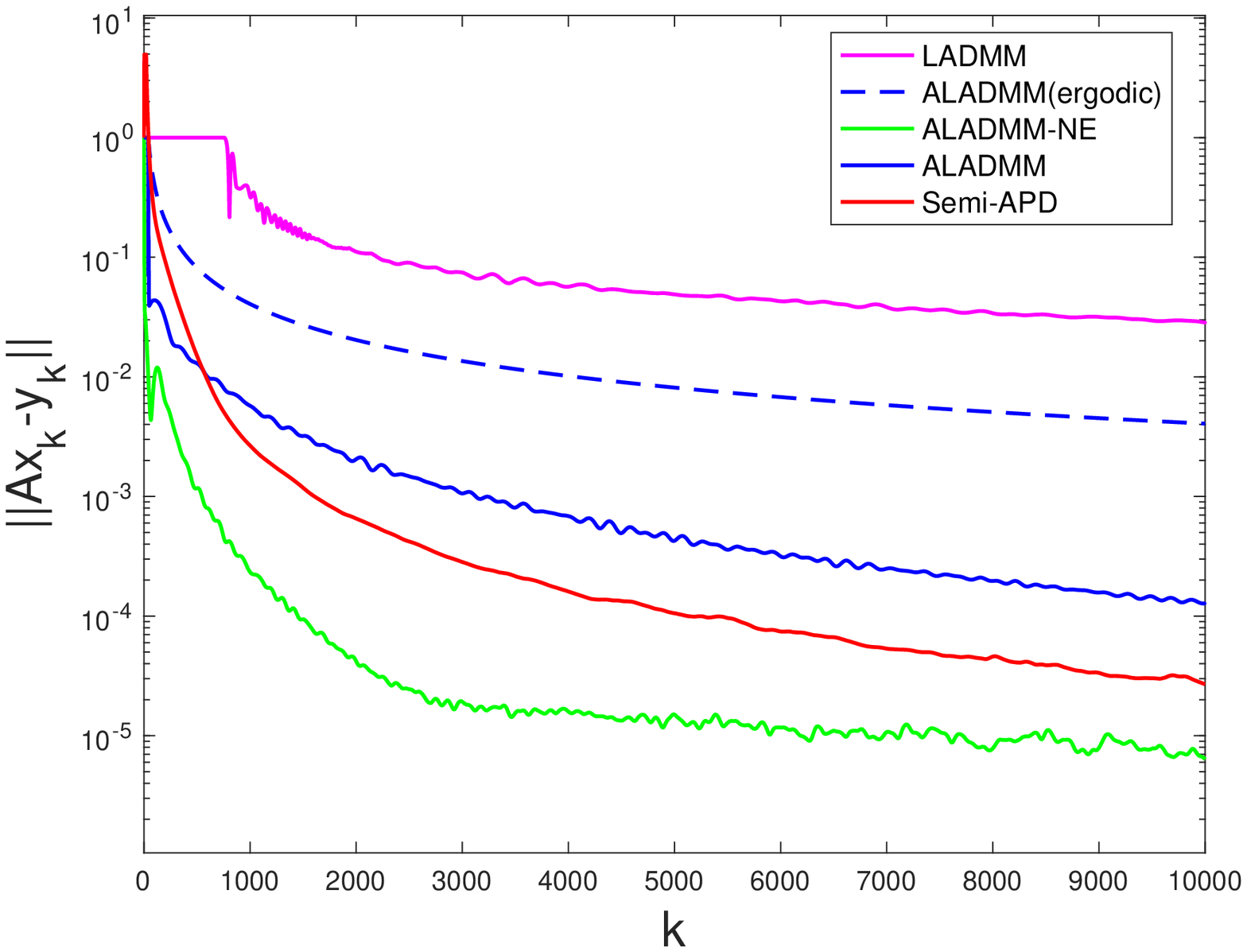}
		\end{subfigure}
		\begin{subfigure}[h]{0.35\textwidth}
			\centering
			\includegraphics[width=\textwidth]{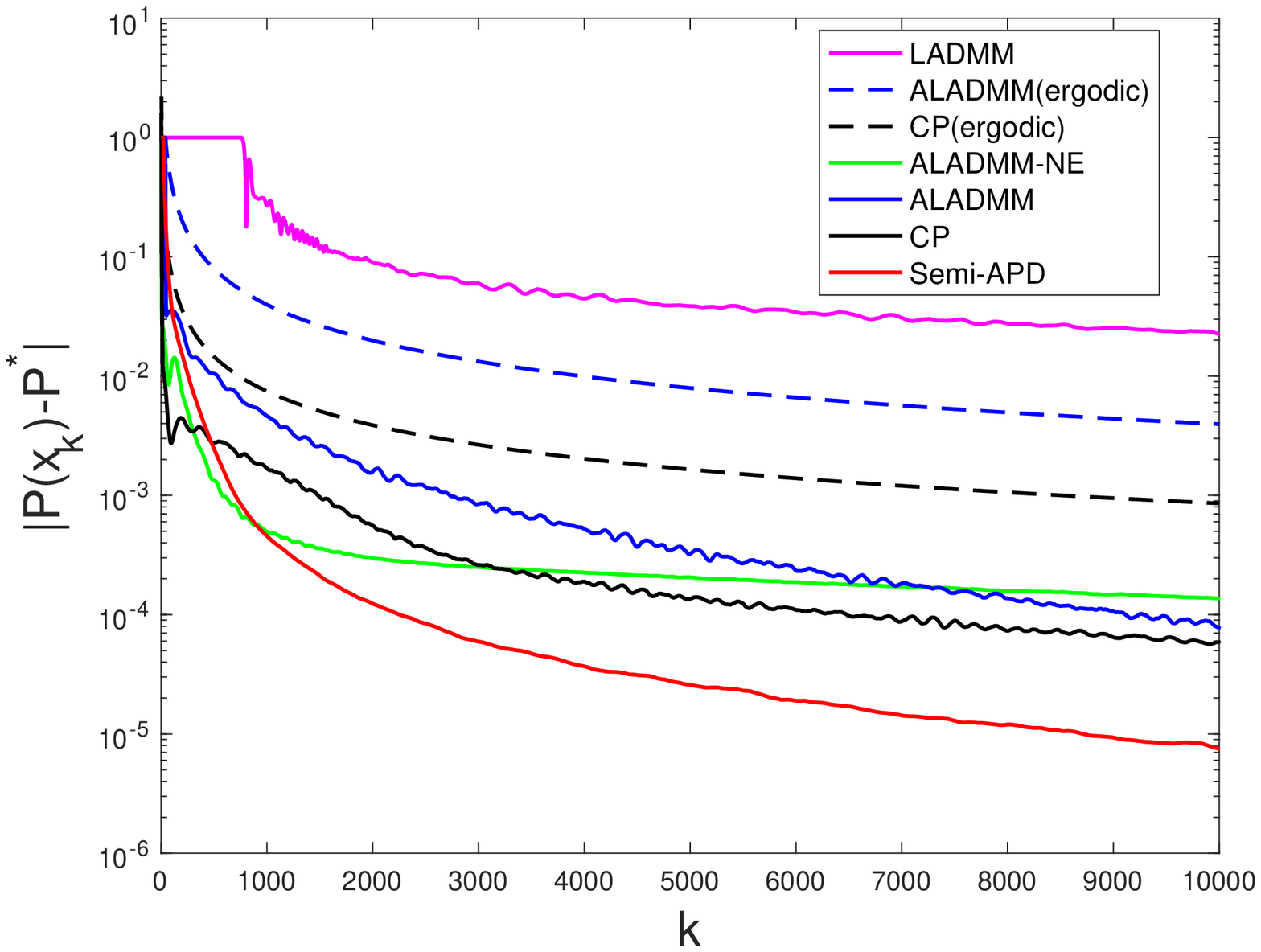}
		\end{subfigure}
		\begin{subfigure}[h]{0.35\textwidth}
			\centering
			\includegraphics[width=\textwidth]{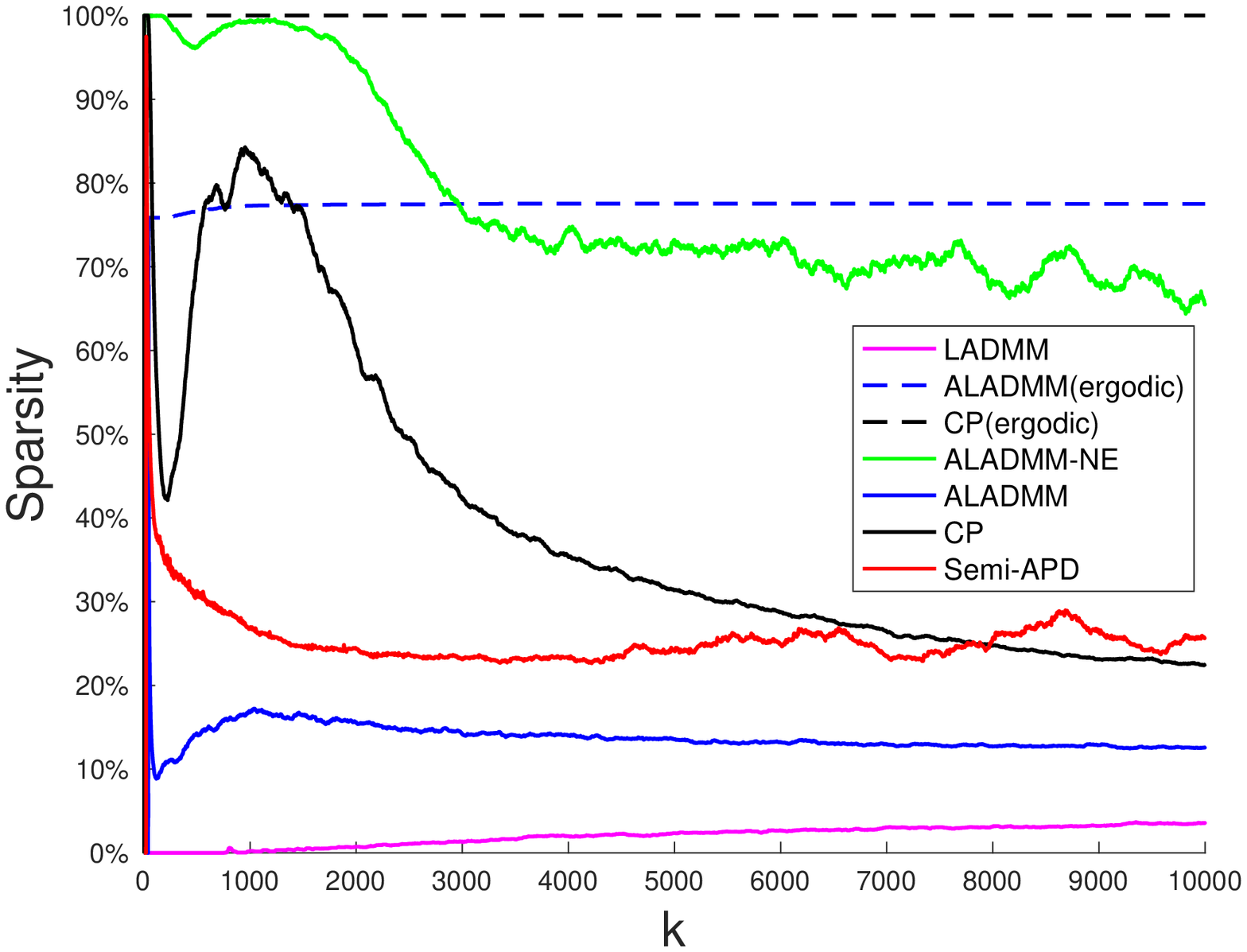}
		\end{subfigure}
		\caption{Numerical results of the LAD regression problem \cref{eq:lasso} under {\it Case 1}. The problem size is $(m,n) = (400,4000)$ and the sparse vector $x^{\#}$ has 10\% nonzero elements.}
		\label{fig-prob2}
	\end{figure}
	
	%
	%
	%
	%
	
	Numerical outputs of {\it Case 1} and {\it Case 2} are displayed
	respectively in \cref{fig-prob2,fig-prob3}. In {\it Case 1}, our Semi-APD performs the best for the objective residual $|F(x_k,y_k)-F^*|$ (top left) and the composite objective residual $\vert P(x_k)-P^*\vert $ (bottom left). For the violation of feasibility $\nm{Ax_k-y_k}$ (top right), however, Semi-APD is inferior to ALADMM-NE but still better than others. As the theoretical rates of ALADMM and CP are in ergodic sense, we also plot the errors in terms of the averaged sequences. It can be seen that ergodic convergence is much slower than that in nonergodic sense.
	
	In the bottom right part of \cref{fig-prob2}, we also report the sparsity of all iterative sequences. As we can see, except ALADMM-NE, all the methods maintain nice sparsity. More precisely, the standard LADMM provides a very sparse solution, and the sequences of the rest methods are dense in the beginning but become more sparse as the iteration step grows up. Besides, ergodic sequences perform not well because the average operation breaks the sparsity.
	

	For {\it Case 2}, the objective $f$ is strongly convex. From \cref{fig-prob3}, we observe that Semi-APD has fast convergence for the composite objective residual but is not competitive with New-PD for the objective residual and the violation of feasibility. However, New-PD requires three proximal calculations in each iteration and provides poor sparsity. As a contrast, our Semi-APD generates almost the same sparsity as ALADMM, Fast-AMA and LADMM. Again, ergodic sequences are inferior to those in nonergodic sense, for both convergence rate and sparsity.
	
	\begin{figure}[H]
		\centering
		\begin{subfigure}[h]{0.4\textwidth}
			\centering
			\includegraphics[width=\textwidth]{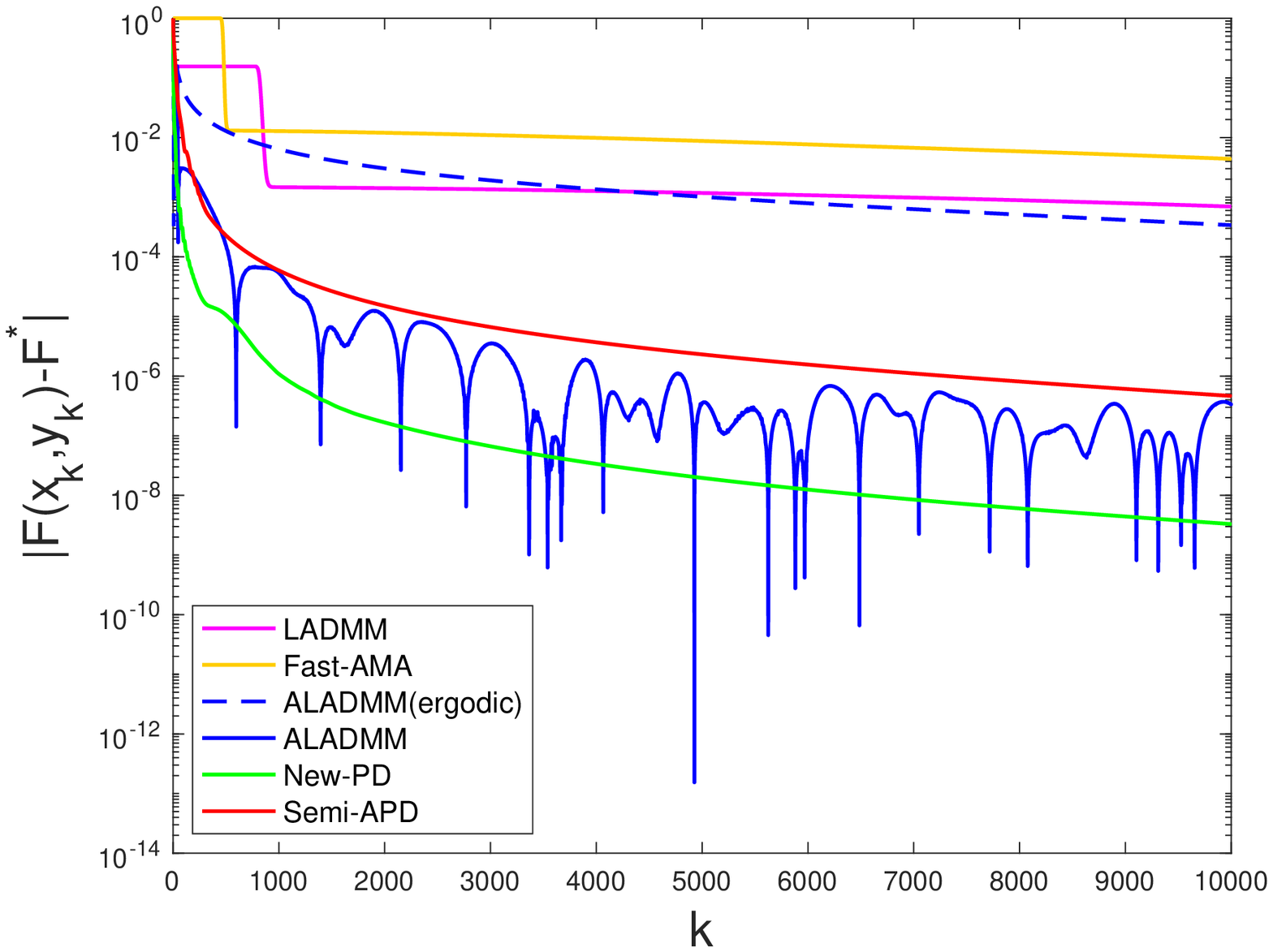}
		\end{subfigure}
		\begin{subfigure}[h]{0.4\textwidth}
			\centering
			\includegraphics[width=\textwidth]{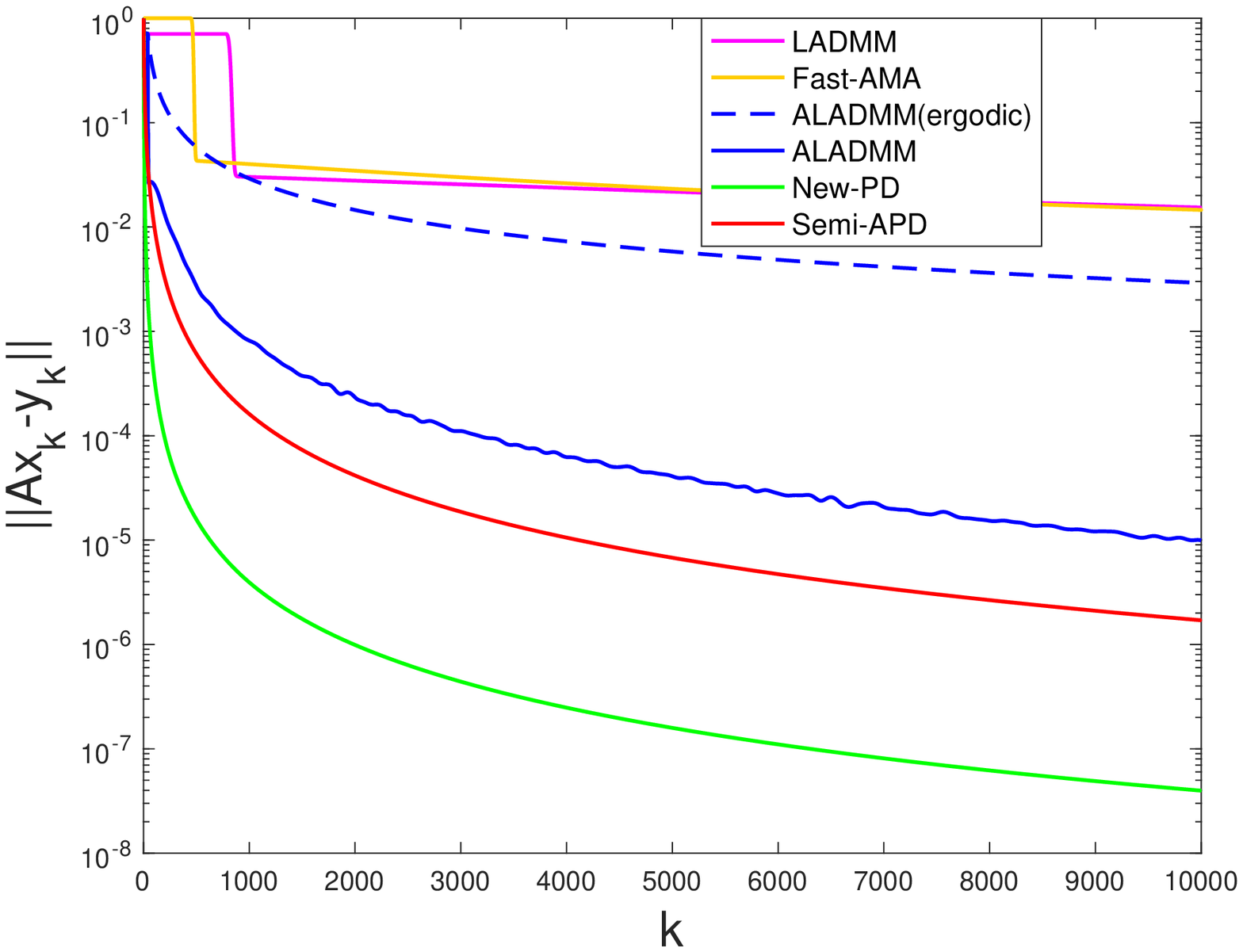}
		\end{subfigure}
		
		\begin{subfigure}[h]{0.4\textwidth}
			\centering
			\includegraphics[width=\textwidth]{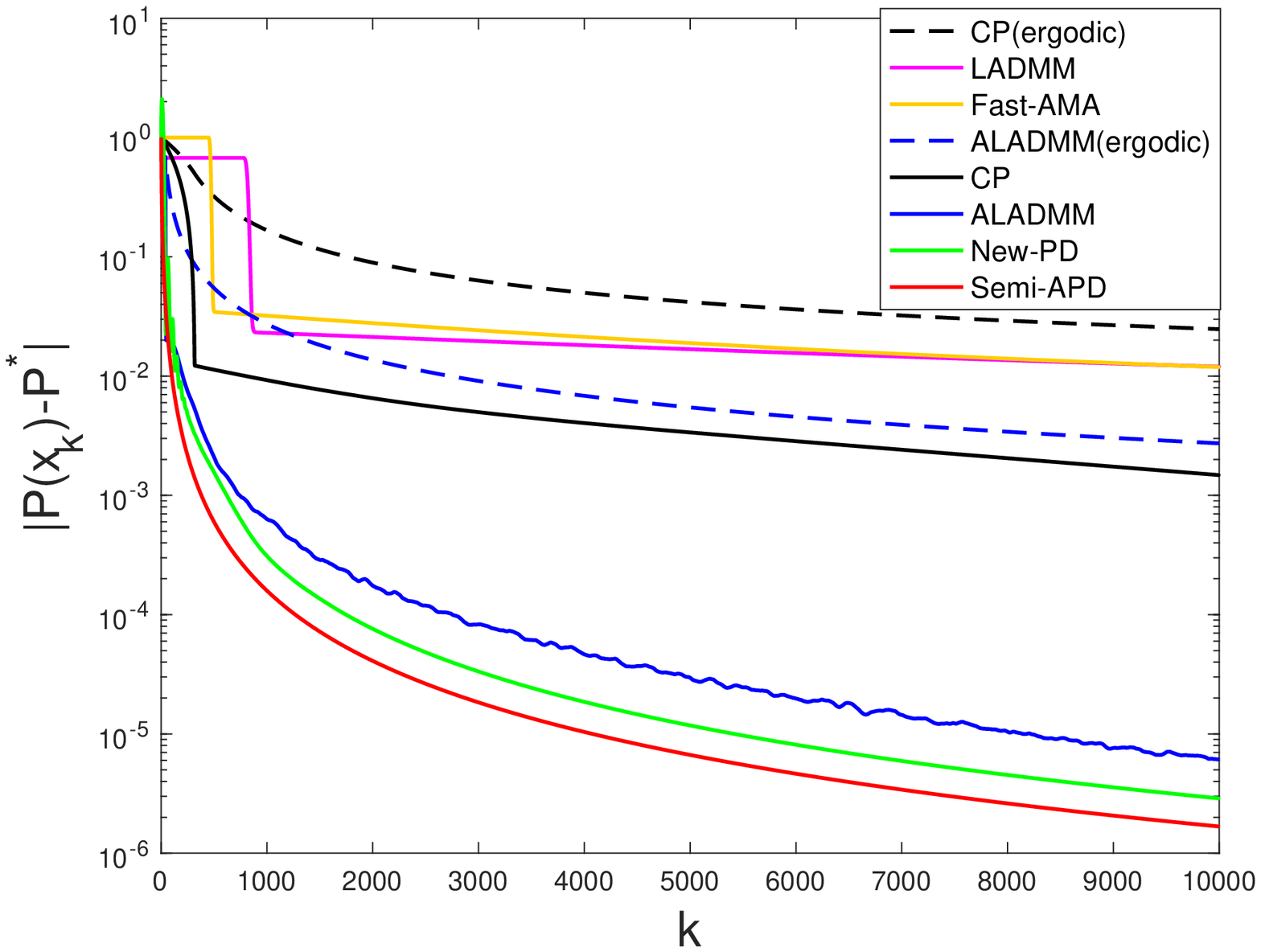}
		\end{subfigure}
		\begin{subfigure}[h]{0.4\textwidth}
			\centering
			\includegraphics[width=\textwidth]{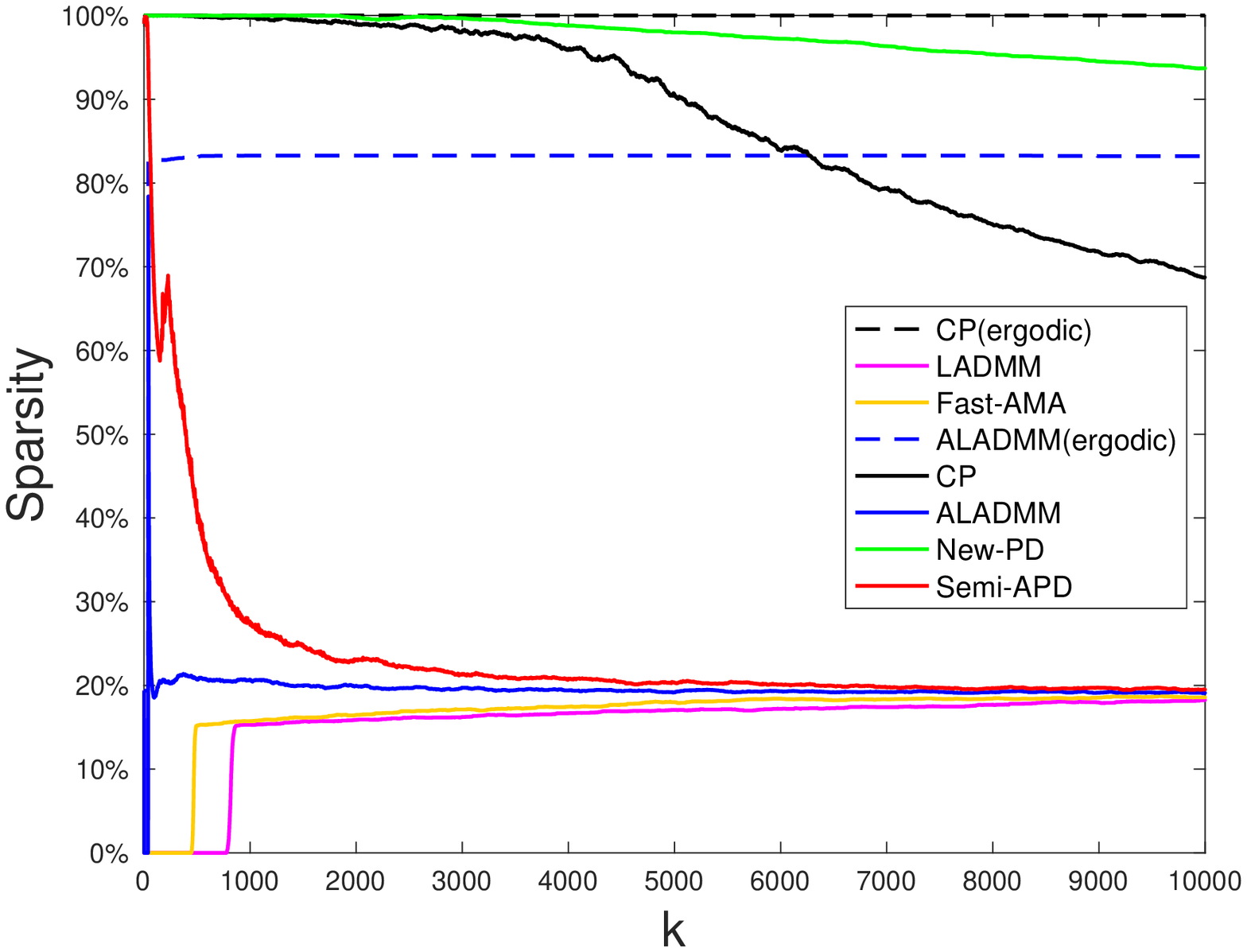}
		\end{subfigure}
		\caption{Numerical results of the LAD regression problem \cref{eq:lasso} under {\it Case 2}, with the same problem size and sparsity setting as {\it Case 1}.}
		\label{fig-prob3}
	\end{figure}

	
	\subsection{Support vector machine}
	\label{sec:num-svm}
	Given a matrix $W\in\R^{m\times n}$, the bias vector $b\in\R^m$ and the classify vector $c\in\R^m$, consider the SVM problem
	\begin{equation}\label{eq:svm}
		\min_{x\in\R^n}\,F(x): = g(x) + \frac{1}{m}\sum_{j= 1}^{m}\ell(c_j,w_j^\top x-b_j),
	\end{equation}
	where $w_j$ is the $j$-th column of $W,\,\ell(a,b) := \max(0,1-ab)$  is the Hinge loss function and $g:\R^n\to \R_+$ is a regularization function. We follow \cite{Xu2017} to generate the problem data and consider
	\begin{itemize}
		\item Binary linear SVM 
		: $g = \rho\nm{\cdot}_1$ with $\rho=0.2$,
		\item
		Elastic net regularized SVM 
		: $g = \rho_1/2\nm{\cdot}^2+\rho_2\nm{\cdot}_1$, with $\rho_1=0.05$ and $\rho_2=0.5$.
	\end{itemize}

		Numerical outputs of two SVM problems are plotted in \cref{fig-SVM-1,fig-SVM-2}. For both two cases, our Semi-APD outperforms others on the objective residual $|F(x_k,y_k)-F^*|$ and the composite objective residual $|P(x_k)-P^*|$.  In \cref{fig-SVM-1},  it provides the smallest violation of feasibility which is comparable with that of ALADMM. While in \cref{fig-SVM-2}, ALADMM is superior than our Semi-APD for the violation of feasibility. In addition, except the objective residual in \cref{fig-SVM-1}, the ergodic sequences of CP and ALADMM  provide slow convergence.
		
		\begin{figure}[H]
			\centering
			\begin{subfigure}[h]{0.4\textwidth}
				\centering
				\includegraphics[width=\textwidth]{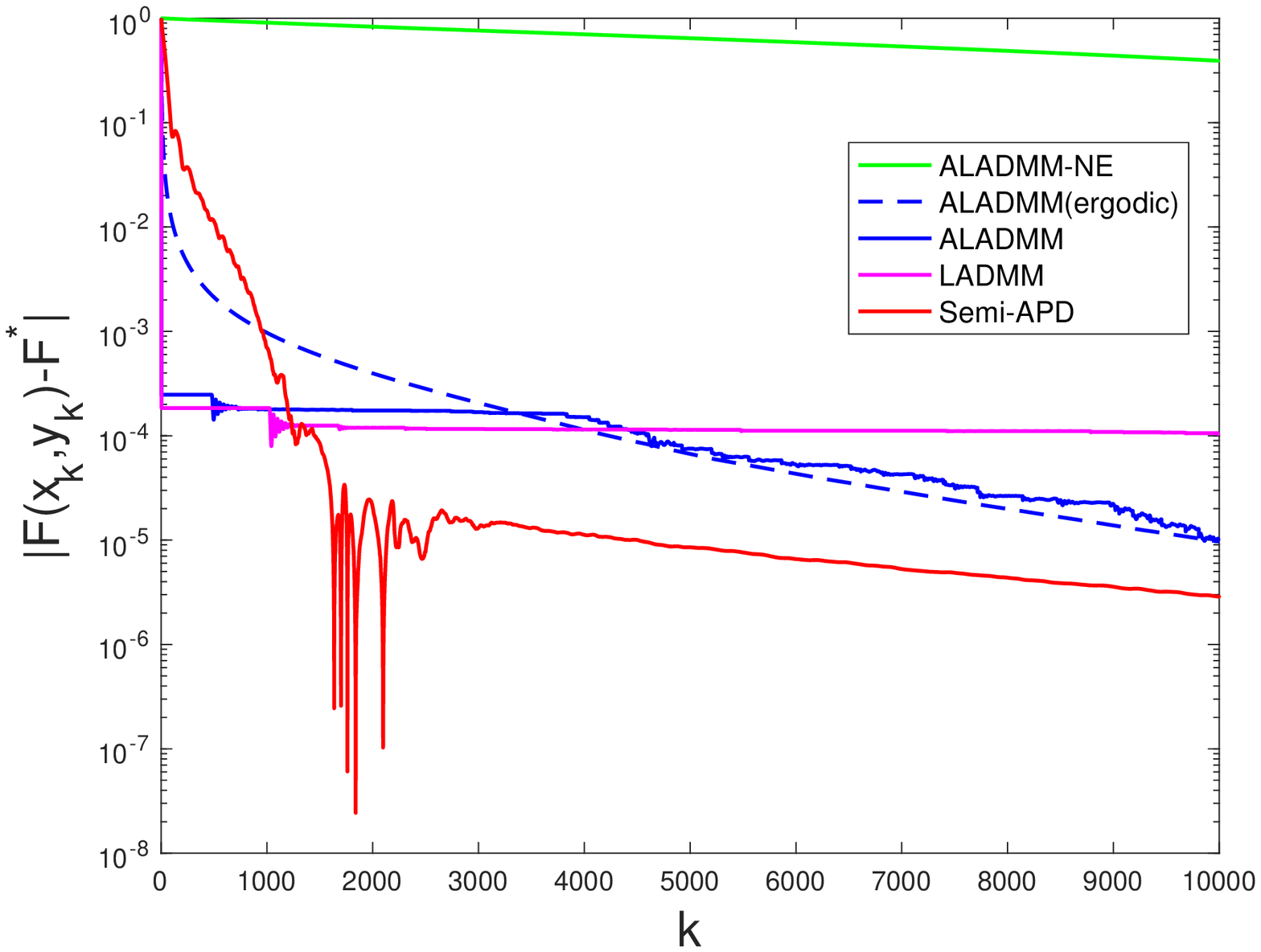}
			\end{subfigure}
			\begin{subfigure}[h]{0.4\textwidth}
				\centering
				\includegraphics[width=\textwidth]{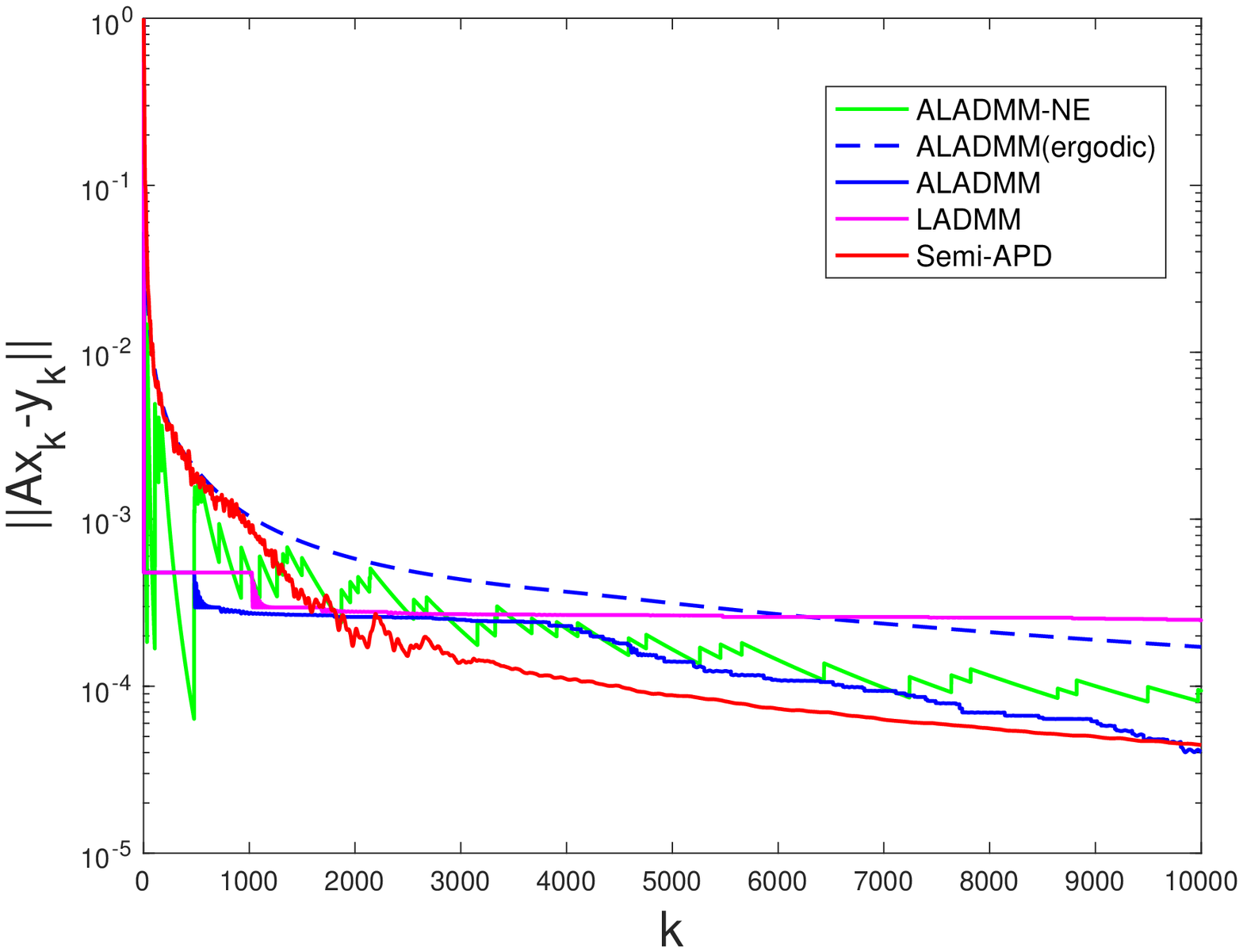}
			\end{subfigure}
			\begin{subfigure}[h]{0.4\textwidth}
				\centering
				\includegraphics[width=\textwidth]{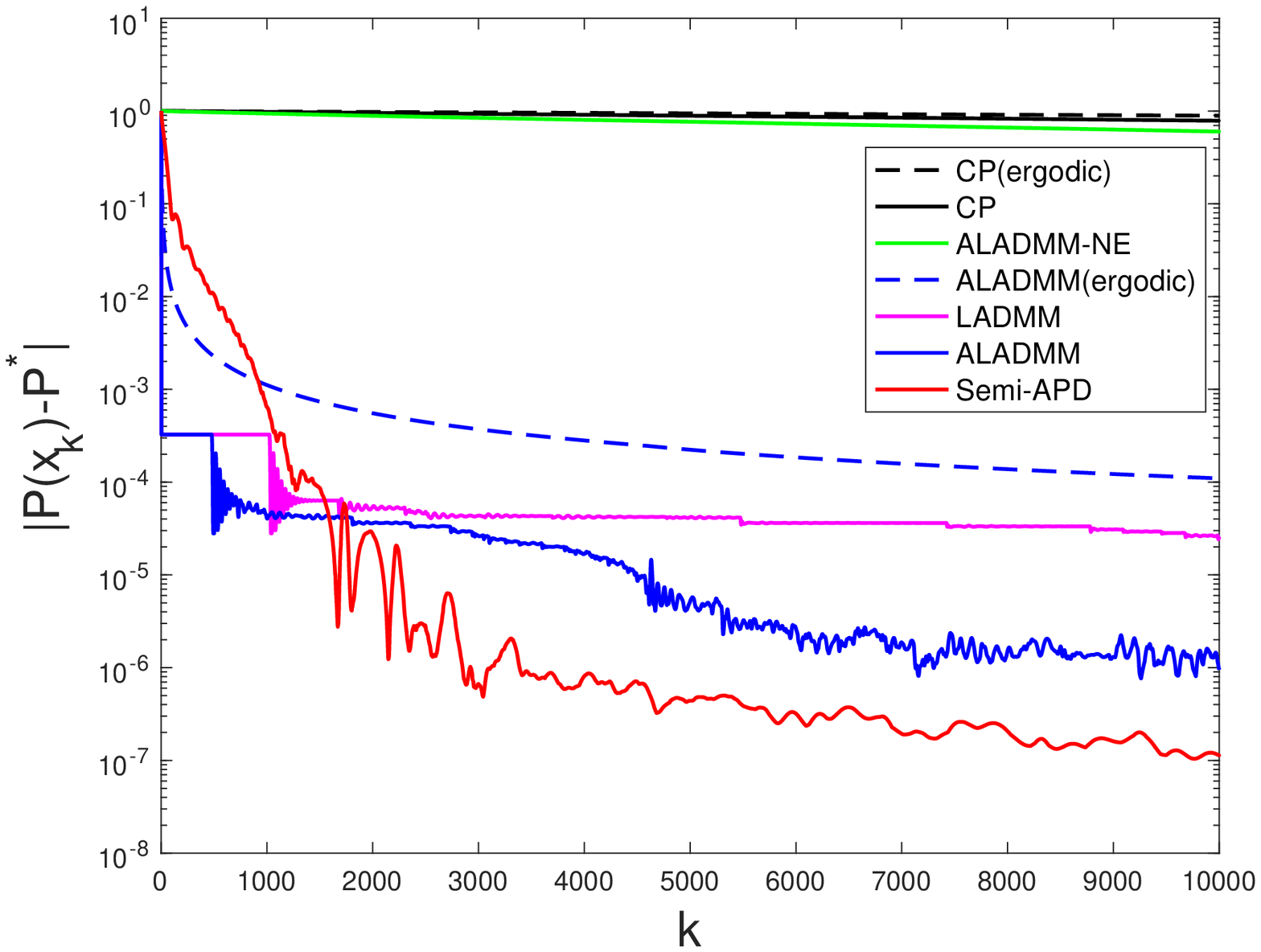}
			\end{subfigure}
			\caption{Numerical results of the binary linear SVM with $(m,n) = (100,500)$.}
			\label{fig-SVM-1}
		\end{figure}

		\begin{figure}[H]
			\centering
			\begin{subfigure}[h]{0.4\textwidth}
				\centering
				\includegraphics[width=\textwidth]{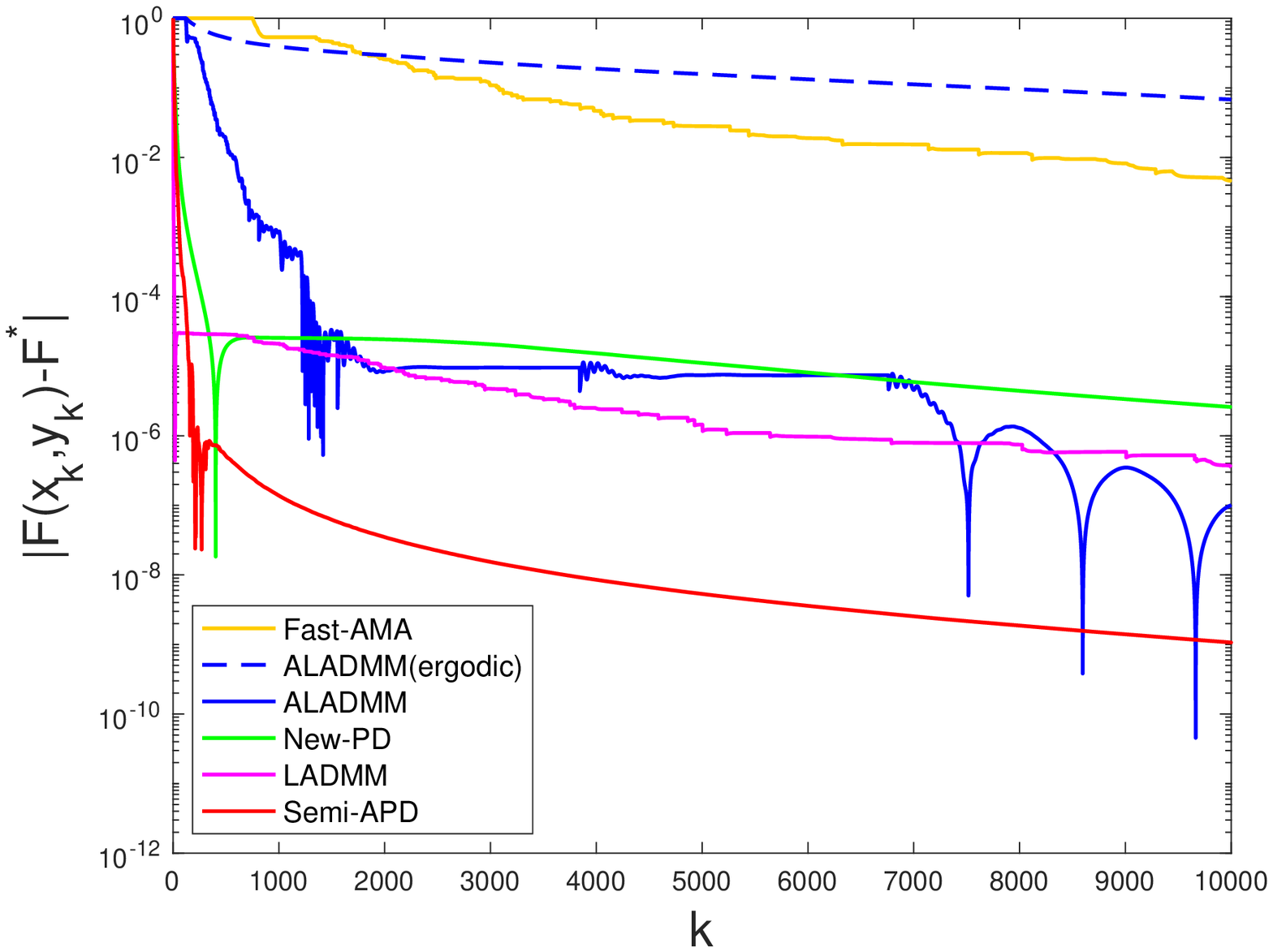}
			\end{subfigure}
			\begin{subfigure}[h]{0.4\textwidth}
				\centering
				\includegraphics[width=\textwidth]{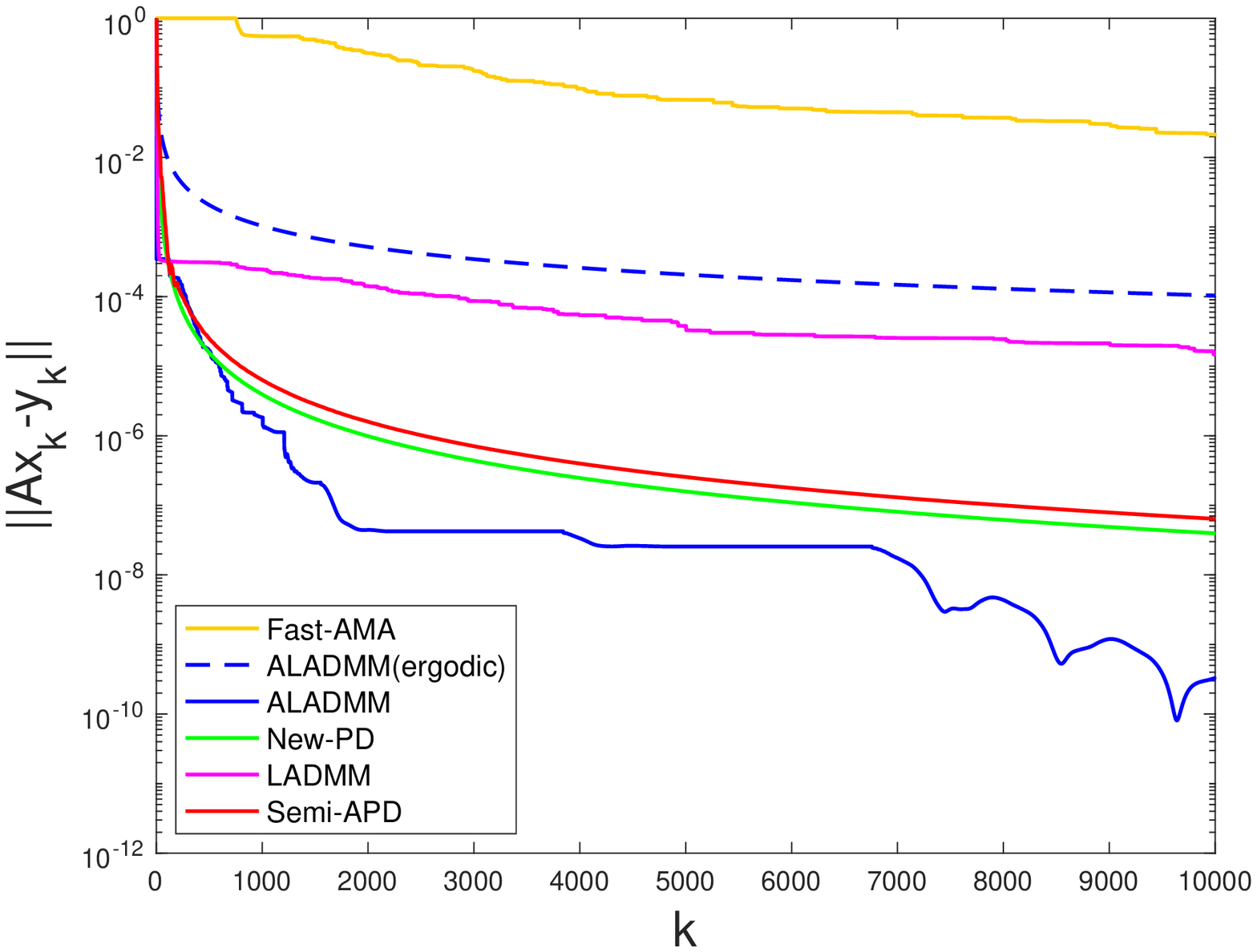}
			\end{subfigure}
			\begin{subfigure}[h]{0.4\textwidth}
				\centering
				\includegraphics[width=\textwidth]{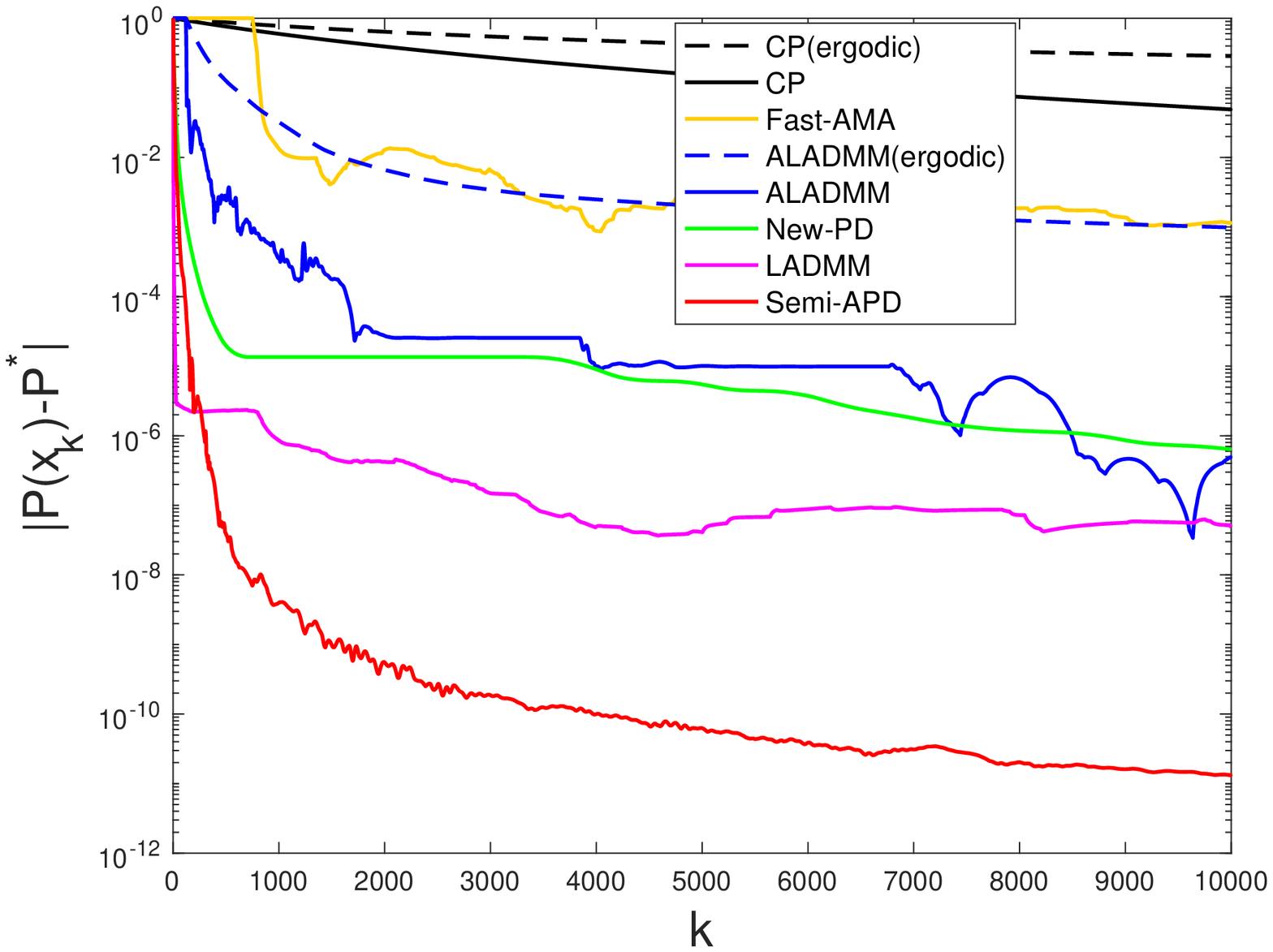}
			\end{subfigure}
			\caption{Numerical results of the elastic net regularized SVM with $(m,n) = (100,500)$.}
			\label{fig-SVM-2}
		\end{figure}

\section{Conclusions and Discussions}
\label{sec:conclu}
In this work, we present a self-contained differential equation solver approach 
for separable convex optimization problems. A novel dynamical system is introduced, and proper time discretizations lead to two families of primal-dual methods with acceleration, linearization and splitting. Besides, nonergodic optimal mixed-type convergence rates are established by a unified Lyapunov function. 

We also conduct some numerical experiments to validate the practical performances of the proposed method, regrading the objective residual, the feasibility violation and the capability for sparsity recovering. Although it does not always outperform existing algorithms on the convergence behavior, it maintains desired sparsity and does never work significantly worse. 

Below, we summarize some discussions and perspectives.
\subsection{Well-posedness of the nonsmooth case}
\label{sec:discuss-well}
To study the differential inclusion \cref{eq:2bapd-inclu}, the Moreau--Yosida approximation is an effective tool for solution existence (cf.\cite{luo_accelerated_2021}).
In \cite{attouch_fast_2021}, Attouch et al. established the existence of a global $C^1$ solution to a temporally rescaled inertial augmented Lagrangian system, which is a second order inclusion system for the convex separable problem \cref{eq:min-f-g-A-B}. However, as mentioned in \cite[Section 4]{attouch_fast_2021}, well-posedness under general nonsmooth setting deserves further study. 
\subsection{The implicit discretization of $\lambda$}
\label{sec:discuss-im}
Among the proposed algorithms in this work, we excluded the implicit choice $\barlk = \lambda_{k+1}$, which makes $x_{k+1}$ and $y_{k+1}$ coupled with each other. Let us take the scheme \cref{eq:2bapd-im-x-im-y-lbar-arg} as an example, which can be formulated by that
\[
\left\{
\begin{aligned}
	x_{k+1}
	={}&\prox_{s_kf}^{\mathcal X}\left(\widetilde{x}_k
	-s_kA^{\top}\lambda_{k+1}\right),\quad s_k = \alpha^2_{k}/\eta_{f,k},
	\\
	y_{k+1}
	={}&\prox_{\tau_kg}^{\mathcal Y}\left(\widetilde{y}_k
	-\tau_kB^{\top}\lambda_{k+1}\right),\quad \tau_k = \alpha^2_{k}/\eta_{g,k},
	\\
	\lambda_{k+1}={}&
	\widetilde{\lambda}_k+\theta_{k+1}^{-1}\left(Ax_{k+1}+By_{k+1}-b\right),\quad \widetilde{\lambda}_k =\lambda_{k}-\theta_{k}^{-1}(Ax_k+By_k-b).
\end{aligned}
\right.
\]
Eliminating $x_{k+1}$ and $y_{k+1}$ gives a nonlinear equation in terms of $\lambda_{k+1}$:
\begin{equation}
	\label{eq:2bapd-im-x-im-y-lbar-sub}
	\theta_{k+1}\lambda_{k+1}-A\prox_{s_kf}^{\mathcal X}\big(\widetilde{x}_k
	-s_kA^{\top}\lambda_{k+1}\big)
	-B\prox_{\tau_kg}^{\mathcal Y}\big(\widetilde{y}_k
	-\tau_kB^{\top}\lambda_{k+1}\big)
	=\theta_{k+1}\widetilde{\lambda}_k- b.
\end{equation}
In addition, applying $\barlk = \lambda_{k+1}$ to \cref{eq:2bapd-ex-x-im-y-lbar-correc}, we obtain the corresponding nonlinear equation that enjoys a similar structure with \cref{eq:2bapd-im-x-im-y-lbar-sub}. Following the spirit of \cite{li_asymptotically_2020,luo_acc_primal-dual_2021,luo_primal-dual_2022,niu_sparse_2021}, one can call the semi-smooth Newton iteration \cite{facchinei_finite-dimensional_2003_vol2} to solve \cref{eq:2bapd-im-x-im-y-lbar-sub} efficiently, provided that the problem itself has nice properties such as sparsity and semismoothness. 
\subsection{Successive choice of $\lambda$}
\label{sec:discuss-suc}
As announced in \cref{sec:2bapd-im-x-im-y}, we restricted ourselves to the same choice $\lambda=\barlk$ for both $\partial_x\mathcal L(x,y,\lambda)$ and $\partial_y\mathcal L(x,y,\lambda)$. Thus, unlike the original ADMM \cref{eq:ADMM} and existing accelerated ADMM \cite{ouyang_accelerated_2015,tran-dinh_primal-dual_2015,Xu2017}, our methods do not involve simultaneously the augmented terms of $x$ and $y$. This can be recovered if we adopt different choices for $\lambda$. For example, one can apply \cref{eq:2bapd-im-x-im-y-lbar-lv} and $\barlk =\lambda_{k+1}$ to \eqref{eq:2bapd-im-x-im-y-lbar-x-arg} and \eqref{eq:2bapd-im-x-im-y-lbar-y-arg}, respectively. However, this successive way brings more cross terms, and the one-iteration analysis (cf. Lemmas \ref{lem:2bapd-im-x-im-y-one-step} and \ref{lem:2bapd-ex-x-im-y-one-step}) deserves further study. 
\subsection{The multi-block case}
\label{sec:multi}
Consider the multi-block case:
\begin{equation}\label{eq:multi}
	F(x)=\sum_{i=1}^{M}f_i(x_i),\quad \sum_{i=1}^{M}A_ix_i=b,
	\quad x=(x_1,x_2,\cdots,x_M),\,M\geq 3.
\end{equation}
It has been showed in \cite{chen_direct_2016} that the direct extension of ADMM is not necessarily convergent unless each $f_i$ is strongly convex (cf. \cite{han_note_2012}). For general convex case, some variants have been proposed with provable convergence \cite{dai_sequential_2016,han_augmented_2014,he_alternating_2012} and the sublinear rate $O(1/k)$ \cite{bai_generalized_2018,he_splitting_2015,he_full_2015}.

We claim that the continuous model \cref{eq:2bapd-inclu} and \cref{thm:ASPD-exp} can be extended to the multi-block case \cref{eq:multi}. As for the discrete level, parallel type methods (cf.\cref{eq:2bapd-im-x-im-y-lk-arg,eq:2bapd-ex-x-im-y-lk}) are more likely to be generalized to this case but more efforts are needed to study the rest Gauss-Seidel type algorithms.

\section*{Acknowledgments}
The authors would like to thank Professor Jun Hu for his fruitful guidances and suggestions about the revision of an early version of the manuscript. 

\appendix
\section{Proof of the Estimate \cref{eq:I3-2bapd-ex-x-im-y}}
\label{app:ex-x-im-y-I3}
Recall the identity \cref{eq:I3-2bapd-im-x-im-y-mid}:
\begin{equation}\label{eq:I3-app}
	\begin{aligned}
		\mathbb I_3
		=&
		\frac{\alpha_k(\mu_f-\gamma_{k+1})}{2}
		\nm{v_{k+1}-x^*}^2-\frac{\gamma_{k}}2
		\nm{v_{k+1}-v_k}^2
		+		\gamma_{k}
		\dual{v_{k+1}-v_k, v_{k+1}-x^*}.
	\end{aligned}
\end{equation}
By \eqref{eq:2bapd-ex-x-im-y-lbar-v-correc}, we have $p_{k+1}	\in\partial f_2(v_{k+1})+N_{\mathcal X}(v_{k+1})$ where
\begin{equation*}
	p_{k+1} :={} \mu_f(u_{k}-v_{k+1})-\gamma_k\frac{v_{k+1}-v_k}{\alpha_k}-\nabla f_1(u_k)-A^{\top}\barlk.
\end{equation*}
Rewrite the last term in \cref{eq:I3-app} by that
\[
\begin{aligned}
	\gamma_{k}
	\dual{v_{k+1}-v_k, v_{k+1} -x^*}
	={}&	\mu_f\alpha_k\dual{u_{k}-v_{k+1}, v_{k+1} -x^*}
	-\alpha_k	\dual{p_{k+1}, v_{k+1} -x^*}\\
	{}&\quad-\alpha_k\dual{\nabla f_1(u_k)+A^{\top}\barlk,v_{k+1}-x^*}.
\end{aligned}
\]
Invoking \cref{eq:x-y-z}, the first cross term is estimate as follows
\[
\begin{aligned}
	\mu_f\alpha_k\dual{u_{k}-v_{k+1}, v_{k+1} -x^*}
	\leq  {}&\frac{\mu_f\alpha_k}{2}\left(
	\nm{u_{k}-x^*}^2-\nm{v_{k+1}-x^*}^2
	\right).
\end{aligned}
\]
For the second term, we have
\[
\begin{aligned}
	{}&	
	-\alpha_k	\dual{p_{k+1}, v_{k+1} -x^*}
	\leq  -\alpha_k(f_2(v_{k+1})-f_2(x^*))\\
	=& -\alpha_k(f_2(x_{k+1})-f_2(x^*))
	-\alpha_k(f_2(v_{k+1})-f_2(x_{k+1})),
\end{aligned}
\]
and summarizing the above results gives
\begin{equation}\label{eq:I3-app-mid}
	\begin{aligned}
		\mathbb I_3
		\leq  &-\alpha_k(f_2(x_{k+1})-f_2(x^*))
		-\frac{\alpha_k\gamma_{k+1}}{2}
		\nm{v_{k+1}-x^*}^2\\
		{}&\quad			 -\alpha_k(f_2(v_{k+1})-f_2(x_{k+1}))
		-	\frac{\gamma_{k}}2\nm{v_{k+1}-v_k}^2\\		
		{}&\qquad+\frac{\mu_f\alpha_k}{2}
		\nm{u_{k}-x^*}^2-\alpha_k\big\langle\nabla f_1(u_k)+A^{\top}\barlk,v_{k+1}-x^* \big\rangle.
	\end{aligned}
\end{equation}

Let us focus on the last term in \cref{eq:I3-app-mid}.
By \eqref{eq:2bapd-ex-x-im-y-lbar-u-correc}, it follows that
\[
\begin{aligned}
	{}&-\alpha_k\dual{\nabla f_1(u_k)+A^{\top}\barlk,v_{k+1}-x^*}\\
	=&-\alpha_k\dual{\nabla f_1(u_k)+A^{\top}\barlk,v_{k+1}-v_k}
	-\alpha_k\dual{\nabla f_1(u_k)+A^{\top}\barlk,u_{k}-x^*}\\
	{}&\quad-\dual{\nabla f_1(u_k)+A^{\top}\barlk,u_{k}-x_k}.
\end{aligned}
\]
Using \cref{eq:def-mu}, the fact $(x_k,u_k)\in\mathcal X\times \mathcal X$ and \cref{assum:ex-x-im-y}, we obtain
\[
\begin{aligned}
	{}&	-\alpha_k\dual{\nabla f_1(u_k)+A^{\top}\barlk,u_{k}-x^*}
	-\dual{\nabla f_1(u_k)+A^{\top}\barlk,u_{k}-x_k}\\
	\leq  {}&\alpha_k\left(f_1(x^*)-f_1(u_k)+\dual{\barlk,A(x^*-u_k)}\right)
	-\frac{\mu_f\alpha_k}{2}\nm{u_k-x^*}^2\\
	&\quad+f_1(x_k)-f_1(u_k)+\dual{\barlk,A(x_k-u_k)}.
\end{aligned}
\]
We then shift $u_k$ to $x_{k+1}$ to get 
\[
\begin{aligned}
	{}&\frac{\mu_f\alpha_k}{2}
	\nm{u_{k}-x^*}^2-\alpha_k\big\langle\nabla f_1(u_k)+A^{\top}\barlk,v_{k+1}-x^* \big\rangle\\
	\leq  {}&\alpha_k\left(f_1(x^*)-f_1(x_{k+1})+\dual{\barlk,A(x^*-x_{k+1})}\right)\\
	&\quad+f_1(x_k)-f_1(x_{k+1})+\dual{\barlk,A(x_k-x_{k+1})}\\
	{}&\quad+(1+\alpha_k)\left(f_1(x_{k+1})-f_1(u_k)\right)-\alpha_k\dual{\nabla f_1(u_k),v_{k+1}-v_k}\\
	{}& \quad+(1+\alpha_k)\dual{\barlk,A(x_{k+1}-u_k)}
	-\alpha_k\dual{A^{\top}\barlk,v_{k+1}-v_k}.
\end{aligned}
\]
Thanks to the relation \cref{eq:xk1-uk}, the last term vanishes. After switching $\barlk$ to $\lambda^*$, we plug the above estimates into \cref{eq:I3-app-mid} to obtain
\[
\begin{aligned}
	\mathbb I_3
	\leq  {}&	\alpha_k\left(f(x^*)-f(x_{k+1})+\dual{\lambda^*,A(x^*-x_{k+1})}\right)
	-\frac{\alpha_k\gamma_{k+1}}{2}
	\nm{v_{k+1}-x^*}^2	\\
	&\quad+f_1(x_k)-f_1(x_{k+1})-\alpha_k(f_2(v_{k+1})-f_2(x_{k+1}))	-	\frac{\gamma_{k}}2\nm{v_{k+1}-v_k}^2\\
	{}& \qquad+(1+\alpha_k)\left(f_1(x_{k+1})-f_1(u_k)\right)
	-\alpha_k\dual{\nabla f_1(u_k),v_{k+1}-v_k}\\
	{}&\quad \qquad+\alpha_k\dual{\barlk-\lambda^*,A(x^*-x_{k+1})}+\dual{\barlk,A(x_k-x_{k+1})}	.
\end{aligned}
\]
This establishes \cref{eq:I3-2bapd-ex-x-im-y}.

\section{An Auxiliary  Differential Inequality}
\label{app:yt}
Denote by $W^{1,\infty}(0,\infty)$ the usual Sobolev space \cite{brezis_functional_2011} consisting of all real-valued functions, which, together with their generalized derivatives, belong to $L^\infty(0,\infty)$. 
Assume $y\in W^{1,\infty}(0,\infty)$ is positive and satisfies the differential inequality
\begin{equation}\label{eq:yt}
	y'(t)\leq -\frac{\sigma(t)y^{\nu}(t)}{\sqrt{Py^2(t)+Qy(t)+R^2}},\quad y(0) = 1,
\end{equation}
where $\nu,\, P,\,Q,\,R\geq 0$ are constants and $\sigma\in L^1(0,\infty)$ is nonnegative. We shall establish sharp decay estimates of  $y(t)$ under two cases that are particularly interested in this paper, and the corresponding discrete versions will be presented later in \cref{app:yk}.
\subsection{Case I}
\label{app:yt-case1}
Let us first consider: $\nu\geq 3/2,\,P = 0$ and $Q>0$. For this case, we cite the result from \cite[Lemma 5.1]{luo_universal_2022}.
\begin{lem}[\cite{luo_universal_2022}]
	\label{lem:est-yt-case1}
	Let $y\in W^{1,\infty}(0,\infty)$ be positive and satisfy \cref{eq:yt} with $\nu\geq 3/2,\,P = 0$ and $Q>0$. Then for all $t> 0$, we have
	\[
	y(t)\leq C_\nu\left\{
	\begin{aligned}
		{}&\left(\frac{\sqrt{Q}}{\Sigma(t)}\right)^{\frac{2}{2\nu-3}}+
		\left(\frac{R	}{\Sigma(t)}\right)^{\frac{1}{\nu-1}}			&&\text{if }\nu>3/2,		\\		
		{}&\exp\left(-\frac{\Sigma(t)}{2\sqrt{Q}}\right)+
		\left(		\frac{R}{\Sigma(t)}\right)^2&&\text{if }\nu=3/2,
	\end{aligned}
	\right.
	\]
	where $\Sigma(t) :=\int_{0}^{t}\sigma(s)\dd s$ and $C_\nu>0$ depends only on $\nu$.
\end{lem}
\subsection{Case II}
\label{app:yt-case2}
We then move to another case: $\nu = 2$ and $P>0$.
\begin{lem}\label{lem:est-yt-case2}
	Let $y\in W^{1,\infty}(0,\infty)$ be positive and satisfy \cref{eq:yt} with $\nu=2$ and $P > 0$. Then for all $t> 0$, we have
	\begin{equation}	\label{eq:est-yt-case2}
		y(t)\leq 	
		\exp\left(-\frac{\Sigma(t)}{2\sqrt{P}}\right)
		+\frac{36Q}{\Sigma^2(t)}
		+\frac{6R}{\Sigma(t)},
	\end{equation}
	where $\Sigma(t)=\int_{0}^{t}\sigma(s)\dd s$.
\end{lem}
\begin{proof}
	From \cref{eq:yt} we obtain
	\[
	y'(t)\leq -\frac{\sigma(t)y^{2}(t)}{\sqrt{P}y(t)+ \sqrt{Q y(t)}+R},
	\]
	which implies
	\[
	\left[\sqrt{P}y^{-1}(t) + \sqrt{Q}y^{-3/2}(t)+Ry^{-2}(t)\right]y'(t)\leq-\sigma(t).
	\]
	Since $y(0) = 1$, integrating over $(0,t)$ gives 
	\begin{equation}\label{eq:y-Sigma}
		\sqrt{P}\ln \frac{1}{y(t) }+
		2\sqrt{Q}\big(y^{-1/2}(t)-1\big)+R\left(y^{-1}(t)-1\right)\geq \int_{0}^{t}\sigma(s)\dd s=
		\Sigma(t).
	\end{equation}
	
	Define $G:(0,\infty)\to[0,\infty)$ as follows
	\[
	G(w): = \sqrt{P}\ln \frac{1}{w}+ 2\sqrt{Q}\big(w^{-1/2}-1\big)+R\left(w^{-1}-1\right)
	\quad \forall\, w>0.
	\]
	Besides, let 
	\[
	Y_1(t) = {}
	\exp\left(-\frac{\Sigma(t)}{2\sqrt{P}}\right)
	,\quad
	Y_2(t) = {}\frac{Q}{\left(\sqrt{Q}+\frac{1}{6}\Sigma(t)\right)^2},\quad\text{and}\quad
	Y_3(t) = {}\frac{R}{R+\frac{1}{6}\Sigma(t)}.
	\]
	One observes 
	\[
	\sqrt{P}\ln \frac{1}{Y_1(t) }=3
	\sqrt{Q}\big(Y_2^{-1/2}(t)-1\big)=3R\left(Y_3^{-1}(t)-1\right)=\frac{1}{2}\Sigma(t),
	\]
	and it follows that
	\[
	\begin{aligned}
		G(Y(t))
		\leq{}& \sqrt{P}\ln \frac{1}{Y_1(t) }+
		2\sqrt{Q}\big(Y_2^{-1/2}(t)-1\big)+R\left(Y_3^{-1}(t)-1\right)
		=\Sigma(t),
	\end{aligned}
	\]
	where $Y(t)=Y_1(t)+Y_2(t)+Y_3(t)$. As \cref{eq:y-Sigma} implies $G(y(t))\geq\Sigma(t)$ and $G(\cdot)$ is monotone decreasing, we conclude that 
	\[
	y(t)\leq Y(t)\leq 	\exp\left(-\frac{\Sigma(t)}{2\sqrt{P}}\right)
	+\frac{36Q}{\Sigma^2(t)}
	+\frac{6R}{\Sigma(t)},
	\]
	which leads to \cref{eq:est-yt-case2} and completes the proof.
\end{proof}

\section{Decay Estimates of Some Difference Equations}
\label{app:yk}
\begin{lem}
	\label{lem:est-yk-case1}
	Let $\{\theta_k\}_{k=0}^\infty$ be a positive real sequence such that
	\begin{equation}\label{eq:tk-case1}
		\theta_{k+1}-\theta_k \leq -\sigma\theta^{\nu}_{k}\theta_{k+1},\quad \theta_0=1,
	\end{equation}
	where $\sigma,\,\nu>0$. If $\theta_{k+1}/\theta_{k}\geq \tau>0$ for all $k\in\mathbb N$, then
	\begin{equation}\label{eq:est-yk-case1}
		\theta_k\leq  \left(1+\sigma\tau\nu k\right)^{-1/\nu}\quad\forall\,k\in\mathbb N.
	\end{equation}
\end{lem}
\begin{proof}
	Define a piece-wise continuous linear function $y:[0,\infty)\to (0,\infty)$ by that
	\begin{equation}\label{eq:aux-yt}
		y(t): = \theta_{k}(k+1-t)+\theta_{k+1}(t-k),
		\quad\,t\in[k,k+1)\quad\forall\, k\in\mathbb N.
	\end{equation}
	Clearly, $y\in W^{1,\infty}(0,\infty)$ is decreasing and $y(0) = 1$. In addition, we have
	\begin{equation}\label{eq:tk-yt}
		\theta_{k+1}\leq y(t)\leq \theta_{k}\quad \text{and}\quad
		\frac{\theta_{k+1}}{y(t)}\geq \frac{\theta_{k+1}}{\theta_k}\geq \tau\quad\forall\,t\in[k,k+1].
	\end{equation}
	According to \cref{eq:tk-case1}, we obtain
	\[
	y'(t) \leq -\sigma\tau  y^{1+\nu}(t)\quad\Longrightarrow\quad 
	y(t)\leq  \left(1+\sigma\tau\nu t\right)^{-1/\nu}.
	\]
	Hence, \cref{eq:est-yk-case1} follows immediately from this estimate and the fact $\theta_k = y(k)$.
\end{proof}

We then apply \cref{lem:est-yt-case1,lem:est-yt-case2} to obtain the optimal decay rates of  two difference equations.
\begin{lem}
	\label{lem:est-yk-case2}
	Let $\{\theta_k\}_{k=0}^\infty$ be a positive real sequence such that
	\begin{equation}\label{eq:tk-case2}
		\theta_{k+1}-\theta_k \leq  -\frac{\sigma\theta^{\nu}_{k}\theta_{k+1}}{\sqrt{Q\theta_{k}+R^2}},\quad \theta_0=1,
	\end{equation}
	where $\sigma,Q>0,\,R\geq 0 $ and $\nu\geq 1/2$. If  $\theta_{k+1}/\theta_{k}\geq \tau>0$ for $k\in\mathbb N$, then we have
	\begin{equation}\label{eq:est-yk-case2}
		\theta_k
		\leq C_\nu\left\{
		\begin{aligned}
			{}&	\left(\frac{\sqrt{Q}}{\sigma\tau k}\right)^{\frac{2}{2\nu-1}}+
			\left(\frac{R}{\sigma\tau k}\right)^{\frac{1}{\nu}}			,&&\text{if }\nu>1/2,\\		
			{}&\exp\left(-\frac{\sigma\tau k}{2\sqrt{Q}} \right)+
			\left(		\frac{R}{\sigma\tau k}\right)^2,&&\text{if }\nu=1/2,
		\end{aligned}
		\right.
	\end{equation}
	for $k\geq 1$, where $C_\nu>0$ depends only on $\nu$.
\end{lem}
\begin{proof}
	Again, we use the piece-wise continuous linear interpolation $y(t)$ defined by \cref{eq:aux-yt}. In view of \cref{eq:tk-yt,eq:tk-case2}, we find
	\begin{equation}\label{eq:aux-y'}
		y'(t)\leq -\frac{\sigma\tau y^{1+\nu}(t)}{\sqrt{Q y(t)+R^2}},
	\end{equation}
	and invoking  \cref{lem:est-yt-case1} proves \cref{eq:est-yk-case2}.
\end{proof}

\begin{lem}
	\label{lem:est-yk-case3}
	Let $\{\theta_k\}_{k=0}^\infty$ be a positive real sequence such that
	\[
	\theta_{k+1}-\theta_k \leq  -\frac{\sigma\theta_{k}\theta_{k+1}}{\sqrt{P\theta_k^2+Q\theta_{k}+R^2}},\quad \theta_0=1,
	\]
	where $\sigma,P>0$ and $Q,R\geq 0 $. If $\theta_{k+1}/\theta_{k}\geq \tau>0$ for all $k\in\mathbb N$, then we have 
	\begin{equation}\label{eq:est-yk-case3}
		\theta_k\leq 
		\exp\left(-\frac{\sigma \tau k}{2\sqrt{P}} \right)
		+\frac{36 Q}{\sigma^2\tau^2k^2}
		+\frac{6R}{\sigma\tau k}\quad\forall\,k\geq 1.
	\end{equation}
\end{lem}
\begin{proof}
	Similarly with \cref{eq:aux-y'}, it is not hard to get
	\[
	y'(t)\leq -\frac{\sigma\tau y^{2}(t)}{\sqrt{P y^2(t)+ Q y(t)+R^2}}.
	\]
	Applying \cref{lem:est-yt-case2} gives \cref{eq:est-yk-case3} and completes the proof.	
\end{proof}


\bibliographystyle{abbrv}


\begin{thebibliography}{10}
	
		\bibitem{attouch_fast_2021}
	H.~Attouch, Z.~Chbani, J.~Fadili, and H.~Riahi.
	\newblock Fast convergence of dynamical {ADMM} via time scaling of damped
	inertial dynamics.
	\newblock {\em J. Optim. Theory Appl.},
	https://doi.org/10.1007/s10957-021-01859-2, 2021.
	
	\bibitem{attouch_fast_2018}
	H.~Attouch, Z.~Chbani, J.~Peypouquet, and P.~Redont.
	\newblock Fast convergence of inertial dynamics and algorithms with asymptotic
	vanishing viscosity.
	\newblock {\em Math. Program. Series B}, 168(1-2):123--175, 2018.
	
	\bibitem{attouch_rate_2019}
	H.~Attouch, Z.~Chbani, and H.~Riahi.
	\newblock Rate of convergence of the {Nesterov} accelerated gradient method in
	the subcritical case $b\leqslant 3$.
	\newblock {\em ESAIM Control Optim. Calc. Var.}, 25(2), 2019.

	\bibitem{bai_generalized_2018}
	J.~Bai, J.~Li, F.~Xu, and H.~Zhang.
	\newblock Generalized symmetric {ADMM} for separable convex optimization.
	\newblock {\em Comput. Optim. Appl.}, 70(1):129--170, 2018.
	
		
	\bibitem{bitterlich_dynamical_2021}
	S.~Bitterlich, E.~R. Csetnek, and G.~Wanka.
	\newblock A dynamical approach to two-block separable convex optimization
	problems with linear constraints.
	\newblock {\em Numerical Functional Analysis and Optimization}, 42(1):1--38,
	2021.
	
	\bibitem{bot_primal-dual_2020}
	R.~Bo\c{t}, E.~Csetnek, and S.~L\'{a}szl\'{o}.
	\newblock A primal-dual dynamical approach to structured convex minimization
	problems.
	\newblock {\em J. Diff. Equ.}, 269:10717--10757, 2020.
	
	\bibitem{bot_fast_2022}
	R.~I. Boţ, E.~R. Csetnek, and D.-K. Nguyen.
	\newblock Fast augmented {Lagrangian} method in the convex regime with
	convergence guarantees for the iterates.
	\newblock {\em Math. Program.},
	https://link.springer.com/10.1007/s10107-022-01879-4, 2022.
	
	
		\bibitem{brezis_functional_2011}
	H. Brezis.
	\newblock {\em Functional Analysis, Sobolev Spaces and Partial Differential Equations}.
	\newblock Universitext. Springer, New York, 201.

	
	
	\bibitem{chambolle_first-order_2011}
	A.~Chambolle and T.~Pock.
	\newblock A first-order primal-dual algorithm for convex problems with
	applications to imaging.
	\newblock {\em J.Math. Imaging Vis.}, 40(1):120--145, 2011.
	
	\bibitem{chambolle_ergodic_2016}
	A.~Chambolle and T.~Pock.
	\newblock On the ergodic convergence rates of a first-order primal-dual
	algorithm.
	\newblock {\em Math. Program.}, 159(1-2):253--287, 2016.
	
	\bibitem{chen_direct_2016}
	C.~Chen, B.~He, Y.~Ye, and X.~Yuan.
	\newblock The direct extension of {ADMM} for multi-block convex minimization
	problems is not necessarily convergent.
	\newblock {\em Math. Program.}, 155:57--79, 2016.
	
		\bibitem{chen_first_2019}
	L.~Chen and H.~Luo.
	\newblock First order optimization methods based on {Hessian}-driven {Nesterov}
	accelerated gradient flow.
	\newblock {\em arXiv:1912.09276}, 2019.
	
	\bibitem{chen_luo_unified_2021}
	L.~Chen and H.~Luo.
	\newblock A unified convergence analysis of first order convex optimization
	methods via strong {L}yapunov functions.
	\newblock {\em arXiv: 2108.00132}, 2021.
	
	\bibitem{chen_transformed_2023}
	L.~Chen and J.~Wei.
	\newblock Transformed primal-dual methods for nonlinear saddle point systems.
	\newblock {\em J. Numer. Math.}, 2023.
	
	\bibitem{chen_revisiting_2022}
	S.~Chen, B.~Shi, and Y.-X. Yuan.
	\newblock Revisiting the acceleration phenomenon via high-resolution
	differential equations.
	\newblock {\em arXiv:2212.05700}, 2022.
	

	\bibitem{chen_proximal-based_1994}
	G.~Chen and M.~Teboulle.
	\newblock A proximal-based decomposition method for convex minimization
	problems.
	\newblock {\em Math. Program.}, 64(1):81--101, 1994.
	
	\bibitem{chen_optimal_2014}
	Y.~Chen, G.~Lan, and Y.~Ouyang.
	\newblock Optimal primal-dual methods for a class of saddle point problems.
	\newblock {\em SIAM J. Optim.}, 24(4):1779--1814, 2014.
	
		\bibitem{cherukuri_saddle-point_2017}
	A.~Cherukuri, B.~Gharesifard, and J.~Cort\'{e}s.
	\newblock Saddle-point dynamics: conditions for asymptotic stability of saddle
	points.
	\newblock {\em SIAM J. Control Optim.}, 55(1):486--511, 2017.
	
	
	\bibitem{dai_sequential_2016}
	Y.-H. Dai, D.~Han, X.~Yuan, and W.~Zhang.
	\newblock A sequential updating scheme of the {Lagrange} multiplier for
	separable convex programming.
	\newblock {\em Math. Comp.}, 86(303):315--343, 2016.
	
	\bibitem{davis_convergence_2016}
	D.~Davis and W.~Yin.
	\newblock Convergence rate analysis of several splitting schemes.
	\newblock {\em Splitting Methods in Communication, Imaging, Science, and
		Engineering}, pages 115--163, 2016.
	
	\bibitem{douglas_numerical_1956}
	J.~Douglas and H.~H. Rachford.
	\newblock On the numerical solution of heat conduction problems in two and
	three space variables.
	\newblock {\em Trans. Amer. Math. Soc.}, 82:421--439, 1956.
	
	\bibitem{eckstein_augmented_2012}
	J.~Eckstein.
	\newblock Augmented {Lagrangian} and alternating direction methods for convex
	optimization: a tutorial and some illustrative computational results.
	\newblock Technical report, Rutgers University, 2012.
	
	\bibitem{eckstein_douglas-rachford_1992}
	J.~Eckstein and D.~P. Bertsekas.
	\newblock On the {Douglas}--{Rachford} splitting method and the proximal point
	algorithm for maximal monotone operators.
	\newblock {\em Math. Program.}, 55(1):293--318, 1992.
	
	\bibitem{esser_general_2010}
	E.~Esser, X.~Zhang, and T.~F. Chan.
	\newblock A general framework for a class of first order primal-dual algorithms
	for convex optimization in imaging science.
	\newblock {\em SIAM J. Imaging Sci.}, 3(4):1015--1046, 2010.
	
	\bibitem{facchinei_finite-dimensional_2003_vol2}
	F.~Facchinei and J.-S. Pang.
	\newblock {\em Finite-{Dimensional} {Variational} {Inequalities} and
		{Complementarity} {Problems}, {Vol} {II}}.
	\newblock Springer series in operations research. Springer, New York, 2003.
	
	
	\bibitem{fazel_hankel_2013}
	M.~Fazel, T.~Pong, D.~Sun, and P.~Tseng.
	\newblock Hankel matrix rank minimization with applications to system
	identification and realization.
	\newblock {\em SIAM J. Matrix Anal. \& Appl.}, 34(3):946--977, 2013.
	
		
	\bibitem{franca_nonsmooth_2021}
	G.~Fran\c{c}a, D.~P. Robinson, and R.~Vidal.
	\newblock A nonsmooth dynamical systems perspective on accelerated extensions
	of {ADMM}.
	\newblock {\em arXiv:1808.04048}, 2021.
	
	
	
	\bibitem{gabay_chapter_1983}
	D.~Gabay.
	\newblock Chapter {I}{X} {A}pplications of the {M}ethod of {M}ultipliers to
	{V}ariational {I}nequalities.
	\newblock In M.~Fortin and R.~Glowinski, editors, {\em Augmented {Lagrangian}
		{Methods}: {Applications} to the {Numerical} {Solution} of
		{Boundary}--{Value} {Problems}}, volume~15 of {\em Studies in Mathematics and
		Its Applications}, pages 299--331. Elsevier, 1983.
	
	\bibitem{gabay_dual_1976}
	D.~Gabay and B.~Mercier.
	\newblock A dual algorithm for the solution of nonlinear variational problems
	via finite element approximation.
	\newblock {\em Comput. Math. Appl.}, 2(1):17--40, 1976.
	
	\bibitem{goldfarb_fast_2013}
	S.~S.~K. Goldfarb, D.and~Ma.
	\newblock Fast alternating linearization methods for minimizing the sum of two
	convex functions.
	\newblock {\em Math. Program.}, 141(1-2):349--382, 2013.
	
	\bibitem{goldstein_fast_2014}
	T.~Goldstein, B.~O'Donoghue, S.~Setzer, and R.~Baraniuk.
	\newblock Fast alternating direction optimization methods.
	\newblock {\em SIAM J. Imaging Sci.}, 7(3):1588--1623, 2014.
	
	\bibitem{han_note_2012}
	D.~Han and X.~Yuan.
	\newblock A note on the alternating direction method of multipliers.
	\newblock {\em J. Optim. Theory Appl.}, 155(1):227--238, 2012.
	
	\bibitem{han_augmented_2014}
	D.~Han, X.~Yuan, and W.~Zhang.
	\newblock An augmented {Lagrangian} based parallel splitting method for
	separable convex minimization with applications to image processing.
	\newblock {\em Math. Comp.}, 83(289):2263--2291, 2014.
	
	
	\bibitem{he_full_2015}
	B.~He, L.~Hou, and X.~Yuan.
	\newblock On full {Jacobian} decomposition of the augmented {Lagrangian} method
	for separable convex programming.
	\newblock {\em SIAM J. Optim.}, 25(4):2274--2312, 2015.
	
	
	\bibitem{he_convergence_rate_2021}
	X.~He, R.~Hu, and Y.~P. Fang.
	\newblock Convergence rates of inertial primal-dual dynamical methods for
	separable convex optimization problems.
	\newblock {\em SIAM J. Control Optim.}, 59(5):3278--3301, 2021.
	
	\bibitem{he_perturbed_2021}
	X.~He, R.~Hu, and Y.-P. Fang.
	\newblock Perturbed primal-dual dynamics with damping and time scaling
	coefficients for affine constrained convex optimization problems.
	\newblock {\em arXiv:2106.13702}, 2021.
	
	\bibitem{HE2022110547}
	X.~He, R.~Hu, and Y.-P. Fang.
	\newblock Fast primal-dual algorithm via dynamical system for a linearly
	constrained convex optimization problem.
	\newblock {\em Automatica}, 146:110547, 2022.
	
	\bibitem{he_inertial_2022}
	X.~He, R.~Hu, and Y.-P. Fang.
	\newblock Inertial accelerated primal-dual methods for linear equality
	constrained convex optimization problems.
	\newblock {\em Numer. Algor.}, 90(4):1669--1690, 2022.
	
	\bibitem{he_strictly_2014}
	B.~He, H.~Liu, Z.~Wang, and X.~Yuan.
	\newblock A strictly contractive {Peaceman}--{Rachford} splitting method for
	convex programming.
	\newblock {\em SIAM J. Optim.}, 24(3):1011--1040, 2014.
	
	\bibitem{he_alternating_2012}
	B.~He, M.~Tao, and X.~Yuan.
	\newblock Alternating direction method with {Gaussian} back substitution for
	separable convex programming.
	\newblock {\em SIAM J. Optim.}, 22(2):313--340, 2012.
	
	\bibitem{he_splitting_2015}
	B.~He, M.~Tao, and X.~Yuan.
	\newblock A splitting method for separable convex programming.
	\newblock {\em IMA J. Numer. Anal.}, 35(1):394--426, 2015.
	
	
	\bibitem{he_convergence_2014}
	B.~He, Y.~You, and X.~Yuan.
	\newblock On the convergence of primal-dual hybrid gradient algorithm.
	\newblock {\em SIAM J. Imaging Sci.}, 7(4):2526--2537, 2014.
	
	\bibitem{He_Yuan2010}
	B.~He and X.~Yuan.
	\newblock On the acceleration of augmented {Lagrangian} method for linearly
	constrained optimization.
	\newblock {\em Optimization online}, pages
	https://optimization--online.org/2010/10/2760/, 2010.
	
	\bibitem{He_Yuan2012}
	B.~He and X.~Yuan.
	\newblock On the ${O}(1/n)$ convergence rate of the {D}ouglas--{R}achford
	alternating direction method.
	\newblock {\em SIAM J. Numer. Anal.}, 50(2):700--709, 2012.
	
	\bibitem{huang_accelerated_2013}
	B.~Huang, S.~Ma, and D.~Goldfarb.
	\newblock Accelerated linearized {B}regman method.
	\newblock {\em J. Sci. Comput.}, 54:428--453, 2013.
	
	
	\bibitem{jordan_dynamical_2019}
	M.~I. Jordan.
	\newblock Dynamical, symplectic and stochastic perspectives on gradient-based
	optimization.
	\newblock In {\em Proceedings of the {International} {Congress} of
		{Mathematicians} ({ICM} 2018)}, pages 523--549, Rio de Janeiro, Brazil, 2019.
	World Scientific.

	\bibitem{kang_inexact_2015}
	M.~Kang, M.~Kang, and M.~Jung.
	\newblock Inexact accelerated augmented {L}agrangian methods.
	\newblock {\em Comput. Optim. Appl.}, 62(2):373--404, 2015.
	
	\bibitem{kang_accelerated_2013}
	M.~Kang, S.~Yun, H.~Woo, and M.~Kang.
	\newblock Accelerated {B}regman method for linearly constrained
	$\ell_1$-$\ell_2$ minimization.
	\newblock {\em J. Sci. Comput.}, 56(3):515--534, 2013.
	
	\bibitem{li_convergence_2017}
	H.~Li, C.~Fang, and Z.~Lin.
	\newblock Convergence rates analysis of the quadratic penalty method and its
	applications to decentralized distributed optimization.
	\newblock {\em arXiv:1711.10802}, 2017.
	
	\bibitem{Li2019}
	H.~Li and Z.~Lin.
	\newblock Accelerated alternating direction method of multipliers: {A}n optimal
	${O}(1/{K})$ nonergodic analysis.
	\newblock {\em J. Sci. Comput.}, 79(2):671--699, 2019.
	
	\bibitem{li_asymptotically_2020}
	X.~Li and K.-C. Sun, D.and~Toh.
	\newblock An asymptotically superlinearly convergent semismooth {N}ewton
	augmented {L}agrangian method for {L}inear {P}rogramming.
	\newblock {\em SIAM J. Optim.}, 30(3):2410--2440, 2020.
	
	\bibitem{li_proximal_2015}
	X.~Li and X.~Yuan.
	\newblock A proximal strictly contractive {Peaceman}--{Rachford} splitting
	method for convex programming with applications to imaging.
	\newblock {\em SIAM J. Imaging Sci.}, 8(2):1332--1365, 2015.
	
		
	\bibitem{lin_control-theoretic_2021}
	T.~Lin and M.~I. Jordan.
	\newblock A control-theoretic perspective on optimal high-order optimization.
	\newblock {\em Math. Program. Series A}, pages
	https://doi.org/10.1007/s10107--021--01721--3, 2021.
	
	\bibitem{lu_osr-resolution_2021}
	H.~Lu.
	\newblock An ${O}(s^r)$-resolution {ODE} framework for understanding
	discrete-time algorithms and applications to the linear convergence of
	minimax problems.
	\newblock {\em arXiv:2001.08826}, 2021.
	
	\bibitem{luo_accelerated_2021}
	H.~Luo.
	\newblock Accelerated differential inclusion for convex optimization.
	\newblock {\em Optimization}, DOI: 10.1080/02331934.2021.2002327, 2021.
	
	\bibitem{luo_acc_primal-dual_2021}
	H.~Luo.
	\newblock Accelerated primal-dual methods for linearly constrained convex
	optimization problems.
	\newblock {\em arXiv:2109.12604}, 2021.
	
	\bibitem{luo_primal-dual_2022}
	H.~Luo.
	\newblock A primal-dual flow for affine constrained convex optimization.
	\newblock {\em ESAIM Control Optim. Calc. Var.}, 28:33, 2022.
	
	\bibitem{luo_universal_2022}
	H.~Luo.
	\newblock A universal accelerated primal-dual method for convex optimization
	problems.
	\newblock {\em arXiv:2211.04245}, 2022.
	
	\bibitem{luo_differential_2021}
	H.~Luo and L.~Chen.
	\newblock From differential equation solvers to accelerated first-order methods
	for convex optimization.
	\newblock {\em Math. Program.}, 195:735--781, 2022.
	
	
	
	\bibitem{monteiro_iteration-complexity_2013}
	R.~Monteiro and B.~Svaiter.
	\newblock Iteration-complexity of block-decomposition algorithms and the
	alternating direction method of multipliers.
	\newblock {\em SIAM J. Optim.}, 23(1):475--507, 2013.
	
	\bibitem{necoara_application_2008}
	I.~Necoara and J.~A.~K. Suykens.
	\newblock Application of a smoothing technique to decomposition in convex
	optimization.
	\newblock {\em IEEE Trans. Autom. Control}, 53(11):2674--2679, 2008.
	
	\bibitem{Nesterov1983}
	Y.~Nesterov.
	\newblock A method of solving a convex programming problem with convergence
	rate ${O}(1/k^2)$.
	\newblock {\em Soviet Mathematics Doklady}, 27(2):372--376, 1983.
	
	\bibitem{nesterov_smooth_2005}
	Y.~Nesterov.
	\newblock Smooth minimization of non-smooth functions.
	\newblock {\em Math. Program.}, 103(1):127--152, 2005.
	
	\bibitem{niu_sparse_2021}
	D.~Niu, C.~Wang, P.~Tang, Q.~Wang, and E.~Song.
	\newblock A sparse semismooth {Newton} based augmented {Lagrangian} method for
	large-scale support vector machines.
	\newblock {\em arXiv:1910.01312}, 2021.
	
	\bibitem{ouyang_accelerated_2015}
	Y.~Ouyang, Y.~Chen, G.~Lan, and E.~Pasiliao.
	\newblock An accelerated linearized alternating direction method of
	multipliers.
	\newblock {\em SIAM J. Imaging Sci.}, 8(1):644--681, 2015.
	
	\bibitem{ouyang_lower_2021}
	Y.~Ouyang and Y.~Xu.
	\newblock Lower complexity bounds of first-order methods for convex-concave
	bilinear saddle-point problems.
	\newblock {\em Math. Program.}, 185(1-2):1--35, 2021.
	
	\bibitem{peaceman_numerical_1955}
	D.~W. Peaceman and H.~H. Rachford.
	\newblock The numerical solution of parabolic and elliptic differential
	equations.
	\newblock {\em J. Soc. Indust. Appl. Math.}, 3(1):28--41, 1955.
	
	\bibitem{rockafellar_augmented_1976}
	R.~T. Rockafellar.
	\newblock Augmented {Lagrangians} and applications of the proximal point
	algorithm in convex programming.
	\newblock {\em Mathematics of OR}, 1(2):97--116, 1976.
	
	\bibitem{sabach_faster_2022}
	S.~Sabach and M.~Teboulle.
	\newblock Faster {Lagrangian}-based methods in convex optimization.
	\newblock {\em SIAM J. Optim.}, 32(1):204--227, 2022.
	
	
	
	
	
	
	\bibitem{schropp_dynamical_2000}
	J.~Schropp and I.~Singer.
	\newblock A dynamical systems approach to constrained minimization.
	\newblock {\em Numerical Functional Analysis and Optimization},
	
	\bibitem{shefi_rate_2014}
	R.~Shefi and M.~Teboulle.
	\newblock Rate of convergence analysis of decomposition methods based on the
	proximal method of multipliers for convex minimization.
	\newblock {\em SIAM J. Optim.}, 24(1):269--297, 2014.
	
	
	
	21(3-4):537--551, 2000.
	
	\bibitem{shi_understanding_2021}
	B.~Shi, S.~S. Du, M.~I. Jordan, and W.~J. Su.
	\newblock Understanding the acceleration phenomenon via high-resolution
	differential equations.
	\newblock {\em Math. Program.}, 195:79--148, 2022.
	
	\bibitem{su_dierential_2016}
	W.~Su, S.~Boyd, and E.~J. Cand\`{e}s.
	\newblock A differential equation for modeling {Nesterov}'s accelerated
	gradient method: theory and insights.
	\newblock {\em J. Mach. Learn. Res.}, 17:1--43, 2016.
	
	
	\bibitem{tao_accelerated_2016}
	M.~Tao and X.~Yuan.
	\newblock Accelerated {Uzawa} methods for convex optimization.
	\newblock {\em Math. Comp.}, 86(306):1821--1845, 2016.
	
	\bibitem{tian_alternating_2018}
	W.~Tian and X.~Yuan.
	\newblock An alternating direction method of multipliers with a worst-case
	${O}(1/n^2)$ convergence rate.
	\newblock {\em Math. Comp.}, 88(318):1685--1713, 2018.
	
	\bibitem{tran-dinh_proximal_2019}
	Q.~Tran-Dinh.
	\newblock Proximal alternating penalty algorithms for nonsmooth constrained
	convex optimization.
	\newblock {\em Comput. Optim. Appl.}, 72(1):1--43, 2019.
	
	\bibitem{tran-dinh_primal-dual_2015}
	Q.~Tran-Dinh and V.~Cevher.
	\newblock A primal-dual algorithmic framework for constrained convex
	minimization.
	\newblock {\em arXiv:1406.5403}, 2015.
	
	\bibitem{tran-dinh_smooth_2018}
	Q.~Tran-Dinh, O.~Fercoq, and V.~Cevher.
	\newblock A smooth primal-dual optimization framework for nonsmooth composite
	convex minimization.
	\newblock {\em SIAM J. Optim.}, 28(1):96--134, 2018.
	
	\bibitem{trandinh_combining_2013}
	Q.~Tran-Dinh, C.~Savorgnan, and M.~Diehl.
	\newblock Combining {Lagrangian} decomposition and excessive gap smoothing
	technique for solving large-scale separable convex optimization problems.
	\newblock {\em Comput. Optim. Appl.}, 55(1):75--111, 2013.
	
	\bibitem{tran-dinh_augmented_2018}
	Q.~Tran-Dinh and Y.~Zhu.
	\newblock Augmented {Lagrangian}-based decomposition methods with non-ergodic
	optimal rates.
	\newblock {\em arXiv:1806.05280}, 2018.
	
	\bibitem{tran-dinh_non-stationary_2020}
	Q.~Tran-Dinh and Y.~Zhu.
	\newblock Non-stationary first-order primal-dual algorithms with faster
	convergence rates.
	\newblock {\em SIAM J. Optim.}, 30(4):2866--2896, 2020.
	
	\bibitem{tseng_on_accelerated_Seattle_2008}
	P.~Tseng.
	\newblock On accelerated proximal gradient methods for convex-concave
	optimization.
	\newblock Technical report, University of Washington, Seattle, 2008.
	
		\bibitem{wang_search_2021}
	Y.~Wang, Z.~Jia, and Z.~Wen.
	\newblock Search direction correction with normalized gradient makes
	first-order methods faster.
	\newblock {\em SIAM J. Sci. Comput.}, 43(5):A3184--A3211, 2021.
	
	\bibitem{wibisono_variational_2016}
	A.~Wibisono, A.~C. Wilson, and M.~Jordan.
	\newblock A variational perspective on accelerated methods in optimization.
	\newblock {\em Proc. Natl. Acad. Sci. USA}, 113(47):E7351--E7358, 2016.
	
	\bibitem{wilson_lyapunov_2021}
	A.~C. Wilson, B.~Recht, and M.~I. Jordan.
	\newblock A {Lyapunov} analysis of accelerated methods in optimization.
	\newblock {\em J. Mach. Learn. Res.}, 22:1--34, 2021.
	
	\bibitem{woodworth_tight_2016}
	B.~E. Woodworth and N.~Srebro.
	\newblock Tight complexity bounds for optimizing composite objectives.
	\newblock In {\em Advances in {Neural} {Information} {Processing} {Systems}
		({NIPS})}, pages 3639--3647, Barcelona, Spain, 2016.
	
	\bibitem{Xu2017}
	Y.~Xu.
	\newblock Accelerated first-order primal-dual proximal methods for linearly
	constrained composite convex programming.
	\newblock {\em SIAM J. Optim.}, 27(3):1459--1484, 2017.
	
	\bibitem{xu_iteration_2021}
	Y.~Xu.
	\newblock Iteration complexity of inexact augmented {Lagrangian} methods for
	constrained convex programming.
	\newblock {\em Math. Program.}, 185(1-2):199--244, 2021.
	
		
	
	\bibitem{zeng_dynamical_2019}
	X.~Zeng, J.~Lei, and J.~Chen.
	\newblock Dynamical primal-dual accelerated method with applications to network
	optimization.
	\newblock {\em IEEE Trans. Automat. Contr.}, page DOI 10.1109/TAC.2022.3152720,
	2022.

	
	\bibitem{zhang_lower_2022}
	J.~Zhang, M.~Hong, and S.~Zhang.
	\newblock On lower iteration complexity bounds for the convex concave saddle
	point problems.
	\newblock {\em Math. Program.}, 194(1-2):901--935, 2022.
	
	\bibitem{zhang_faster_2022}
	T.~Zhang, Y.~Xia, and S.~R. Li.
	\newblock Faster {Lagrangian}-based methods: a unified prediction-correction
	framework.
	\newblock {\em arXiv:2206.05088}, 2022.
	
		
	\bibitem{zhao_accelerated_2022}
	Y.~Zhao, X.~Liao, X.~He, and C.~Li.
	\newblock Accelerated primal-dual mirror dynamical approaches for constrained
	convex optimization.
	\newblock {\em arXiv:2205.15983}, 2022.
	
	\bibitem{zhu_unified_2022}
	Z.~Zhu, F.~Chen, J.~Zhang, and Z.~Wen.
	\newblock A unified primal-dual algorithm framework for inequality constrained
	problems.
	\newblock {\em arXiv:2208.14196}, 2022.
			
	
	
\end{thebibliography}

\end{document}